\documentclass[oneside, reqno, 10pt, a4paper]{amsart}

\title[A tensor-train reduced basis solver for parameterized partial differential equations on Cartesian grids]{A tensor-train reduced basis solver for parameterized partial differential equations on Cartesian grids}
\date{\today}
\address{$^\dagger$School of mathematics\\Monash university\\Clayton\\Victoria 3800\\Australia}
\address{$^\P $School of Mathematics and Statistics \\University of Sydney\\New South Wales 2006\\Australia}
\author[N. Mueller]{Nicholas Mueller$^\dagger$}
\email{nicholas.mueller@monash.edu}
\author[Y. Zhao]{Yiran Zhao$^\P$}
\email{yiran.zhao@monash.edu}
\author[S. Badia]{Santiago Badia$^{\dagger}$}
\email{santiago.badia@monash.edu}
\author[T. Cui]{Tiangang Cui$^\P$}
\email{tiangang.cui@monash.edu}

\usepackage{dutchcal}
\usepackage{microtype,times}
\AtBeginDocument{
  \DeclareSymbolFont{AMSb}{U}{msb}{m}{n}
  \DeclareSymbolFontAlphabet{\mathbb}{AMSb}
}
\DeclareFontFamily{U}{mathx}{\hyphenchar\font45}
\DeclareFontShape{U}{mathx}{m}{n}{<-> mathx10}{}
\DeclareSymbolFont{mathx}{U}{mathx}{m}{n}
\DeclareMathAccent{\widebar}{0}{mathx}{"73}

\usepackage[dvipsnames]{xcolor}
\usepackage{graphicx}
\usepackage{subfig}
\usepackage{svg}
\usepackage{tikz}
\tikzstyle{arrow} = [thick,->,>=stealth]
\usetikzlibrary{shapes.misc,matrix,fit,positioning,arrows.meta,decorations.pathreplacing,calc,shapes.geometric,arrows}
\tikzset{Matrix/.style={matrix of nodes, font=\footnotesize,text height=1pt, text depth=0.5pt, text width=8.5pt, align=center, column sep=0pt, row sep=0pt, nodes in empty cells}}

\usepackage[a4paper, pdftex, left=2cm, top=2cm, right=2cm, bottom=2cm]{geometry}
\usepackage[final]{pdfpages}
\usepackage{changepage}

\usepackage[foot]{amsaddr}

\usepackage{url}
\usepackage{hyperref}
\hypersetup{
    breaklinks=true,
    bookmarksopen=true,
    pdftitle={A tensor-train reduced basis solver for parameterized partial differential equations on Cartesian grids}, 
    pdfauthor={Nicholas Mueller, Yiran Zhao, Santiago Badia, Tiangang Cui}, 
    pdfsubject={Model order reduction, reduced basis method, tensor train decompositions}, 
    pdfkeywords={complexity reduction, algorithmic efficiency, parametric dependence}, 
    colorlinks=true,
    linkcolor=black,
    citecolor=blue,
    filecolor=black,
    urlcolor=blue
}

\usepackage{csquotes}
\usepackage[backend=biber, defernumbers=false, maxbibnames=5, style=numeric-comp, sorting=none, isbn=false, bibencoding=utf8, safeinputenc, url=false, doi=true, giveninits=true]{biblatex}

\PassOptionsToPackage{normalem}{ulem}
\usepackage{ulem}
\providecolor{added}{rgb}{0,0,1}
\providecolor{deleted}{rgb}{1,0,0}

\usepackage{float}
\usepackage{framed}
\usepackage{verbatim}
\usepackage{fancyvrb}
\usepackage{booktabs}
\usepackage{colortbl}
\usepackage{multirow}
\usepackage{algorithm}
\usepackage{algpseudocode}
\makeatletter
\algdef{S}[IF]{IfNoThen}[1]{\algorithmicif\ #1}
\makeatother
\usepackage[inline]{enumitem}
\usepackage{siunitx}
\usepackage{mathtools}
\mathtoolsset{showonlyrefs}

\newtheorem{remark}{Remark}

\newtheorem{theorem}{Theorem}

\usepackage[breakable]{tcolorbox}
\tcbuselibrary{theorems}
\tcbuselibrary{skins}
\tcbset{
	commonstyle/.style={
		theorem style=plain,
		enhanced jigsaw,
		fonttitle=\bfseries,
		fontupper=\itshape,
		halign=justify,
		separator sign=:,
		description delimiters none,
		description font=\bfseries, 
		terminator sign={.\hspace{0.25em}},
		arc=0mm,outer arc=0mm,
		boxrule=0pt,toprule=0pt,bottomrule=0pt,leftrule=0pt,rightrule=0pt,
		titlerule=0pt,toptitle=0pt,bottomtitle=0pt,top=0pt,
		colback=white,coltitle=black,
		boxsep=0pt, bottom=0pt, left=0pt, 
	}
}
\newtcbtheorem[]{myproblem}{Problem}%
{center, commonstyle, fonttitle=\bfseries}{pb}
\newtcbtheorem[]{mydefinition}{Definition}
{center, commonstyle, fonttitle=\bfseries}{pb}
\newtcbtheorem[]{myassumption}{Assumption}
{center, commonstyle, fonttitle=\bfseries}{pb}

\definecolor{shadecolor}{gray}{.92}
\definecolor{incolor}{rgb}{0,0,.7}
\definecolor{outcolor}{rgb}{.65,0,0}
\definecolor{syntaxcolor}{rgb}{.65,0,0}
\definecolor{bg}{rgb}{0.93,0.93,0.93}
\definecolor{myblue}{RGB}{93,188,210}
\definecolor{mygreen}{RGB}{189,210,93}
\definecolor{myorange}{RGB}{210,173,93}
\definecolor{myred}{RGB}{210,93,130}
\definecolor{mydarkblue}{RGB}{93,130,210}
\definecolor{mydarkgreen}{RGB}{93,210,173}

\usepackage{amsmath, mathabx, amsfonts, amssymb, amscd, bm, mathtools}

\newcommand{\N}{\mathbb{N}}
\newcommand{\R}{\mathbb{R}}

\newcommand{\norm}[1]{\left\lVert#1\right\rVert}

\DeclareMathOperator*{\argmax}{\arg\max}
\newcommand{\zero}[1]{\bm{0}_{#1}}

\usepackage{acronym}
\acrodef{pde}[PDE]{partial differential equation}
\acrodef{fe}[FE]{finite element}
\acrodef{fem}[FEM]{finite element method}
\acrodef{dof}[DOF]{degree of freedom}
\acrodefplural{dof}[DOFs]{degrees of freedom}
\acrodef{be}[BE]{Backward Euler}
\acrodef{hf}[HF]{high-fidelity}
\acrodef{fom}[FOM]{full-order model}
\acrodef{lhs}[LHS]{left-hand side}
\acrodef{rhs}[RHS]{right-hand side}
\acrodef{rom}[ROM]{reduced-order model}
\acrodef{rb}[RB]{reduced basis}
\acrodefplural{rb}[RBs]{reduced bases}
\acrodef{svd}[SVD]{singular value decomposition}
\acrodef{rsvd}[RSVD]{randomized SVD}
\acrodef{sthosvd}[ST-HOSVD]{sequentially truncated high-order singular value decomposition}
\acrodef{pod}[POD]{proper orthogonal decomposition}
\acrodef{tpod}[TPOD]{truncated proper orthogonal decomposition}
\acrodef{strb}[ST-RB]{space-time reduced basis}
\acrodef{eim}[EIM]{empirical interpolation method}
\acrodef{deim}[DEIM]{discrete empirical interpolation method}
\acrodef{mdeim}[MDEIM]{empirical interpolation method in matrix form}
\acrodef{stmdeimrb}[ST-MDEIM-RB]{space-time MDEIM-RB}
\acrodef{dl}[DL]{deep learning}
\acrodef{nn}[NN]{neural network}
\acrodefplural{nn}[NNs]{neural networks}
\acrodef{tt}[TT]{tensor train}
\acrodef{ttrb}[TT-RB]{tensor train reduced basis}
\acrodef{ttsvd}[TT-SVD]{tensor train SVD}
\acrodef{ttcross}[TT-CROSS]{tensor train cross}
\acrodef{ttmdeim}[TT-MDEIM]{tensor train MDEIM}
\acrodef{jit}[JIT]{just-in-time}
\acrodef{als}[ALS]{alternating linear scheme}

\graphicspath{{img/}}

\begin{document}

\begin{abstract}
    In this manuscript, we introduce the tensor-train reduced basis method, a novel projection-based reduced-order model designed for the efficient solution of parameterized partial differential equations. While reduced-order models are widely used for their computational efficiency compared to full-order models, they often involve significant offline computational costs. Our proposed approach mitigates this limitation by leveraging the tensor train format to efficiently represent high-dimensional finite element quantities. This method offers several advantages, including a reduced number of operations for constructing the reduced subspaces, a cost-effective hyper-reduction strategy for assembling the PDE residual and Jacobian, and a lower dimensionality of the projection subspaces for a given accuracy. We provide a posteriori error estimates to validate the accuracy of the method and evaluate its computational performance on benchmark problems, including the Poisson equation, heat equation, and transient linear elasticity in two- and three-dimensional domains. Although the current framework is restricted to problems defined on Cartesian grids, we anticipate that it can be extended to arbitrary shapes by integrating the tensor-train reduced basis method with unfitted finite element techniques.
\end{abstract}

\maketitle

\section{Introduction}
\label{sec: introduction}

Projection-based \acp{rom} are advanced numerical techniques designed to approximate parametric \ac{hf} models, which typically involve finely resolved spatio-temporal discretizations of \acp{pde}. These methods aim to capture the \ac{hf} parameter-to-solution manifold within a carefully chosen vector subspace. The process generally consists of a computationally intensive offline phase, during which the subspace is constructed and the (Petrov-)Galerkin projection of the \ac{hf} equations onto this subspace is performed. This phase often includes the hyper-reduction of nonaffinely parameterized \ac{hf} quantities, such as residuals and Jacobians. Once the offline phase is complete, an efficient online phase follows, enabling the rapid computation of accurate solutions for new parameter selections. 

Among the most widely recognized projection-based \acp{rom} is the \ac{rb} method \cite{quarteroni2015reduced,negri2015reduced,rozza2013reduced,dalsanto2019algebraic}. This approach constructs a reduced-dimensional subspace by extracting it from a dataset of \ac{hf} solutions, commonly referred to as snapshots. The reduced subspace is then used to minimize the \ac{fom} residual under a suitable norm. However, the standard \ac{rb} algorithm struggles to efficiently address time-dependent problems, prompting the development of novel space-time \acp{rom} that simultaneously reduce both the spatial and temporal complexity of the \ac{fom}. Among these, the \ac{strb} method, first introduced in \cite{choi2021space} for solving a linear 2D Boltzmann transport equation, stands out as a prominent example. This method constructs a projection space by taking the Kronecker product of spatial and temporal subspaces, which are derived from the \ac{hf} snapshots. 

The \ac{strb} approach has been extended in various works, such as \cite{doi:10.1137/22M1509114}, which addresses a more complex 3D Stokes equation, and \cite{MUELLER2024115767}, where multiple hyper-reduction strategies are proposed within the same space-time framework. In \ac{strb}, the \ac{hf} snapshots are reshaped into either a spatial or temporal matrix. The column space of the spatial matrix captures the spatial evolution of the \ac{hf} variable for fixed time and parameter values and is used to compute the spatial reduced subspace. Similarly, the columns of the temporal matrix represent the temporal evolution of the \ac{hf} variable for fixed spatial coordinates and parameter values and are used to compute the temporal reduced subspace. These subspaces are typically constructed using direct methods such as \ac{tpod}, although greedy algorithms \cite{quarteroni2015reduced,prudhomme:hal-00798326,Prudhomme} may also be employed. This procedure can be interpreted as a Tucker decomposition \cite{hackbusch2014numerical,hackbusch2012tensor,tucker1966some} of the snapshots, where the data is viewed as a tensor with axes corresponding to the spatial, temporal, and parametric dimensions. 

In this work, we introduce a novel \ac{rb} method, termed \ac{ttrb}, which leverages the recently developed \ac{tt} decomposition technique for tensors. Our approach identifies a joint spatio-temporal subspace by applying a \ac{tt} decomposition \cite{oseledets2010tt,oseledets2011tensor,gorodetsky2019continuous} to the snapshots, represented as a tensor, akin to the process used in Tucker decompositions. The \ac{tt} representation of a tensor is typically computed using one of two strategies. The first, known as \ac{ttsvd} \cite{oseledets2010tt,oseledets2011tensor}, involves the successive application of \ac{tpod} to various matrix unfoldings of the snapshots tensor. The second, referred to as \ac{ttcross} \cite{gorodetsky2019continuous,oseledets2010tt,bigoni2016spectral,cui2021deep,dolgov2018approximation}, constructs the \ac{tt} decomposition greedily through a pseudo-skeleton approximation \cite{dolgov2014alternating,goreinov1997pseudo}. While \ac{ttsvd} generally provides higher accuracy, it is computationally more expensive compared to \ac{ttcross}.

Recent works have explored the use of \ac{tt} decompositions within the \ac{rb} framework \cite{mamonov2024priorianalysistensorrom,MAMONOV2022115122}, exploiting the Cartesian structure of parameter spaces to achieve low-rank tensor approximations. While these approaches show promise, we argue that addressing the spatial complexity of solutions offers even greater potential for computational speedup. This is because the spatial number of \acp{dof} is often significantly larger than the temporal and parametric sizes. The core efficiency of our \ac{ttrb} method lies in the so-called ``split-axes'' representation of snapshots, which decomposes the spatial evolution of the snapshots along each Cartesian direction. This representation enables more efficient computation of operations involving spatial quantities. For problems defined on Cartesian geometries, the ``split-axes'' representation is straightforward to derive. For more general geometries, it could potentially be obtained using unfitted element discretizations \cite{BADIA2018533}, though we leave this avenue for future investigation. While this strategy could, in principle, be applied within a \ac{strb} framework, doing so would significantly increase the dimensionality of the reduced subspaces, thereby degrading the online performance of the algorithm. 

\ac{ttrb} decompositions offer significantly higher accuracy than \ac{strb} for a given reduced subspace dimension. This advantage enables the construction of reduced-order models on subspaces with very small dimensions while maintaining high accuracy. Throughout this work, we substantiate this claim through both theoretical analysis and numerical validation. While our primary focus is on addressing time-dependent problems, we present \ac{ttrb} as a robust alternative to \ac{strb}. Nonetheless, it is important to highlight that \ac{ttrb} is a versatile \ac{rom} framework, capable of efficiently approximating parameterized \acp{pde}, whether or not they involve time dependence. 

The goal of this work is to develop a projection-based \ac{rom} tailored for the efficient approximation of parametric, potentially transient \acp{pde}, which satisfy the reducibility criteria outlined in \cite{Unger_2019,quarteroni2015reduced}. Our proposed \ac{rom} exclusively exploits \ac{tt} decompositions of \ac{hf} quantities and performs all necessary operations directly within this decomposition framework. The main contributions of this paper are summarized as follows:
\begin{enumerate}
  \item We propose a novel \ac{ttsvd} algorithm for efficiently constructing a projection subspace characterized by a non-standard orthogonality condition, such as one induced by a given inner product norm. Unlike the standard \ac{ttsvd}, which produces a basis orthogonal in the Euclidean norm, our approach accommodates alternative orthogonality conditions. We provide a detailed comparison between the proposed algorithm and the basis construction via \ac{tpod}, highlighting differences in computational cost and accuracy.
  \item We introduce \ac{ttmdeim}, a hyper-reduction strategy within the \ac{tt} framework that simplifies the TT-cross-DEIM procedure presented in \cite{dektor2024collocationmethodsnonlineardifferential}. Unlike the standard \ac{mdeim}, this method empirically interpolates tensors directly in their \ac{tt} format. We employ \ac{ttmdeim} to achieve an affine decomposition of the residual and Jacobian for the problems under consideration, demonstrating that it achieves accuracy comparable to \ac{mdeim}.
  \item We derive a posteriori error estimates for the combination of the method indicated above, referred to as the \ac{ttrb} method. These estimates reveal that the accuracy of the procedure is closely tied to a user-defined tolerance, which governs both the precision of the \ac{tt} subspace and the \ac{ttmdeim} approximation.
\end{enumerate}

This article is organized as follows. We conclude this section by introducing the notation used throughout the paper and briefly reviewing the key properties of \ac{tt} decompositions. In Sect.~\ref{sec: rb}, we present the \ac{fom} defined by a parameterized transient \ac{pde} and outline the basic implementation of the \ac{strb} approach. Sect.~\ref{sec: tt-rb} introduces our novel \ac{ttrb} strategy, detailing the construction of the \ac{tt} reduced subspace, the \ac{tt} hyper-reduction technique, and the projection of the \ac{fom} onto this subspace. We also analyze the computational cost and derive a posteriori error estimates for these steps. In Sect.~\ref{sec: results}, we showcase the numerical results obtained by applying the \ac{ttrb} method to various test cases. Finally, in Sect.~\ref{sec: conclusions}, we summarize our findings and discuss potential extensions of this work.

\subsection{Notation}
\label{subs: notation}
 
The notation used in this paper is inspired by \cite{MUELLER2024115767}. We work with multidimensional arrays (tensors) that represent quantities dependent on space, time, and parameters. The subscripts $s$, $t$, and $\mu$ denote the spatial, temporal, and parametric axes, respectively. For Cartesian geometries, the spatial axis can be further decomposed into $d$ Cartesian directions, with subscripts $1, \dots, d$ referring to each direction. A superscript $\mu$ is used to indicate quantities that depend on an unspecified parameter value. The parameters considered in this work are $p$-dimensional vectors sampled from a given parameter space, with a generic parameter denoted as $\bm{\mu} \in \R^p$. 

In the two-dimensional case, we have the parameter-dependent vectors
\begin{equation*}
    \bm{U}^\mu_1 \in \R^{N_1}, \qquad 
    \bm{U}^\mu_2 \in \R^{N_2}, \qquad 
    \bm{U}^\mu_t \in \R^{N_t} 
\end{equation*}
belonging to the first coordinate space $ \bm{\mathcal U}^\mu_1$, the second coordinate space $ \bm{\mathcal U}^\mu_2 $ and the temporal space $ \bm{\mathcal U}^\mu_t $, respectively. Their tensor product 
\begin{equation*}
    \bm{U}^{\mu}_{1,2,t} = \bm{U}^\mu_1 \otimes \bm{U}^\mu_2 \otimes \bm{U}^\mu_t
\end{equation*}
is an element in the parameter-dependent tensor product space 
\begin{equation*}
    \bm{\mathcal{U}}^{\mu}_{1,2,t} = \bm{\mathcal{U}}^{\mu}_1 \otimes \bm{\mathcal{U}}^{\mu}_2 \otimes \bm{\mathcal{U}}^{\mu}_t.
\end{equation*}
Throughout this work, we frequently perform re-indexing operations on tensors, which rearrange their subscripts without altering their entries. To simplify notation, we use the same variable names to represent these tensors, even when their subscripts are rearranged. For instance, consider a tensor
\begin{equation*}
    \bm{U}_{1,2,t,\mu} \in \R^{N_1 \times N_2 \times N_t \times N_{\mu}}.
\end{equation*}
We indicate with 
\begin{equation*}
    \bm{U}_{2,1,\mu,t} \in \R^{N_2 \times N_1 \times N_{\mu} \times N_t}
\end{equation*}
the result of a \textit{permutation of axes}, and with 
\begin{equation*}
    \bm{U}_{12,t\mu} \in \R^{N_{12} \times N_tN_{\mu}}, \quad N_{12} = N_1N_2,
\end{equation*}
the result of a \textit{merging of axes}. For convenience, we define $N_s = \prod_{i=1}^d N_i$, where the subscript $s$ denotes the merging all spatial axes.

The goal of our research is to leverage low-rank approximations of solution manifolds to develop an efficient \ac{rom} solver for \acp{pde}. This involves identifying a subspace of $\bm{\mathcal{U}}^{\mu}_{1,2,t}$ that effectively represents its elements. Throughout this work, we use the $\widehat{\cdot}$ symbol to denote low-rank approximations, e.g., $\widehat{\bm{U}}^{\mu}_{1,2,t} \approx \bm{U}^{\mu}_{1,2,t}$. The low-rank approximation $\widehat{\bm{U}}^{\mu}_{1,2,t}$ belongs to a reduced subspace spanned by a specific reduced basis, which we represent in a \ac{tt} format. The \ac{tt} format expresses (an approximation of) the entries of a multidimensional tensor as a sequence of three-dimensional \ac{tt} cores. For instance, consider $\bm{U}_{1,2,t,\mu}$. Its \ac{tt} representation is given by
    \begin{equation}
    \label{eq: tt cores}
    \bm{\Phi}_{\widehat{0},1,\widehat{1}} \in \R^{1 \times N_1 \times r_1}, \quad 
    \bm{\Phi}_{\widehat{1},2,\widehat{2}} \in \R^{r_1 \times N_2 \times r_2}, \quad 
    \bm{\Phi}_{\widehat{2},t,\widehat{t}} \in \R^{r_2 \times N_t \times r_t}, \quad
    \bm{\Phi}_{\widehat{t},\mu,\widehat{\mu}} \in \R^{r_t \times N_\mu \times 1},
\end{equation}
with reduced ranks (or reduced dimensions) $r_1$, $r_2$, and $r_t$, respectively. We use the notation $\widehat{i}$ to refer to the axis corresponding to a reduced rank $r_i$. For instance, in \eqref{eq: tt cores}, the subscript $\widehat{1}$ denotes an axis of length $r_1$, and similarly for other indices. As discussed in Sect.~\ref{sec: tt-rb}, the cores in \eqref{eq: tt cores} are obtained by applying a low-rank approximation method within the \ac{tt} framework to $\bm{U}_{1,2,t,\mu}$:
\begin{equation} 
    \label{eq: tensor train naive}
    \bm{U}_{1,2,t,\mu} \approx \widehat{\bm{U}}_{1,2,t,\mu} = \sum\limits_{\alpha_1 = 1}^{r_1} \sum\limits_{\alpha_2 = 1}^{r_2} \sum\limits_{\alpha_t = 1}^{r_t} \bm{\Phi}_{\widehat{0},1,\widehat{1}}\left[1,:,\alpha_1\right] \otimes \bm{\Phi}_{\widehat{1},2,\widehat{2}}\left[\alpha_1,:,\alpha_2\right] \otimes \bm{\Phi}_{\widehat{2},t,\widehat{t}}\left[\alpha_2,:,\alpha_t\right] \otimes \bm{\Phi}_{\widehat{t},\mu,\widehat{\mu}}[\alpha_t, :, 1]. 
\end{equation}
This representation is particularly advantageous for high-dimensional tensors, as the storage requirements for the \ac{tt} decomposition scale linearly with the number of dimensions and quadratically with the ranks. Low-rank algorithms are designed to compute the \ac{tt} cores by evaluating an accuracy measure as a function of the ranks and selecting the smallest ranks that ensure the error remains below a specified threshold. Assuming the input tensor is reducible \cite{quarteroni2015reduced} and the threshold is not excessively small, it is common to have $r_i \ll N_i$. 

Since \eqref{eq: tensor train naive} can be equivalently expressed by omitting the trivial axes, we may simplify the notation by disregarding the subscripts $\widehat{0}$ and $\widehat{t}$. In this interpretation, the first and last cores are treated as matrices rather than three-dimensional arrays.
In order to simplify the notation of \eqref{eq: tensor train naive}, we employ the contraction along a common axis of two multi-dimensional arrays 
\begin{equation*}
    \bm{R}_{a, b, c} \in \R^{N_a \times N_b \times N_c}, \qquad 
    \bm{S}_{c, d, e} \in \R^{N_c \times N_d \times N_e}
\end{equation*}
defined as 
\begin{equation}
    \label{eq: tt product}
    \R^{N_a \times N_b \times N_d \times N_e} \ni \bm{T}_{a, b, d, e} \doteq \bm{R}_{a, b, c} \bm{S}_{c, d, e}, \qquad \bm{T}_{a, b, d, e}\left[\alpha_a,\alpha_b,\alpha_d,\alpha_e\right] = \sum\limits_{\alpha_c} \bm{R}_{a, b, c} \left[\alpha_a,\alpha_b,\alpha_c\right] \bm{S}_{c, d, e} \left[\alpha_c,\alpha_d,\alpha_e\right].
\end{equation}
By applying \eqref{eq: tt product}, we can rewrite \eqref{eq: tensor train naive} in a more compact form as
\begin{equation}
    \label{eq: tensor train}
    \widehat{\bm{U}}_{1,2,t,\mu} = \bm{\Phi}_{\widehat{0},1,\widehat{1}} \bm{\Phi}_{\widehat{1},2,\widehat{2}} \bm{\Phi}_{\widehat{2},t,\widehat{t}} \bm{\Phi}_{\widehat{t},\mu,\widehat{\mu}} \ .
\end{equation}
Occasionally, we employ a matrix-by-tensor multiplication, often referred to as a mode-$k$ contraction. For instance, given 
\begin{equation*}
    \bm{R}_{a, b} \in \R^{N_a \times N_b}, \qquad 
    \bm{S}_{c, b, d} \in \R^{N_c \times N_b \times N_d},
\end{equation*}
we define the mode-$2$ contraction as  
\begin{equation}
    \label{eq: mode contraction}
    \R^{N_a \times N_c \times N_d} \ni \bm{T}_{a, c, d} = \bm{R}_{a, b} \odot_2 \bm{S}_{c, b, d}, \qquad \bm{T}_{a, c, d}\left[\alpha_a,\alpha_c,\alpha_d\right] = \sum\limits_{\alpha_b} \bm{R}_{a, b} \left[\alpha_a,\alpha_b\right] \bm{S}_{c, b, d} \left[\alpha_c,\alpha_b,\alpha_d\right].
\end{equation}
The \ac{tt} decomposition is a particular case of a hierarchical tensor format. In this framework, the \ac{tt} cores introduced in \eqref{eq: tt cores} are used to recursively build a hierarchical basis for the space $\bm{\mathcal{U}}_{12t}$, which can be represented as:
\begin{align}
    \label{eq: hierarchical bases}
    \bm{\mathcal{U}}_{1} = \mathrm{col} \left(\bm{\Phi}_{1,\widehat{1}}\right), \qquad
    \bm{\mathcal{U}}_{12} = \mathrm{col} \left(\bm{\Phi}_{12,\widehat{2}}\right), \qquad
    \bm{\mathcal{U}}_{12t} = \mathrm{col} \left(\bm{\Phi}_{12t,\widehat{t}}\right),
\end{align}
where $\mathrm{col}$ denotes the column space of a matrix.
\begin{figure}[t]
    \centering 
    \tikzset{every picture/.style={line width=0.75pt}} 

    \begin{tikzpicture}[x=0.6pt,y=0.6pt,yscale=-1,xscale=1]

    \draw  [fill={rgb, 255:red, 182; green, 208; blue, 238 }  ,fill opacity=1 ] (99,212) -- (260,51) -- (311,51) -- (311,72) -- (150,233) -- (99,233) -- cycle ; \draw   (311,51) -- (150,212) -- (99,212) ; \draw   (150,212) -- (150,233) ;
    \draw  [fill={rgb, 255:red, 177; green, 204; blue, 236 }  ,fill opacity=1 ] (144.44,60) -- (165.43,60) -- (44,152) -- (23.01,152) -- cycle ;
    \draw  [color={rgb, 255:red, 0; green, 0; blue, 0 }  ,draw opacity=1 ][fill={rgb, 255:red, 189; green, 213; blue, 240 }  ,fill opacity=1 ] (302,172.1) -- (379,95.8) -- (379,144.62) -- (302,220.91) -- cycle ;
    \draw  [dash pattern={on 0.84pt off 2.51pt}]  (99,233) -- (260,73) ;
    \draw  [dash pattern={on 0.84pt off 2.51pt}]  (260,73) -- (260,51) ;
    \draw  [dash pattern={on 0.84pt off 2.51pt}]  (260,73) -- (311,72) ;
    \draw  [color=red  ,draw opacity=1 ][fill=red  ,fill opacity=1 ] (77,111.25) -- (97,111.25) -- (91,116.25) -- (70,116.25) -- cycle ; 
    \draw  [color=red  ,draw opacity=1 ][fill=red  ,fill opacity=1 ] (180,131.25) -- (230,131.25) -- (224.11,136.25) -- (175,136.25) -- cycle ;
    \draw  [color=red  ,draw opacity=1 ][dash pattern={on 1.69pt off 2.76pt}] (180,152) -- (230,152) -- (224.11,156.25) -- (175,156.25) -- cycle ;
    \draw  [color=red  ,draw opacity=1 ][fill=red  ,fill opacity=1 ] (230,131.25) -- (230,152) -- (225,158) -- (225,134.49) -- cycle ;
    \draw  [color=red  ,draw opacity=1 ][dash pattern={on 1.69pt off 2.76pt}] (180,131.25) -- (180,152) -- (175,157) -- (175,135) -- cycle ;
    \draw  [color=red  ,draw opacity=1 ][fill=red  ,fill opacity=1 ] (379,96) -- (379,145) -- (374,149) -- (374,101.18) -- cycle ;
    \draw    (50.5,161) -- (174.92,65.22) ;
    \draw [shift={(176.5,64)}, rotate = 142.41] [color={rgb, 255:red, 0; green, 0; blue, 0 }  ][line width=0.75]    (10.93,-3.29) .. controls (6.95,-1.4) and (3.31,-0.3) .. (0,0) .. controls (3.31,0.3) and (6.95,1.4) .. (10.93,3.29)   ;
    \draw    (50.5,161) -- (19.5,161) ;
    \draw [shift={(17.5,161)}, rotate = 1.74] [color={rgb, 255:red, 0; green, 0; blue, 0 }  ][line width=0.75]    (10.93,-3.29) .. controls (6.95,-1.4) and (3.31,-0.3) .. (0,0) .. controls (3.31,0.3) and (6.95,1.4) .. (10.93,3.29)   ;
    \draw    (164.5,244) -- (322,84.42) ;
    \draw [shift={(323.5,83)}, rotate = 134.64] [color={rgb, 255:red, 0; green, 0; blue, 0 }  ][line width=0.75]    (10.93,-3.29) .. controls (6.95,-1.4) and (3.31,-0.3) .. (0,0) .. controls (3.31,0.3) and (6.95,1.4) .. (10.93,3.29)   ;
    \draw    (164.5,244) -- (106.5,244) ;
    \draw [shift={(104.5,244)}, rotate = 360] [color={rgb, 255:red, 0; green, 0; blue, 0 }  ][line width=0.75]    (10.93,-3.29) .. controls (6.95,-1.4) and (3.31,-0.3) .. (0,0) .. controls (3.31,0.3) and (6.95,1.4) .. (10.93,3.29)   ;
    \draw    (164.5,244) -- (164.5,215) ;
    \draw [shift={(164.5,213)}, rotate = 90] [color={rgb, 255:red, 0; green, 0; blue, 0 }  ][line width=0.75]    (10.93,-3.29) .. controls (6.95,-1.4) and (3.31,-0.3) .. (0,0) .. controls (3.31,0.3) and (6.95,1.4) .. (10.93,3.29)   ;
    \draw    (292.5,244) -- (293,184) ;
    \draw [shift={(293.5,182)}, rotate = 90.92] [color={rgb, 255:red, 0; green, 0; blue, 0 }  ][line width=0.75]    (10.93,-3.29) .. controls (6.95,-1.4) and (3.31,-0.3) .. (0,0) .. controls (3.31,0.3) and (6.95,1.4) .. (10.93,3.29)   ;
    \draw    (293,244) -- (379.09,157.41) ;
    \draw [shift={(380.5,156)}, rotate = 135] [color={rgb, 255:red, 0; green, 0; blue, 0 }  ][line width=0.75]    (10.93,-3.29) .. controls (6.95,-1.4) and (3.31,-0.3) .. (0,0) .. controls (3.31,0.3) and (6.95,1.4) .. (10.93,3.29)   ;

    \draw (10,60) node [anchor=north west][inner sep=0.75pt][font=\small]   [align=left] {$\bm{\Phi}^{\mu_{\star}}_{\widehat{0},1,\widehat{1}}\left[:,\textcolor{red}{i_1},:\right]$};

    \draw (235,22) node [anchor=north west][inner sep=0.75pt][font=\small]   [align=left] {$\bm{\Phi}^{\mu_{\star}}_{\widehat{1},2,\widehat{2}}\left[:,\textcolor{red}{i_2},:\right]$};

    \draw (350,65) node [anchor=north west][inner sep=0.75pt][font=\small]   [align=left] {$\bm{\Phi}^{\mu_{\star}}_{\widehat{2},t,\widehat{t}}\left[:,\textcolor{red}{i_t},:\right]$};

    \draw (16,164) node [anchor=north west][inner sep=0.75pt][font=\scriptsize]   [align=left] {$\widehat{1}$};
    \draw (178,64) node [anchor=north west][inner sep=0.75pt][font=\scriptsize]   [align=left] {$1$}; 
    \draw (104.5,246) node [anchor=north west][inner sep=0.75pt][font=\scriptsize]   [align=left] {$\widehat{2}$};
    \draw (169,211) node [anchor=north west][inner sep=0.75pt][font=\scriptsize]   [align=left] {$\widehat{1}$};
    \draw (323.5,83) node [anchor=north west][inner sep=0.75pt][font=\scriptsize]   [align=left] {$2$};
    \draw (280,182) node [anchor=north west][inner sep=0.75pt][font=\scriptsize]   [align=left] {$\widehat{2}$};
    \draw (380.5,156) node [anchor=north west][inner sep=0.75pt][font=\scriptsize]   [align=left] {$t$};

    \end{tikzpicture}
    \caption{Illustration of how the entry $\bm{U}_{1,2,t}^{\mu_{\star}}\left[i_1,i_2,i_t\right]$ is approximated using the \ac{tt} cores. First, the vector-matrix product is performed between the first and second cores, evaluated at the specified indices. Next, the resulting row vector is multiplied by the column vector from the third core, also evaluated at the corresponding index. The resulting scalar represents the approximated tensor entry.}
    \label{fig: tt example}
\end{figure}

To approximate $\bm{U}_{12t}^{\mu_{\star}}$ in the \ac{tt} format for a given parameter $\bm{\mu} = \bm{\mu}_{\star}$, the following steps are performed:
\begin{itemize}
    \item Construct the snapshot tensor $\bm{U}_{1,2,t,\mu}$ for $N_{\mu}$ parameter samples. Note that the superscript $\mu$ is omitted here, as the sampled parameters are assumed to sufficiently cover the parameter space.
    \item Perform a \ac{tt} decomposition of $\bm{U}_{1,2,t,\mu}$ to compute the cores $\bm{\Phi}_{\widehat{0},1,\widehat{1}}, \bm{\Phi}_{\widehat{1},2,\widehat{2}}, \bm{\Phi}_{\widehat{2},t,\widehat{t}}$, which span the \ac{rb} subspace $\bm{\mathcal{U}}_{12t}^{\mu_{\star}}$. Note that the parametric core $\bm{\Phi}_{\widehat{t},\mu,\widehat{\mu}}$ is not computed. The dimension of the \ac{rb} space corresponds to the temporal rank $r_t$, which is no longer necessarily equal to $1$.
    \item Determine $\widehat{\bm{U}}_{\widehat{t}}^{\mu_{\star}}$ such that 
\end{itemize}
\begin{equation}
    \label{eq: tt approximation}
    \bm{U}_{1,2,t}^{\mu_{\star}} \approx \widehat{\bm{U}}_{1,2,t}^{\mu_{\star}} = \bm{\Phi}_{\widehat{0},1,\widehat{1}} \bm{\Phi}_{\widehat{1},2,\widehat{2}} \bm{\Phi}_{\widehat{2},t,\widehat{t}} \ \widehat{\bm{U}}_{\widehat{t}}^{\mu_{\star}},
\end{equation}
As discussed in Sect.~\ref{sec: tt-rb}, $\widehat{\bm{U}}_{\widehat{t}}^{\mu_{\star}}$ represents the vector of coordinates of $\bm{U}_{1,2,t}^{\mu_{\star}}$ in the \ac{tt} basis. This vector is the unknown in the \ac{ttrb} method.

Finally, we introduce the following ``multi-axes'' notation to describe tensors and their dimensions when the value of $d$ is not explicitly specified:
\begin{equation}
    \label{eq: multi-axes notation}
    \bm{U}_{1,\hdots,\mu} \in \R^{N_{1,\hdots,\mu}} = \bm{U}_{1,2,\hdots,\mu} \in \R^{N_{1,2,\hdots,\mu}} = \hdots = \bm{U}_{1,2,\hdots,d,t,\mu} \in \R^{N_{1,2,\hdots,d,t,\mu}},
\end{equation}
where
\begin{equation*}
    N_{1,\hdots,\mu} = N_{1,2,\hdots,\mu} = \hdots = N_{1,2,\hdots,d,t,\mu} = N_1 \times N_2 \times \cdots \times N_d \times N_t \times N_{\mu}.
\end{equation*}
When merging multiple axes, on the other hand, we employ 
\begin{equation}
    \label{eq: multi-axes vec notation}
    \bm{U}_{1:\mu} \in \R^{N_{1:\mu}} = \bm{U}_{12:\mu} \in \R^{N_{12:\mu}} = \hdots = \bm{U}_{12\cdots d t:\mu} \in \R^{N_{12\cdots d t:\mu}},
\end{equation}
where 
\begin{equation*}
    N_{1:\mu} = N_{12:\mu} = N_{12\cdots d\cdot t:\mu} = N_1 \cdot N_2 \cdots N_d \cdot N_t \cdot N_{\mu} = N_{s t \mu}.
\end{equation*}
Naturally, we can combine the notations introduced in \eqref{eq: multi-axes notation} and \eqref{eq: multi-axes vec notation}, e.g.,
\begin{equation*}
    \bm{U}_{1:i,\hdots,\mu} \in \R^{N_{1:i,\hdots,\mu}}
\end{equation*}
represents a tensor with $d-i+3$ dimensions. Additionally, the temporal axis is often referred to as the $(d+1)$-th axis for clarity. 
\begin{remark}
    For vector-valued problems, the snapshots must be expressed in the ``split-axes'' format by introducing an additional axis for the components:
    \begin{equation}
      \label{eq: vector value tt}
      \bm{U}_{1,\dots,d,c,t,\mu} \in \R^{N_1 \times \cdots \times N_d \times N_c \times N_t \times N_{\mu}},
    \end{equation} 
    where the subscript $c$ represents the components of the vector field, and $N_c$ denotes the number of components (typically $N_c = d$ for most problems). In this case, the \ac{tt} decomposition includes an additional three-dimensional core to account for the components axis.
\end{remark}
\section{Reduced basis method in space time}
\label{sec: rb}

We begin this section by introducing the \ac{fom} given by a transient, parameterized \ac{pde} on a $d$-cube. Then, we provide an overview of the \ac{strb} method applied to the \ac{fom}.

\subsection{Full order model} 
\label{subs:full order model}
We consider a $d$-cube
\begin{equation*}
    \Omega = \Omega_1 \times \cdots \times \Omega_d \quad \text{where} \quad \Omega_i \subset \R \ \forall \ i = 1,\hdots, d,
\end{equation*}
with boundary $\partial\Omega$. For transient problems, we also introduce a temporal domain $[0,T] \subset \R_+ \cup \{0\}$ and a parameter space $\mathcal{D} \subset \R^p$. For a given parameter $\bm{\mu} \in \mathcal{D}$, we consider a generic parameterized \ac{pde} defined over the space-time domain $\Omega \times [0,T]$, which takes the form:
\begin{equation}
    \begin{split}
        \label{eq: strong form pde}
        \frac{\partial u^{\mu}}{\partial t} + \mathcal{A}^{\mu} (u^{\mu}) &= f^{\mu} \quad \text{in } \Omega \times (0,T], \\
        u^{\mu} &= u_0^{\mu} \quad \text{in } \Omega \times \{0\},
    \end{split} 
\end{equation}
with appropriate boundary conditions on $\partial \Omega$. Here, $u^{\mu}: \Omega \times [0,T] \to \R$ represents the unknown state variable, $f^{\mu}: \Omega \times [0,T] \to \R$ is the forcing term, and $u_0^{\mu}: \Omega \to \R$ specifies the initial condition. The operator $\mathcal{A}^{\mu}: \R \to \R$ is a linear differential operator, whose explicit form is left unspecified. We also define the Dirichlet and Neumann boundaries, $\Gamma_D$ and $\Gamma_N$, respectively, such that $\{\Gamma_D, \Gamma_N\}$ partitions $\partial\Omega$. For simplicity, we assume that $\Gamma_D$ consists of entire legs of the $d$-cube, meaning that each element of $\Gamma_D$ corresponds to a complete $(d-1)$-dimensional facet of $\partial \Omega$. This assumption is made solely to facilitate the analysis of the \ac{fom} with strongly imposed Dirichlet conditions and does not entail a loss of generality. If this condition is not satisfied, the Dirichlet data can instead be imposed weakly, for instance, using a Nitsche penalty method \cite{Hansbo,BurmanHansbo}.

We now introduce a conforming, quasi-uniform quadrilateral partition of $\Omega$, denoted by $\mathcal{T}_h$, along with a uniform temporal partition of $\left[0,T\right]$, denoted by $\{t_n\}_{n=0}^{N_t}$. The spatial mesh $\mathcal{T}_h$ is characterized by a parameter $h$ representing its size, while the temporal mesh is defined such that $t_n = n \delta$, where $\delta = T/N_t$ is the time-step size. Notably, $\mathcal{T}_h$ can be expressed as the tensor product of one-dimensional partitions defined on $\Omega_1, \hdots, \Omega_d$, a property that will be extensively utilized throughout this work. For the spatial discretization of \eqref{eq: strong form pde}, we consider the following Hilbert spaces:
\begin{equation}
    \mathcal{V} = H^1(\Omega); \qquad \mathcal{V}^0_{\Gamma_D}  
    = \left\{ v \in H^1(\Omega) \ \ : \   v=0 \  \text{on} \ \Gamma_D\right\},
\end{equation}
and their finite-dimensional counterparts 
\begin{equation*}
    \mathcal{V}_h = \{v_h \in C^0(\Omega) \ : \ v_h|_K \in \mathcal{P}_p(K) \ \forall \ K \in \mathcal{T}_h, \ p \geq 1 \} \subset \mathcal{V}, \qquad 
    \mathcal{V}_h^0 = \mathcal{V}_h \cap \mathcal{V}^0_{\Gamma_D}.
\end{equation*}
Here, $\mathcal{P}_p$ represents the space of polynomials of degree at most $p$. For temporal discretization, we adopt the \ac{be} scheme for clarity of exposition, although a Crank-Nicolson method is employed in our numerical experiments. In this work, Dirichlet boundary conditions are enforced strongly. By formulating the lifted weak form of \eqref{eq: strong form pde} and applying numerical integration, the \ac{fom} can be expressed algebraically as
\begin{align}
\label{eq: theta method}
    \left(\delta^{-1}\bm{M}_{s,s}+ \bm{A}_{s,s}^{\mu}(t_n)\right)\left(\bm{U}^{\mu}_s\right)_n = 
    \bm{L}_s^{\mu}(t_n) +  \delta^{-1}\bm{M}_{s,s} \left(\bm{U}^{\mu}_s\right)_{n-1} , \quad n \in \N(N_t), 
\end{align}
where $\N(k) = \{1,\hdots,k\}$ for any positive integer $k$. The vector $\left(\bm{U}^{\mu}_s\right)_n \in \R^{N_s}$ represents the \acp{dof} of the \ac{fe} approximation $\left(u_h^{\mu}\right)_n \in \mathcal{V}_h^0$ at the time instant $t_n$, with homogeneous Dirichlet boundary conditions. Due to the definitions of $\Omega$ and $\mathcal{T}_h$, the \ac{fe} basis functions spanning $\mathcal{V}_h$ (and consequently $\mathcal{V}_h^0$, given the assumptions on $\Gamma_D$) exhibit a tensor product structure. Specifically, these basis functions can be expressed as the tensor product of \ac{fe} basis functions defined on the one-dimensional partitions that collectively form $\mathcal{T}_h$ \cite{LeDret2016}. This tensor product structure allows us to identify $\left(\bm{U}^{\mu}_s\right)_n$ with $(\bm{U}^{\mu}_{1,\hdots,d})_n \in \R^{N_1 \times \cdots \times N_d}$. In \eqref{eq: theta method}, the symbols $\bm{M}_{s,s}$, $\bm{A}_{s,s}^{\mu}$, and $\bm{L}_s^{\mu}$ denote the mass matrix (independent of $t$ and $\bm{\mu}$ in this work), the stiffness matrix, and the \ac{rhs} vector, respectively. The \ac{rhs} vector can also be represented as a $d$-dimensional tensor, while the mass and stiffness matrices can be viewed as 2-$d$ tensors. Notably, $\bm{M}_{s,s}$ and $\bm{A}^{\mu}_{s,s}$ are sparse matrices and can therefore be equivalently represented by their vectors of nonzero entries, $\bm{M}_z$ and $\bm{A}^{\mu}_z \in \R^{N_z}$, where $N_z$ denotes the number of nonzero entries.
Writing \eqref{eq: theta method} at every time step yields the space-time algebraic system
\begin{align}
    \label{eq: space time FOM}
    \bm{K}^{\mu}_{st,st}\bm{U}^{\mu}_{st}=\bm{L}^{\mu}_{st}.
\end{align}
In this formulation, $\bm{K}^{\mu}_{st,st}$ is a block bi-diagonal matrix with $N_t$ diagonal blocks of the form $\delta^{-1} \bm{M}_{s,s} + \bm{A}_{s,s}^{\mu}(t_i)$ and lower diagonal blocks given by $\delta^{-1} \bm{M}_{s,s}$. The space-time vectors $\bm{U}^{\mu}_{st}$ and $\bm{L}^{\mu}_{st}$ are constructed by vertically concatenating their spatial components across all time steps. Additionally, we define the discrete Laplacian matrix $\bm{X}_{s,s}$, which is symmetric, positive definite, and represents the $H^1_0$ inner product on $\mathcal{V}_h^0$. On a Cartesian mesh, this matrix exhibits a structured sparsity pattern.
\begin{equation}
    \label{eq: H1 norm matrix}
    \begin{split}
        \bm{X}_{s,s} &= \bm{X}_{1,1} \otimes \bm{M}_{2,2} \otimes \cdots \otimes \bm{M}_{d,d} + \bm{M}_{1,1} \otimes \bm{X}_{2,2} \otimes \cdots \otimes \bm{M}_{d,d}
        + \hdots +
        \bm{M}_{1,1} \otimes \bm{M}_{2,2} \otimes \cdots \otimes \bm{X}_{d,d} \\ 
        &= \sum\limits_{k=1}^d \overset{d}{\underset{i=1}{\otimes}} \bm{Y}_{i,i}^k, \quad \text{where} \quad
        \bm{Y}_{i,i}^k = \bm{X}_{i,i} \quad \text{if} \quad i=k, \quad \bm{Y}_{i,i}^k = \bm{M}_{i,i} \quad \text{otherwise}.
    \end{split}
\end{equation}
Here, $\bm{X}_{i,i}$ and $\bm{M}_{i,i}$ denote the discrete Laplacian and mass matrices associated with the $i$th one-dimensional \ac{fe} space, respectively. Recall that $\mathcal{V}_h$ and $\mathcal{V}_h^0$ are constructed with a tensor product structure. Consequently, $\bm{X}_{s,s}$ can be interpreted as a $d$-rank tensor. Additionally, we define the global spatio-temporal norm matrix $\bm{X}_{st,st}$ as a block-diagonal matrix with $N_t$ blocks, each given by $\delta \bm{X}_{s,s}$. The factor $\delta$ accounts for the $L^2(0,T;\mathcal{V})$ inner product of the \ac{fe} basis functions.

\subsection{Space-time reduced-basis method}
\label{subs: st-rb}
The \ac{strb} method is a data-driven approach that involves two main stages:
\begin{enumerate}
  \item A computationally intensive \emph{offline phase}, during which the spatio-temporal basis is constructed, and the (Petrov-) Galerkin projection of the \ac{fom} \eqref{eq: space time FOM} is precomputed.
  \item A computationally efficient \emph{online phase}, where the \ac{rb} approximation is rapidly evaluated for any given parameter $\bm{\mu}$.
\end{enumerate}
We define two distinct sets of parameters: the offline parameter set, $\mathcal{D}_{\mathrm{off}}$, and the online parameter set, $\mathcal{D}_{\mathrm{on}}$. These sets are given by 
\begin{equation}
    \label{eq: disjoint params}
    \mathcal{D}_{\mathrm{off}} = \{\bm{\alpha}_k\}_{k=1}^{N_{\mu}} \subset \mathcal{D}; \qquad 
    \mathcal{D}_{\mathrm{on}} = \{\bm{\beta}_k\}_{k=1}^{N_{\mathrm{on}}} \subset \mathcal{D}.
\end{equation}
To construct the (offline) \ac{fom} snapshots $\bm{U}_{s,t,\mu}$, we solve and store the solution of \eqref{eq: space time FOM} for each parameter $\bm{\mu}_k \in \mathcal{D}_{\mathrm{off}}$. Unlike the ``split-axes'' representation of snapshots, a standard \ac{strb} method directly uses these snapshots. From the computed snapshots, we derive an $\bm{X}_{s,s}$-orthogonal spatial basis and an $\ell^2$-orthogonal temporal basis. The space-time basis, $\bm{X}_{st,st}$-orthogonal, is then constructed as the Kronecker product of these two bases. This space-time basis defines the projection subspace used in the \ac{strb} method. The entire procedure, along with the computational cost of each step, is detailed in Alg. \ref{alg: tpod}. 
\begin{algorithm}[t]
    \caption{\texttt{TPOD}: Construct the $\bm{X}_{s,s}$-orthogonal spatial basis $\bm{\Phi}_{s,\widehat{s}}$ and the $\ell^2$-orthogonal temporal basis $\bm{\Phi}_{t,\widehat{t}}$ from the tensor of space-time snapshots $\bm{U}_{s,t,\mu}$, given a prescribed accuracy $\varepsilon$ and the norm matrix $\bm{X}_{s,s}$.}
    \begin{algorithmic}[1]
        \Function{\texttt{TPOD}}{$\bm{U}_{s,t,\mu}, \bm{X}_{s,s}, \varepsilon $}
        \State Cholesky factorization: $\bm{H}^T_{s,s}\bm{H}_{s,s} = \texttt{Cholesky}\left(\bm{X}_{s,s} \right)$ \Comment{$\mathcal{O}\left(N_s b^2\right)$}
        \State Spatial rescaling: $\widetilde{\bm{U}}_{s,t\mu} = \bm{H}_{s,s}\bm{U}_{s,t\mu}$ \Comment{$\mathcal{O}\left(N_z N_t N_{\mu}\right)$}
        \State Spatial reduction: $\widetilde{\bm{\Phi}}_{s,\widehat{s}},\widetilde{\bm{R}}_{\widehat{s},t\mu} = \texttt{RSVD}(\widetilde{\bm{U}}_{s,t\mu},\varepsilon)$ \Comment{$\mathcal{O}\left(N_sN_t N_{\mu} \log{(N_t N_{\mu})}\right)$}
        \State Spatial inverse rescaling: $\bm{\Phi}_{s,\widehat{s}} = \bm{H}_{s,s}^{-1} \widetilde{\bm{\Phi}}_{s,\widehat{s}}$ \Comment{$\mathcal{O}\left(r_sN_s^2\right)$}
        \State Spatial contraction: $\widehat{\bm{U}}_{\widehat{s} ,t\mu} = \widetilde{\bm{\Phi}}_{\widehat{s},s} \widetilde{\bm{U}}_{s,t\mu}$ \Comment{$\mathcal{O}\left(r_s N_s N_t N_{\mu}\right)$}
        \State Temporal reduction: $\bm{\Phi}_{t,\widehat{t}},\bm{R}_{\widehat{t},\widehat{s}\mu} = \texttt{RSVD}(\widehat{\bm{U}}_{t,\widehat{s}\mu},\varepsilon)$ \Comment{$\mathcal{O}\left(r_sN_tN_{\mu}\log{\min\{r_sN_{\mu},N_t\}}\right)$}
        \State Return $\bm{\Phi}_{s,\widehat{s}}$, $\bm{\Phi}_{t,\widehat{t}}$
        \EndFunction
    \end{algorithmic}
    \label{alg: tpod}
\end{algorithm} 
The two $\texttt{RSVD}$ functions are used to compute the $r_s$-dimensional spatial subspace and the $r_t$-dimensional temporal subspace, respectively. The \ac{rsvd}, originally introduced in \cite{halko2010findingstructurerandomnessprobabilistic}, is designed to reduce the computational cost of the traditional \ac{svd}. Assuming $N_s > N_tN_{\mu}$, the computational complexity of a full \ac{rsvd} is given by: 
\begin{equation}
    \label{eq: cost RSVD}
      \texttt{RSVD}(\bm{U}_{s,t\mu}) \sim \mathcal{O}(N_s N_t N_{\mu} \log{(N_t N_{\mu})}).
    \end{equation}
    The ranks $r_s$ and $r_t$ are determined using the relative energy criterion, as described in \cite{quarteroni2015reduced}. For further details on the practical implementation of \texttt{RSVD}, we refer the reader to \cite{ken_ho_2022_6394438}. Additionally, we note that the quantity
\begin{equation}
    \widetilde{\bm{R}}_{\widehat{s},t\mu} = \bm{\Sigma}_{\widehat{s},\widehat{s}} \bm{V}^T_{\widehat{s},t \mu},
\end{equation}
represents the residual of the spatial compression. Specifically, the matrix $\bm{\Sigma}_{\widehat{s},\widehat{s}}$ contains the singular values, while $\bm{V}^T_{\widehat{s},t \mu}$ holds the right singular vectors of $\bm{U}_{s,t\mu}$. Similarly, $\widetilde{\bm{R}}_{\widehat{t},\widehat{s}\mu}$ represents the residual of the temporal compression. These residuals play a crucial role, particularly when employing \ac{ttsvd}, as will be discussed in the next section. The cost estimates in Alg. \ref{alg: tpod} are derived based on \eqref{eq: cost RSVD} and the following observations:
\begin{itemize}
    \item The Cholesky factorization of an $N_s \times N_s$ sparse matrix has a computational complexity of $\mathcal{O}\left(N_s b^2\right)$ \cite{doi:https://doi.org/10.1002/0471249718.ch1}, where $b$ denotes the semi-bandwidth of the matrix. Further details regarding $b$ are provided in Sect. \ref{sec: tt-rb}.
    \item The Cholesky factor $\bm{H}_{s,s}$ is sparse, with approximately $N_z$ nonzero entries. Consequently, the cost of the spatial rescaling $\bm{H}_{s,s}\bm{U}_{s,t\mu}$ is $\mathcal{O}\left(N_z N_t N_{\mu}\right)$. On a Cartesian mesh, the number of nonzero entries is given by 
      \begin{equation*}
        N_z = \mathcal{O}\left((pd)^d N_s\right),
      \end{equation*}
      where $p$ is the polynomial order of the \ac{fe} space. Thus, the complexity can also be expressed as $\mathcal{O}\left((pd)^d N_s N_t N_{\mu}\right)$.
    \item Since $\bm{H}_{s,s}$ is sparse, the inverse rescaling $\bm{H}_{s,s}^{-1} \widetilde{\bm{\Phi}}_{s,\widehat{s}}$ generally requires fewer operations than the worst-case scenario of $\mathcal{O}(r_s N_s^2)$, which would occur if $\bm{H}_{s,s}$ were a full triangular matrix. However, due to the lack of a general expression for the bandwidth of $\bm{H}_{s,s}$, we conservatively assume the worst-case cost. This assumption does not affect the overall computational cost analysis, as this step is computationally negligible.
    \item The spatial contraction step corresponds to the \ac{sthosvd} approach \cite{carlberg2017galerkin}, which reduces the cost of temporal compression. The computational cost of \ac{sthosvd} is dominated by the spatial compression step, which involves two matrix-matrix multiplications. In comparison, the subsequent \ac{rsvd} step is negligible. Alternatively, the temporal basis could be computed directly from $\widetilde{\bm{R}}_{\widehat{s},t\mu}$, bypassing the spatial contraction step. This approach essentially corresponds to the Tucker decomposition of $\widetilde{\bm{U}}_{s,t,\mu}$. However, both strategies yield similar results and do not impact the overall computational cost of \ac{tpod}, so this remains a minor implementation detail.
\end{itemize}
As outlined earlier, the \ac{strb} method employs a space-time basis defined as 
\begin{equation*}
    \R^{N_{st} \times r_{st}} \ni \bm{\Phi}_{st,\widehat{st}} = \bm{\Phi}_{s,\widehat{s}} \otimes \bm{\Phi}_{t,\widehat{t}},
\end{equation*}
where $r_{st} = r_s r_t$ denotes the dimension of the reduced subspace, indexed by the subscript $\widehat{st}$. In \ac{tpod}, this dimension is the product of the spatial and temporal subspace dimensions. The accuracy of the space-time basis is given by:
\begin{equation}
    \label{eq: TPOD}
    \sum\limits_{j=1}^{N_{\mu}} \|\left(\bm{U}_{st,\mu} -  \bm{\Phi}_{st, \widehat{st}} \bm{\Phi}_{\widehat{st}, st} \bm{X}_{st,st} \bm{U}_{st,\mu}\right)\left[:,j\right]\|^2_{\bm{X}_{st,st}} \leq 
    \varepsilon^2 \left(\|\bm{U}_{s,t\mu}\|_F^2 + \|\widehat{\bm{U}}_{t,\widehat{s}\mu}\|_F^2\right).
\end{equation}
The term $\widehat{\bm{U}}_{t,\widehat{s}\mu}$ in the estimate corresponds to the spatial contraction computed during the \ac{sthosvd} procedure. The result in \eqref{eq: TPOD} is derived by combining the findings from \cite{MUELLER2024115767}, which address the case of an $\ell^2$-orthogonal basis, with the results from \cite{quarteroni2015reduced,negri2015reduced}, where the relationship between an $\ell^2$-orthogonal basis and an $\bm{X}_{st,st}$-orthogonal basis is established.

We now describe the online phase of the \ac{strb} method, where the reduced version of \eqref{eq: space time FOM} is assembled and solved. Specifically, we consider the reduced equations obtained by projecting \eqref{eq: space time FOM} onto the subspace spanned by $\bm{\Phi}_{st,\widehat{st}}$ using a Galerkin projection. For a comprehensive review of the more general Petrov-Galerkin projections in the context of \ac{rb} methods, we refer the reader to \cite{dalsanto2019algebraic}. Using algebraic notation, the \ac{strb} problem is formulated as:
\begin{equation}
    \label{eq: space time reduced problem}
    \text{find} \quad \widehat{\bm{U}}_{\widehat{st}}^{\mu} \quad \text{such that} \quad 
    \bm{\Phi}_{\widehat{st},st} \left(\bm{L}^{\mu}_{st} - \bm{K}^{\mu}_{st,st}\bm{\Phi}_{st, \widehat{st}}\widehat{\bm{U}}_{\widehat{st}}^{\mu}\right)  = \zero{\widehat{st}} \Longleftrightarrow \widehat{\bm{K}}_{\widehat{st},\widehat{st}}^{\mu}\widehat{\bm{U}}_{\widehat{st}}^{\mu} = \widehat{\bm{L}}_{\widehat{st}}^{\mu},
\end{equation}
where 
\begin{equation*}
    \widehat{\bm{K}}_{\widehat{st},\widehat{st}}^{\mu} = \bm{\Phi}_{\widehat{st}, st} \bm{K}^{\mu}_{st,st}\bm{\Phi}_{st, \widehat{st}}, \qquad 
    \widehat{\bm{L}}_{\widehat{st}}^{\mu} = \bm{\Phi}_{\widehat{st}, st} \bm{L}^{\mu}_{st}
\end{equation*}
represent the Galerkin projections of the space-time Jacobian and residual, respectively. Since computing these quantities involves operations that scale with the full-order dimensions, employing a hyper-reduction strategy to approximate the Jacobians and residuals is essential for efficiency.
\begin{remark}
    Although this work focuses exclusively on linear problems, we use the terms Jacobian and residual to refer to the \ac{lhs} and \ac{rhs} of the problem, respectively, for the sake of generality. Notably, the described hyper-reduction techniques are equally applicable to both linear and nonlinear problems without requiring any modifications.
\end{remark}
To achieve hyper-reduction, methods like \ac{mdeim} \cite{negri2015reduced,MUELLER2024115767} aim to construct the following affine expansions:
\begin{equation}
    \label{eq: affine expansions}
    \bm{K}^{\mu}_{st,st} \approx \sum_{i=1}^{r_{st}^{\bm{K}}} \bm{\Phi}^{\bm{K}}_{st,\widehat{st}^{\bm{K}},st}[:,i,:] \widehat{\bm{K}}^{\mu}_{\widehat{st}^{\bm{K}}}[i]; \qquad 
    \bm{L}^{\mu}_{st} \approx \sum_{i=1}^{r_{st}^{\bm{L}}}\bm{\Phi}^{\bm{L}}_{st,\widehat{st}^{\bm{L}}}[:,i] \widehat{\bm{L}}^{\mu}_{\widehat{st}^{\bm{L}}}[i].
\end{equation}  
Here, 
\begin{equation*}
    \bm{\Phi}^{\bm{K}}_{st,\widehat{st}^{\bm{K}},st} \in \R^{N_{st} \times r_{st}^{\bm{K}} \times N_{st}}, \qquad 
    \bm{\Phi}^{\bm{L}}_{st,\widehat{st}^{\bm{K}}} \in \R^{N_{st} \times r_{st}^{\bm{L}}}
\end{equation*}
represent two bases that span reduced-dimensional subspaces, used to approximate the manifold of parameterized Jacobians and residuals. The goal is to solve the approximated \ac{rom} by substituting the affine expansions \eqref{eq: affine expansions} into \eqref{eq: space time reduced problem}:
\begin{equation}
    \label{eq: space time mdeim reduced problem}
    \text{find} \quad \widehat{\bm{U}}_{\widehat{st}}^{\mu} \quad \text{such that} \quad \widebar{\bm{K}}_{\widehat{st},\widehat{st}}^{\mu}\widehat{\bm{U}}_{\widehat{st}}^{\mu} = \widebar{\bm{L}}_{\widehat{st}}^{\mu},
\end{equation}
where
\begin{equation*}
    \widehat{\bm{K}}^{\mu}_{\widehat{st},\widehat{st}} \approx \widebar{\bm{K}}^{\mu}_{st,st} = \sum_{i=1}^{r_{st}^{\bm{K}}}\bm{\Phi}_{\widehat{st},st} \bm{\Phi}^{\bm{K}}_{st,\widehat{st}^{\bm{K}},st}\left[:,i,:\right] \bm{\Phi}_{st,\widehat{st}} \widehat{\bm{K}}^{\mu}_{\widehat{st}^{\bm{K}}}\left[i\right]; \qquad 
    \widehat{\bm{L}}^{\mu}_{\widehat{st}} \approx \widebar{\bm{L}}^{\mu}_{st} = \sum_{i=1}^{r_{st}^{\bm{L}}}\bm{\Phi}_{\widehat{st},st} \bm{\Phi}^{\bm{L}}_{st,\widehat{st}^{\bm{L}}}\left[:,i\right] \widehat{\bm{L}}^{\mu}_{\widehat{st}^{\bm{L}}}\left[i\right].
\end{equation*}
Since the bases are $\bm{\mu}$-independent, most of the Galerkin projection computations can be performed offline. During the online phase, it suffices to compute the reduced coefficients $\widehat{\bm{K}}_{\widehat{st}^{\bm{K}}}$ and $\widehat{\bm{L}}_{\widehat{st}^{\bm{L}}}$, followed by their respective multiplications with the projected bases. These operations depend only on the reduced dimensions $r_{st}$, $r_{st}^{\bm{K}}$, and $r_{st}^{\bm{L}}$, making them independent of the \ac{fom} dimensions. This ensures the computational efficiency typically associated with \acp{rom}. In this work, we employ \ac{mdeim} as the hyper-reduction strategy, with all relevant details provided in Subsection~\ref{subs: eim}.

\section{A novel TT-RB solver}
\label{sec: tt-rb}

In this section, we discuss the \ac{ttrb} method. As before, our goal is to solve the reduced problem \eqref{eq: space time reduced problem}, but the projection operator is now expressed in a \ac{tt} format:
\begin{equation}
    \label{eq: ttrb basis}
    \bm{\Phi}_{st,\widehat{t}} = \bm{\Phi}_{\widehat{0},1,\widehat{1}} \cdots \bm{\Phi}_{\widehat{d},t,\widehat{t}}.
\end{equation}
We recall that the dimension of the projection subspace is now represented by the axis $\widehat{t}$, instead of $\widehat{st}$ as in \ac{strb}. Specifically, the dimension of a \ac{tt} subspace is determined by the last reduced dimension (see \eqref{eq: hierarchical bases} for more details), whereas in the \ac{strb} case, it is given by the product of the reduced dimensions. In essence, the dimension of the \ac{tt} subspace is independent of the dimension of the snapshots, unlike its \ac{strb} counterpart. 

The content of the section is organized as follows: 
\begin{itemize}
    \item We introduce the standard (Euclidean) \ac{ttsvd} algorithm, originally proposed in \cite{oseledets2011tensor}. We also discuss the accuracy of the procedure, both in the Euclidean case and when imposing an $\bm{X}_{st,st}$-orthogonality condition.
    \item We present a modified \ac{ttsvd} algorithm for computing an $\bm{X}_{st,st}$-orthogonal \ac{tt} decomposition, when $\bm{X}_{s,s}$ is a rank-1 matrix.
    \item We extend the previous algorithm to the more general case of a rank-$K$ norm matrix $\bm{X}_{s,s}$.
    \item We define the standard \ac{mdeim} procedure and describe its extension to empirically interpolate \ac{tt} decompositions, referred to as \ac{ttmdeim}.
    \item We elaborate a method for projecting the \ac{ttmdeim} approximation of residuals and Jacobians, solely exploiting operations on the cores.
    \item Lastly, we present an accuracy measure for the resulting \ac{ttrb} method.
\end{itemize}

\subsection{Basis construction: a priori estimates}
\label{subs: basis construction}

We consider the \ac{fom} snapshots $\bm{U}_{s,t,\mu}$ already introduced in Sect~\ref{sec: rb}, but now expressed in the ``split axes'' format, i.e. $\bm{U}_{1,\hdots,\mu}$. A \ac{tt} decomposition is commonly extracted from the snapshots by running either a \ac{ttsvd} or a \ac{ttcross} strategy. Despite being cheaper, the latter presents two drawbacks: the hierarchical bases have a larger rank for a fixed accuracy, and the unavailability of a priori error estimates. Moreover, in this work we only deal with snapshots tensors of order at most $5$ (in transient problems on a 3-$d$ domain). The \ac{ttcross} method is generally shown to outperform \ac{ttsvd} primarily when compressing tensors of much higher order than those considered here. For these reasons, we construct our \ac{rb} subspaces using the \ac{ttsvd} methodology.

In Alg. \ref{alg: ttsvd}, we review the \ac{ttsvd} presented in \cite{oseledets2011tensor}. For conciseness, we employ the multi-axes notation in \eqref{eq: multi-axes notation}-\eqref{eq: multi-axes vec notation}, and we use the subscript $d+1$ to denote the temporal axis.
\begin{algorithm}[t]
    \caption{\texttt{TT-SVD}: Given the snapshots tensor in the ``split-axes'' format $ \bm{U}_{1,\hdots,\mu}$ and the prescribed accuracy $ \varepsilon $, build the \ac{tt} cores $\bm{\Phi}_{\widehat{0},1,\widehat{1}},\hdots,\bm{\Phi}_{\widehat{d-1},d,\widehat{d}},\bm{\Phi}_{\widehat{d},t,\widehat{t}}$.} 
    \begin{algorithmic}[1]
        \Function{\texttt{TT-SVD}}{$ \bm{U}_{1,\hdots,\mu}, \varepsilon $}
        \State Initialize unfolding matrix: $\bm{T}_{1,2:\mu} = \bm{U}_{1,2:\mu}$ \Comment{$\bm{T}_{1,2:\mu} \in \R^{N_1 \times N_{2:\mu}}$}
        \For{i = $1, \hdots , d $} 
        \State $i$th spatial reduction: $\bm{\Phi}_{\widehat{i-1} i, \widehat{i}}\bm{R}_{\widehat{i},i+1 :  \mu} = \texttt{RSVD}(\bm{T}_{\widehat{i-1} i, i+1 : \mu},\varepsilon)$ \Comment{$\bm{\Phi}_{\widehat{i-1} i, \widehat{i}} \in \R^{r_{i-1}N_i \times r_i}$}
        \State Update unfolding matrix: $\bm{T}_{\widehat{i} ,i+1 : \mu} = \bm{R}_{\widehat{i},i+1 : \mu}$ \Comment{$\bm{T}_{\widehat{i} ,i+1 : \mu} \in \R^{r_i \times N_{i+1:\mu}}$}
        \EndFor
        \State Temporal reduction: $\bm{\Phi}_{\widehat{d}t,\widehat{t}}\bm{R}_{\widehat{t},\mu} = \texttt{RSVD}(\bm{T}_{\widehat{d}  t,\mu},\varepsilon)$ \Comment{$\bm{\Phi}_{\widehat{d}t,\widehat{t}} \in \R^{r_d N_t \times r_t}$}
        \State \Return  $\bm{\Phi}_{\widehat{0},1,\widehat{1}}, \hdots, \bm{\Phi}_{\widehat{d-1},d,\widehat{d}}, \bm{\Phi}_{\widehat{d},t,\widehat{t}}$
        \EndFunction
    \end{algorithmic}
    \label{alg: ttsvd}
\end{algorithm} 
The \ac{ttsvd} computes the \ac{tt} cores by successively applying a \ac{rsvd} on the (truncated) remainder of the previous iteration. Similarly to \ac{tpod}, we use the hyper-parameter $\varepsilon$ to control the error of the algorithm. In particular, the following accuracy statement holds.
\begin{theorem}
    \label{thm: ttsvd accuracy}
    Suppose the unfolding matrices $\bm{T}_{\widehat{i-1}  i,i+1 : \mu}$
    \begin{equation*}
        \bm{T}_{1,2:\mu} = \bm{U}_{1,2:\mu}, \quad \bm{T}_{\widehat{i-1}  i,i+1 : \mu} = \bm{R}_{\widehat{i-1}  i,i+1 : \mu} \quad \forall \ i=2,\hdots,d+1
    \end{equation*}
    admit a low-rank approximation with relative errors $ \varepsilon_i$ for given ranks $r_i$ : 
    \begin{align}
        \label{eq: assumption thm 1}
        \bm{T}_{\widehat{i-1}  i, i+1 : \mu} = \bm{\Phi}_{\widehat{i-1}  i,\widehat{i}}\bm{R}_{\widehat{i},i+1:\mu} + \bm{E}_{\widehat{i-1}  i, i+1 : \mu} \ , \quad 
        \|\bm{E}_{\widehat{i-1}  i, i+1 : \mu}\|^2_F = \varepsilon_i^2 \|\bm{T}_{\widehat{i-1}  i, i+1 : \mu}\|^2_F, \quad \forall \ i=1,\hdots,d+1.
    \end{align}
    Then, the projection operator $\bm{\Phi}_{st,\widehat{t}} $ in \eqref{eq: ttrb basis} satisfies:
    \begin{equation}
        \label{eq: ttsvd accuracy}
        \|\bm{U}_{st,\mu} - \bm{\Phi}_{st,\widehat{t}}\bm{\Phi}_{\widehat{t},st}\bm{U}_{st,\mu}\|_F^2 \leq \sum\limits_{i=1}^{d+1}\varepsilon_i^2 \|\bm{T}_{\widehat{i-1}  i, i+1 : \mu}\|^2_F \leq \varepsilon^2 (d+1) \|\bm{U}_{st,\mu}\|^2_F,
    \end{equation}
    where $\varepsilon = \sup_i \varepsilon_i$.
    \begin{proof}
        By virtue of \eqref{eq: assumption thm 1}, we may express $\bm{U}_{st,\mu}$ as: 
        \begin{equation*}
            \bm{U}_{1,\hdots,\mu} = \bm{\Phi}_{1,\widehat{1}}\bm{R}_{\widehat{1},2,\hdots,\mu} + \bm{E}_{1,\hdots,\mu} = \bm{\Phi}_{1,\widehat{1}}\left(\bm{\Phi}_{\widehat{1},2,\widehat{2}}\left(\cdots \left(\bm{\Phi}_{\widehat{d},t,\widehat{t}}\bm{R}_{\widehat{t},\mu}+\bm{E}_{\widehat{d},t,\mu}\right) \cdots \right)+ \bm{E}_{\widehat{1},2,\hdots,\mu}\right) + \bm{E}_{1,\hdots,\mu}.
        \end{equation*}
        On the other hand, the approximated snapshots 
        \begin{equation*}
            \bm{\Phi}_{st,\widehat{t}}\bm{\Phi}_{\widehat{t},st}\bm{U}_{st,\mu} = \widehat{\bm{U}}_{st,\mu}
        \end{equation*}
        can be written according to 
        \begin{equation*}
            \widehat{\bm{U}}_{1,\hdots,\mu} = \bm{\Phi}_{1,\widehat{1}}\widehat{\bm{R}}_{\widehat{1},2,\hdots,\mu},
        \end{equation*}
        where $\widehat{\bm{R}}_{\widehat{1},2,\hdots,\mu}$ is the \ac{ttsvd} approximation of $\bm{R}_{\widehat{1},2,\hdots,\mu}$, with ranks $(r_2,\hdots,r_t)$. By performing all the \ac{ttsvd} iterations we derive: 
        \begin{equation*}
            \widehat{\bm{U}}_{1,\hdots,\mu} = \bm{\Phi}_{1,\widehat{1}}\widehat{\bm{R}}_{\widehat{1},2,\hdots,\mu} = \bm{\Phi}_{1,\widehat{1}}\bm{\Phi}_{\widehat{1},2,\widehat{2}} \cdots \bm{\Phi}_{\widehat{d},t,\widehat{t}}\bm{R}_{\widehat{t},\mu}.
        \end{equation*}
        Consequently, we can write
        \begin{equation*}
            \bm{U}_{1,2:\mu} - \widehat{\bm{U}}_{1,2:\mu} = \bm{\Phi}_{1,\widehat{1}}\bm{R}_{\widehat{1},2:\mu} + \bm{E}_{1,2:\mu} - \bm{\Phi}_{1,\widehat{1}}\widehat{\bm{R}}_{\widehat{1},2:\mu}.
        \end{equation*}
        By virtue of \eqref{eq: assumption thm 1} and of the orthogonality of every matrix $\bm{\Phi}_{\widehat{i-1}i,\widehat{i}}$, we have: 
        \begin{equation*}
            \begin{split}
                \| \bm{U}_{1,2:\mu} - \widehat{\bm{U}}_{1,2:\mu} \|_F^2 
                &\leq \varepsilon_1^2 \|\bm{T}_{1,2:\mu}\|_F^2 + \|\bm{R}_{\widehat{1},2:\mu} - \widehat{\bm{R}}_{\widehat{1},2:\mu}\|_F^2 = \varepsilon_1^2 \|\bm{T}_{1,2:\mu}\|_F^2 + \|\bm{\Phi}_{\widehat{1},2,\widehat{2}}\bm{R}_{\widehat{2},3:\mu} + \bm{E}_{\widehat{1},2,3:\mu} - \bm{\Phi}_{\widehat{1},2,\widehat{2}}\widehat{\bm{R}}_{\widehat{2},3:\mu}\|_F^2
                \\ &\leq 
                \varepsilon_1^2 \|\bm{T}_{1,2:\mu}\|_F^2 + \varepsilon_2^2 \|\bm{T}_{\widehat{1}2,3:\mu}\|_F^2 + \|\bm{R}_{\widehat{2},3:\mu} - \widehat{\bm{R}}_{\widehat{2},3:\mu}\|_F^2 \leq \hdots \leq \sum\limits_{i=1}^{d+1}\varepsilon_i^2 \|\bm{T}_{\widehat{i-1}  i, i+1 : \mu}\|^2_F.
            \end{split}
        \end{equation*}
        Since $ \bm{T}_{\widehat{i-1}  i, i+1 : \mu} = \bm{R}_{\widehat{i-1},i:\mu} $ is the output of an \ac{rsvd}, we have 
        \[
            \|\bm{T}_{\widehat{i-1}  i, i+1 : \mu}\|^2_F \leq \|\bm{T}_{\widehat{i-2} \,  i-1, i : \mu}\|^2_F \leq \cdots \leq \|\bm{T}_{1,2:\mu}  \|^2_F = \| \bm{U}_{1,2:\mu} \|^2_F.
        \]
        Then the result follows.
    \end{proof}
\end{theorem}
Thm. \ref{thm: ttsvd accuracy} states that, under an appropriate assumption of reducibility of the snapshots, running the \ac{ttsvd} with a modified tolerance $\widetilde{\varepsilon} = \varepsilon/\sqrt{d+1}$ yields a basis characterized by the same accuracy as that achieved by the \ac{tpod} basis. We refer to \cite{oseledets2011tensor} -- where a similar accuracy result is stated under slightly different assumptions -- for more details. We also note that one may also run Alg. \ref{alg: ttsvd} on a snapshots tensor with permuted axes, thus obtaining a different reduced subspace. Although an axis ordering exists that minimizes the dimension of the output subspace for a fixed $\varepsilon$, in this work we adopt a consistent heuristic: the first $d$ axes are the spatial ones, ordered according to a standard Cartesian coordinate system, while the last two axes correspond to the temporal and parametric dimensions. \\

In the context of model order reduction for \acp{pde}, it is desirable for the \ac{rb} to satisfy an orthogonality condition in the energy norm of the \ac{fe} spaces. This condition ensures that the \ac{rb} spans a subspace included in the \ac{fom} space, which is important for the well-posedness of the resulting \ac{rom}. In certain cases -- such as \acp{rom} for saddle-point problems -- this condition alone may still not suffice to guarantee well-posedness (see, for example, \cite{ballarin2015supremizer}). However, such scenarios are not addressed in this work.

Our aim is to modify Alg. \ref{alg: ttsvd} so that the resulting \ac{tt} decomposition is $\bm{X}_{st,st}$-orthogonal. In the following theorem, we state the accuracy of the \ac{ttsvd} algorithm under this orthogonality constraint.
\begin{theorem}
    \label{thm: ttsvd accuracy X}
    Let $\bm{X}_{s,s}$ be a matrix representing a norm on a finite-dimensional subspace of a Hilbert space. Let $\widetilde{\bm{U}}_{s,t\mu} = \bm{H}_{s,s}\bm{U}_{s,t\mu}$, with $\bm{H}_{s,s}$ being the upper-triangular Cholesky factor of $\bm{X}_{s,s}$. Let $(\widetilde{\bm{T}}_{s,t\mu},\widetilde{\bm{T}}_{\widehat{s}t,\mu})$ be the spatial and temporal unfoldings of $\widetilde{\bm{U}}_{s,t,\mu}$. Assuming these admit \ac{tt} cores ($\widetilde{\bm{\Phi}}_{s,\widehat{s}}$, $\widetilde{\bm{\Phi}}_{\widehat{s},t,\widehat{t}}$) with relative errors ($ \widetilde{\varepsilon}_s $, $ \widetilde{\varepsilon}_t$) for given ranks ($r_s$, $r_t$) (see \eqref{eq: assumption thm 1}), the following inequality holds:
    \begin{equation}
        \label{eq: ttsvd no split accuracy X}
        \sum\limits_{j=1}^{N_{\mu}} \|\left(\bm{U}_{st,\mu} - \bm{\Phi}_{st,\widehat{t}}\bm{\Phi}_{\widehat{t},st}\bm{X}_{st,st}\bm{U}_{st,\mu}\right)\left[:,j\right]\|_{\bm{X}_{st,st}}^2  \leq (\widetilde{\varepsilon}_s^2 + \widetilde{\varepsilon}_t^2 ) \|\bm{U}_{st,\mu}\|^2_{\bm{X}_{st,st}},
    \end{equation} 
    where 
    \begin{equation}
        \label{eq: ttsvd no split quantities}
        \widetilde{\bm{\Phi}}_{s,\widehat{s}} = \text{\texttt{RSVD}}\left(\widetilde{\bm{T}}_{s,t\mu},\widetilde{\varepsilon}_s\right) , \quad 
        \bm{\Phi}_{s,\widehat{s}} = \bm{H}_{s,s}^{-1}\widetilde{\bm{\Phi}}_{s,\widehat{s}}, \quad 
        \widetilde{\bm{\Phi}}_{\widehat{s}t,\widehat{t}} = \text{\texttt{RSVD}}\left(\widetilde{\bm{T}}_{\widehat{s}t,\mu},\widetilde{\varepsilon}_t\right) , \quad
        \bm{\Phi}_{\widehat{s},t,\widehat{t}} = \widetilde{\bm{\Phi}}_{\widehat{s},t,\widehat{t}}, \quad 
        \bm{\Phi}_{st,\widehat{t}} = \bm{\Phi}_{s,\widehat{s}}\bm{\Phi}_{\widehat{s},t,\widehat{t}}.
    \end{equation} 
    Moreover, if $ \widetilde{\bm{U}}_{s,t,\mu} $ admits a ``split-axes'' representation, and its unfoldings $\widetilde{\bm{T}}_{\widehat{i-1}  i, i+1 : \mu}$ admit \ac{tt} cores $\widetilde {\bm{\Phi}}_{\widehat{i-1}  i,\widehat{i}} $ with relative errors $\widetilde{\varepsilon}_i$ for given ranks $r_i$,
    the following inequality holds: 
    \begin{equation}
        \label{eq: ttsvd accuracy X}
        \sum\limits_{j=1}^{N_{\mu}} \|\left(\bm{U}_{st,\mu} - \bm{\Phi}_{st,\widehat{t}}\bm{\Phi}_{\widehat{t},st}\bm{X}_{st,st}\bm{U}_{st,\mu}\right)\left[:,j\right]\|_{\bm{X}_{st,st}}^2 \leq \sum\limits_{i=1}^{d+1}\varepsilon_i^2 \| \widetilde{\bm{T}}_{\widehat{i-1}  i, i+1 : \mu} \|^2_F \leq \widetilde{\varepsilon}^2 (d+1) \|\bm{U}_{st,\mu}\|^2_{\bm{X}_{st,st}},
    \end{equation} 
    where $\widetilde{\varepsilon} = \sup_i \widetilde{\varepsilon}_i$. 
\end{theorem}
\begin{proof}
    We first consider the tuple $\left(\widetilde{\bm{U}}_{s,t\mu}, \widetilde{\bm{\Phi}}_{s,\widehat{s}}\right)$ and apply Thm. \ref{thm: ttsvd accuracy}:
    \begin{equation}
        \|\widetilde{\bm{U}}_{st,\mu} - \widetilde{\bm{\Phi}}_{st,\widehat{t}}\widetilde{\bm{\Phi}}_{\widehat{t},st}\widetilde{\bm{U}}_{st,\mu} \|^2_F \leq  (\widetilde{\varepsilon}_s^2 + \widetilde{\varepsilon}_t^2 )\|\widetilde{\bm{U}}_{st,\mu}\|^2_F, 
    \end{equation} 
    where 
    \begin{equation*}
        \widetilde{\bm{\Phi}}_{s,t,\widehat{t}} = \widetilde{\bm{\Phi}}_{s,\widehat{s}}\bm{\Phi}_{\widehat{s},t,\widehat{t}}.
    \end{equation*}
    Then, recalling the definition of $\bm{X}_{st,st}$ in Subsection \ref{subs:full order model} (we may disregard the multiplicative constant $\delta$, representing the time step), we have
    \begin{equation}
        \label{eq: frobenius - X relationship2}
        \begin{split}
            \sum\limits_{j=1}^{N_{\mu}} \|\left(\bm{U}_{st,\mu} - \bm{\Phi}_{st,\widehat{t}}\bm{\Phi}_{\widehat{t},st}\bm{X}_{st,st}\bm{U}_{st,\mu}\right)\left[:,j\right]\|_{\bm{X}_{st,st}}^2 &= \sum\limits_{j=1}^{N_{\mu}} \|\left(\bm{U}_{st,\mu} - \bm{\Phi}_{st,\widehat{t}}\widetilde{\bm{\Phi}}_{st,\widehat{t}}\widetilde{\bm{U}}_{st,\mu}\right)\left[:,j\right]\|_{\bm{X}_{s,s}}^2 
            \\ &= \|\widetilde{\bm{U}}_{st,\mu} - \widetilde{\bm{\Phi}}_{st,\widehat{t}}\widetilde{\bm{\Phi}}_{\widehat{t},st}\widetilde{\bm{U}}_{st,\mu} \|^2_F.
        \end{split}
    \end{equation}
    Eq. \eqref{eq: ttsvd no split accuracy X} follows from $ \|\widetilde{\bm{U}}_{st,\mu}\|^2_F = \|\bm{U}_{st,\mu}\|^2_{\bm{X}_{st,st}} $. Finally, if $ \widetilde{\bm{U}}_{s,t,\mu} $ admits a ``split-axes'' representation, \eqref{eq: ttsvd accuracy X} can be obtained by using the same procedure as in Thm. \ref{thm: ttsvd accuracy}.
\end{proof}
Therefore, we first pre-multiply the snapshots tensor by the Cholesky factor $\bm{H}_{s,s}$. Secondly, we post-multiply the \ac{tt} basis we extract from the resulting tensor by $\bm{H}_{s,s}^{-1}$. This subspace, in the energy norm, has the same approximation capability as the one computed via standard \ac{ttsvd} in the Euclidean norm. However, it is important to note that both the computation of $ \bm{H}_{s,s} $ and $ \widetilde{\bm{U}}_{s,t\mu} $ involve matrices of size $ N_s \times N_s $, which entails considerable costs. To address this issue, we develop a sequential algorithm that exploits the ``split-axes'' principle. We first present the idea in a simplified case where $ \bm{X}_{s,s} $ is a rank-one matrix, such as when $\bm{X}_{s,s} = \bm{M}_{s,s}$, as the mass matrix is indeed rank-1. We then extend the algorithm to handle a generic $ \bm{X}_{s,s} $ of rank-$ K $.

\subsection{Basis construction: the case of a rank-1 norm matrix}
\label{subs: basis construction X-1}

When $ \bm{X}_{s,s} $ is a rank-1 matrix, the Cholesky decomposition of a Kronecker product matrix applies:
\begin{equation}
    \label{eq: cholesky kronecker}
    \texttt{Cholesky}\left(\bm{X}_{1,1} \otimes \cdots \otimes \bm{X}_{d,d}\right) = \bm{H}_{1,1} \otimes \cdots \otimes \bm{H}_{d,d}.
\end{equation}
Now, let us recall the diadic representation \cite{oseledets2011tensor} of $\widetilde{\bm{U}}_{s,t\mu}$. Let us consider the first unfolding
\begin{equation*}
    \widetilde{\bm{T}}_{1,\hdots,\mu} = \widetilde{\bm{U}}_{1,\hdots,\mu}.
\end{equation*}
By applying a \texttt{RSVD} with $\varepsilon = 0$, we can write 
\begin{equation*}
    \widetilde{\bm{T}}_{1,2:\mu} = \widetilde{\bm{\Phi}}_{1,\sigma_1} \widetilde{\bm{R}}_{\sigma_1,2:\mu},
\end{equation*}
where the subscript $\sigma_1$ refers to $N_{\sigma_1} = \min\{N_1,N_{2:\mu}\}$, i.e. the rank of $\widetilde{\bm{T}}_{1,2:\mu}$. By applying a \ac{rsvd} on the second unfolding 
\begin{equation*}
    \widetilde{\bm{T}}_{\sigma_1  2,3:\mu} = \widetilde{\bm{R}}_{\sigma_1 2, 3:\mu},
\end{equation*}
we can write 
\begin{equation*}
    \widetilde{\bm{T}}_{\sigma_1  2,3:\mu} = \widetilde{\bm{\Phi}}_{\sigma_1  2,\sigma_2} \widetilde{\bm{R}}_{\sigma_2,3:\mu},
\end{equation*}
where $\sigma_2$ refers to $N_{\sigma_2} = \min\{r_1N_2,N_{3:\mu}\}$. This procedure continues iteratively, until we get
\begin{equation}
    \label{eq: diadic representation}
    \widetilde{\bm{U}}_{s,t\mu} = \widetilde{\bm{\Phi}}_{1,\sigma_1} \widetilde{\bm{\Phi}}_{\sigma_1,2,\sigma_2} \cdots \widetilde{\bm{\Phi}}_{\sigma_{d-1},d,\sigma_d} \widetilde{\bm{R}}_{\sigma_d,t\mu}.
\end{equation}
Now let us consider a different diadic decomposition for $\widetilde{\bm{U}}_{s,t\mu}$. Notice that
\begin{equation*}
    \widetilde{\bm{U}}_{s,t\mu} = \left(\bm{H}_{1,1} \otimes \cdots \otimes \bm{H}_{d,d}\right) \bm{U}_{s,t\mu} 
    = \bm{H}_{d, d} \odot_d \left(\cdots \ \left( \bm{H}_{1, 1} \odot_1 \bm{U}_{ 1,\hdots,\mu} \right)\right).
\end{equation*}
We introduce the rescaled unfolding 
\begin{equation*}
    \widebar{\bm{T}}_{1,2:\mu} = \bm{H}_{1,1} \bm{U}_{1,2:\mu},
\end{equation*}
which admits a diadic representation 
\begin{equation*}
    \widebar{\bm{T}}_{1,2:\mu} = \widebar{\bm{\Phi}}_{1,\sigma_1}\widebar{\bm{R}}_{\sigma_1,2:\mu}.
\end{equation*}
Then, we rescale the unfolding $\widebar{\bm{R}}_{\sigma_1,2:\mu}$ by $\bm{H}_{2,2}$, i.e. 
\begin{equation*}
    \widebar{\bm{T}}_{\sigma_1,2,3:\mu} = \bm{H}_{2,2} \odot_2 \widebar{\bm{R}}_{\sigma_1,2,3:\mu},
\end{equation*}
which admits its own diadic format, and so on. By iteration, we can derive a similar expression to \eqref{eq: diadic representation}:
\begin{equation}
    \label{eq: X diadic representation}
    \widetilde{\bm{U}}_{s,t\mu} 
    = \widebar{\bm{\Phi}}_{1,\sigma_1} \widebar{\bm{\Phi}}_{\sigma_1,2,\sigma_2} \cdots \widebar{\bm{\Phi}}_{\sigma_{d-1},d,\sigma_d} \widetilde{\bm{R}}_{\sigma_d,t\mu}.
\end{equation}
The representations \eqref{eq: diadic representation} and \eqref{eq: X diadic representation} admit a spectrally equivalent final unfolding $\widetilde{\bm{R}}_{\sigma_d t,\mu}$. For this reason, we have employed the same variable name in both \eqref{eq: diadic representation} and \eqref{eq: X diadic representation}. Specifically, one can empirically verify that the two unfoldings are equal in absolute value, up to machine precision. This implies that the temporal subspace associated with \eqref{eq: diadic representation} is equivalent to the one in \eqref{eq: X diadic representation}. This outcome is expected, since pre-multiplying by $\bm{X}_{s,s}$ does not affect the temporal evolution of the snapshots. On the other hand, the spatial cores in \eqref{eq: diadic representation} differ from those in \eqref{eq: X diadic representation}. \\ 
By virtue of the hierarchical property \eqref{eq: hierarchical bases}, it follows that the \ac{tt} decompositions \eqref{eq: diadic representation}-\eqref{eq: X diadic representation} span an equivalent subspace. From a computational perspective, however, running the latter is significantly more efficient, as it avoids the costly Cholesky factorization of $\bm{X}_{s,s}$ and the rescaling of the snapshots by $\bm{H}_{s,s}$. Instead, for every iteration $i$ of the \ac{ttsvd} \emph{for} loop, it merely computes $\bm{H}_{i,i}$ and rescales the respective unfolding. The procedure is summarized in Alg. \ref{alg: ttsvd orthogonality}. \\
\begin{algorithm}[t]
    \caption{\texttt{$X^1$-TT-SVD}: Given the snapshots tensor in the ``split-axes'' format $ \bm{U}_{1,\hdots,\mu}$, the prescribed accuracy $ \varepsilon $, and the 1-d norm matrices $\bm{X}_{1,1}, \hdots, \bm{X}_{d,d}$, build the $\bm{X}_{st,st}$-orthogonal TT-cores $\bm{\Phi}_{\widehat{0},1,\widehat{1}},\hdots,\bm{\Phi}_{\widehat{d-1},d,\widehat{d}},\bm{\Phi}_{\widehat{d},t,\widehat{t}}$.} 
    \begin{algorithmic}[1]
        \Function{\texttt{$X^1$-TT-SVD}}{$ \bm{U}_{1,\hdots,\mu}, \bm{X}_{i,i}, \varepsilon $}
        \State Set $\bm{T}_{1,\hdots,\mu} = \bm{U}_{1,\hdots,\mu}$
        \For{i = $1, \hdots , d$}
        \State $i$th Cholesky factorization: $\bm{H}_{i,i} = \texttt{Cholesky}(\bm{X}_{i,i})$  \Comment{$\mathcal{O}\left(N_i b_i^2\right)$} 
        \State $i$th spatial rescaling: $\widetilde{\bm{T}}_{\widehat{i-1}, i, i+1:\mu} = \bm{H}_{i,i} \odot_2 \bm{T}_{\widehat{i-1}, i, i+1:\mu}$ \Comment{$\mathcal{O}\left(r_{i-1} N_{z_i} N_{i+1:\mu} \right)$}
        \State $i$th spatial reduction: $\widetilde{\bm{\Phi}}_{\widehat{i-1}  i,\widehat{i}} \widetilde{\bm{R}}_{\widehat{i},i+1:\mu} = \texttt{RSVD}(\widetilde{\bm{T}}_{\widehat{i-1}  i, i+1:\mu})$ \Comment{$\mathcal{O}\left( r_{i-1}N_{i:\mu} \log{ (\min \{ N_{i+1:\mu}, r_{i-1}N_i \}) }\right)$}
        \State $i$th spatial inverse rescaling: $\bm{\Phi}_{\widehat{i-1}, i, \widehat{i}} = \bm{H}_{i,i}^{-1} \odot_2 \widetilde{\bm{\Phi}}_{\widehat{i-1}, i, \widehat{i}}$ \Comment{$\mathcal{O}\left(r_{i-1} r_i N_i^2 \right)$}
        \State Update unfolding matrix: $\bm{T}_{\widehat{i} ,i+1:\mu} = \bm{R}_{\widehat{i} ,i+1:\mu}$
        \EndFor
        \State Compute $\bm{\Phi}_{\widehat{d}, t, \widehat{t}}$ as in Alg. \ref{alg: ttsvd} \Comment{$\mathcal{O}\left(r_d N_t N_{\mu} \log{N_{\mu}} \right)$}
        \State \Return  $\bm{\Phi}_{\widehat{0},1,\widehat{1}}, \hdots, \bm{\Phi}_{\widehat{d-1},d,\widehat{d}}, \bm{\Phi}_{\widehat{d},t,\widehat{t}}$
        \EndFunction
    \end{algorithmic}
    \label{alg: ttsvd orthogonality}
\end{algorithm} 
Here, we use $N_{z_i}$ and $b_i$ to indicate the nonzero entries and the semi-bandwidth of the sparse matrix $\bm{X}_{i,i}$. As already observed, the operations concerning the $\bm{X}_{s,s}$-orthogonality are now independent of the global size of the problem, as they only scale as the 1-d sizes $N_i$. The only operation that still depends on $N_{st}$ is the \ac{rsvd}. We briefly discuss the bound on the computational cost of Alg. \ref{alg: ttsvd orthogonality}, and we qualitatively compare this bound with the cost of the \ac{tpod}. To facilitate this comparison, we assume for simplicity that:
\begin{itemize}
    \item The snapshots tensor is a ``perfect cube'': $N_1 = \hdots = N_d = N_t = N_{\mu} = M$.
    \item There exists a ``bounding rank'' $r \ll M$ such that $r_i < r$ for every $i$. Our numerical experiments indicate that such a rank exists, even though it might be that $r = \mathcal{O}(M)$ in applications characterized by poor reducibility, as defined in \cite{quarteroni2015reduced}.
\end{itemize}
Under the ``perfect cube'' assumption, we can express the semi-bandwidth in Alg. \ref{alg: tpod} as 
\begin{equation*}
    b = \mathcal{O}(M^{d-1}).
\end{equation*}
This result is intuitive: since $\bm{X}_{s,s}$ is computed by following the usual \ac{fe} integration-assembly routines, its sparsity pattern is determined by the proximity of neighboring \ac{fe} cells. Under the ``perfect cube'' assumption, the maximum distance between two neighboring cells scales with $M^{d-1}$. Consequently, the procedure in Alg. \ref{alg: tpod} has a cost of
\begin{equation}
    \label{eq: cost tpod}
    \mathcal{O}\left( M^{3d-2} + (pd)^d M^{d+2} + M^{d+2}\log{M} + r_s M^{2d} + r_s M^{d+2} + r_s M^2 \log{M}\right) = \mathcal{O}\left( M^{3d-2} + M^{d+2}\log{M} \right) \ \forall \ d \geq 1.
\end{equation}
On the other hand, under the same assumptions, we can bound the complexity of Alg. \ref{alg: ttsvd orthogonality} as follows: 
\begin{equation*}
    \text{cost}(\text{\texttt{$X^1$-TT-SVD}}) \leq (d+1) \cdot \text{cost} (\text{$1$st iteration of \texttt{$X^1$-TT-SVD}}).
\end{equation*} 
This result is a consequence of the decreasing cost of the \ac{ttsvd} iterations, which is due to the fact that $N_{i:\mu}$ itself decreases. Based on this observation, and recalling that $r_0 = 1$, the complexity of Alg. \ref{alg: ttsvd orthogonality} can be bounded by:
\begin{equation}
    \label{eq: cost ttsvd}
    (d+1) \cdot \mathcal{O} \left(M + M^{d+2} + M^{d+2}\log{M} + r M^2\right) = \mathcal{O}\left( (d+1) M^{d+2}\log{M} \right) \ \forall \ d \geq 1.
\end{equation}
Comparing the estimates \eqref{eq: cost tpod} and \eqref{eq: cost ttsvd}, we immediately observe that \ac{ttsvd} reduces the cost of every operation in \ac{tpod}, except for the compression step. Another important observation is:
\begin{equation}
    \label{eq: cost result}
    \text{cost}(\text{\texttt{$X^1$-TT-SVD}}) < \text{cost}(\text{\texttt{TPOD}}) \ \forall \ d > 2.
\end{equation}
The reason \ac{ttsvd} outperforms \ac{tpod} for $d>2$ is that the cost of the latter scales with the Cholesky decomposition of $\bm{X}_{s,s}$. On the other hand, the former essentially performs all operations involving the norm matrices $\bm{X}_{i,i}$ at negligible cost, as the dominant cost always lies in compressing the snapshots. An obvious consequence is that, when we seek an $\ell^2$-orthogonal basis, the cost of \ac{ttsvd} matches that of \ac{tpod}. In this case, the only viable way to reduce the complexity of the \ac{tt} decomposition is to employ a more efficient algorithm than \ac{rsvd} for the compression, for e.g. \ac{ttcross}.

\subsection{Basis construction: the case of a rank-$K$ norm matrix}
\label{subs: basis construction X-K}
Now, let us consider the more involved case where $\bm{X}_{s,s}$ admits the form 
\begin{equation}
    \label{eq: generic form X}
    \bm{X}_{s,s} = \sum\limits_{k=1}^K \overset{d}{\underset{i=1}{\otimes}} \bm{Y}_{i,i}^k,
\end{equation}
with $K$ representing the rank of the norm matrix. This situation arises in our \ac{fom}, where the norm matrix is given by a stiffness matrix (see \eqref{eq: H1 norm matrix}). In this scenario, we cannot exploit \eqref{eq: cholesky kronecker} as we do when deriving Alg. \ref{alg: ttsvd orthogonality}. To obtain an $\bm{X}_{s,s}$-orthogonal basis in the case of a rank-$K$ norm matrix, we propose the following procedure: 
\begin{enumerate}
    \item Compute the temporary \ac{tt} decomposition:
    \begin{equation}
        \label{eq: new X cores}
        \widecheck{\bm{\Phi}}_{\widehat{0},1,\widehat{1}},\hdots,\widecheck{\bm{\Phi}}_{\widehat{d-1},d,\widehat{d}},\widecheck{\bm{\Phi}}_{\widehat{d},t,\widehat{t}}.
    \end{equation} 
    \item Perform an $\bm{X}_{s,s}$-orthogonalization procedure on the spatial cores of \eqref{eq: new X cores}.
\end{enumerate}
Let us first detail the orthogonalization strategy. An efficient algorithm that only runs operations on the cores \eqref{eq: new X cores} can be derived by writing down the orthogonality condition:
\begin{equation*}
    \bm{I}_{\widehat{s},\widehat{s}} = \widecheck{\bm{\Phi}}_{\widehat{s},s}\bm{X}_{s,s}\widecheck{\bm{\Phi}}_{s,\widehat{s}}.
\end{equation*}
For simplicity, we consider the case $d=2$. Exploiting the mixed-product property of the Kronecker product, we have
\begin{equation}
    \label{eq: ttsvd orthogonality}
    \begin{split}
        \widecheck{\bm{\Phi}}_{\widehat{2},12} \bm{X}_{12,12} \widecheck{\bm{\Phi}}_{12,\widehat{2}} &= \sum_{\alpha_1,\alpha_2,\beta_1,\beta_2,k} \left(\widecheck{\bm{\Phi}}_{\widehat{0},1,\widehat{1}} \left[1,:,\alpha_1\right]^T \bm{Y}_{1,1}^k \widecheck{\bm{\Phi}}_{\widehat{0},1,\widehat{1}} \left[1,:,\beta_1\right]\right) \left(\widecheck{\bm{\Phi}}_{\widehat{1},2,\widehat{2}} \left[\alpha_1,:,\alpha_2\right]^T \bm{Y}_{2,2}^k \widecheck{\bm{\Phi}}_{\widehat{1},2,\widehat{2}} \left[\beta_1,:,\beta_2\right]\right)
        \\ &= \sum_{\alpha_2,\beta_2}\sum_{\alpha_1,\beta_1} \widecheck{\bm{\Phi}}_{\widehat{1},2,\widehat{2}} \left[\alpha_1,:,\alpha_2\right]^T \left(\sum_k \widehat{\bm{Y}}_{\widehat{1},\widehat{1}}^k\left[\alpha_1,\beta_1\right] \bm{Y}_{2,2}^k\right)  \widecheck{\bm{\Phi}}_{\widehat{1},2,\widehat{2}} \left[\beta_1,:,\beta_2\right]
        \\ &= \widecheck{\bm{\Phi}}_{\widehat{2},\widehat{1}2} \widehat{\bm{X}}_{\widehat{1}2,\widehat{1}2} \widecheck{\bm{\Phi}}_{\widehat{1}2,\widehat{2}}.
    \end{split}
\end{equation}
The matrices $\widehat{\bm{Y}}_{\widehat{1},\widehat{1}}^k$ and $\widehat{\bm{X}}_{\widehat{1}2,\widehat{1} 2}$ are given by 
\begin{align*}
    \widehat{\bm{Y}}_{\widehat{1},\widehat{1}}^k\left[\alpha_1,\beta_1\right] = \widecheck{\bm{\Phi}}_{\widehat{0},1,\widehat{1}} \left[1,:,\alpha_1\right]^T \bm{Y}_{1,1}^k \widecheck{\bm{\Phi}}_{\widehat{0},1,\widehat{1}}\left[1,:,\beta_1\right], \qquad 
    \widehat{\bm{X}}_{\widehat{1}2,\widehat{1}2} = \sum\limits_{k=1}^K \bm{Y}^k_{2,2} \otimes \widehat{\bm{Y}}_{\widehat{1},\widehat{1}}^k.
\end{align*}

Eq. \eqref{eq: ttsvd orthogonality} implies that requiring $\bm{X}_{12,12}$-orthogonality of $\widecheck{\bm{\Phi}}_{12,\widehat{2}}$ is equivalent to enforcing the $\widehat{\bm{X}}_{\widehat{1}2,\widehat{1} 2}$-orthogonality of $\widecheck{\bm{\Phi}}_{\widehat{1} 2,\widehat{2}}$. Note that the modified norm matrix $\widehat{\bm{X}}_{\widehat{1}2,\widehat{1}2}$ is sparse, with a number of nonzero elements equal to $r_1^2 N_{z_2}$. Generalizing to the case of an arbitrary $d \geq 2$, we require $\widecheck{\bm{\Phi}}_{\widehat{d-1}  d,\widehat{d}}$ to be orthogonal with respect to
\begin{equation}
    \label{eq: ttsvd final weight}
    \widehat{\bm{X}}_{\widehat{d-1}  d,\widehat{d-1}  d} = \sum\limits_{k=1}^K \bm{Y}^k_{d,d} \otimes \widehat{\bm{Y}}_{\widehat{d-1},\widehat{d-1}}^k,
\end{equation}
where
\begin{align}
    \label{eq: ttsvd orthogonality recursive}
    \widehat{\bm{Y}}^k_{\widehat{i},\widehat{i}}\left[\alpha_i,\beta_i\right] = \sum\limits_{\alpha_{i-1},\beta_{i-1}}\widehat{\bm{Y}}^k_{\widehat{i-1},\widehat{i-1}}\left[\alpha_{i-1},\beta_{i-1}\right] \widecheck{\bm{\Phi}}_{\widehat{i-1},i,\widehat{i}} \left[\alpha_{i-1},:,\alpha_i\right]^T \bm{Y}_{i,i}^k \widecheck{\bm{\Phi}}_{\widehat{i-1},i,\widehat{i}}\left[\beta_{i-1},:,\beta_i\right]. 
\end{align}
Since the relationship \eqref{eq: ttsvd orthogonality recursive} is recursive, the final matrix \eqref{eq: ttsvd final weight} can be built iteratively. In terms of computational cost, under the usual ``perfect square'' and ``bounded rank'' assumptions, the cost of the orthogonalization procedure is given by 
\begin{equation*}
    \mathcal{O}\left( r^3 M + r^4 (d-1) K M \right).
\end{equation*}
Recalling \eqref{eq: cost ttsvd}, we notice that the cost of the orthogonalization scheme is negligible compared to the cost of computing the \ac{tt} cores, for every $d$.

We now address the computation of the cores \eqref{eq: new X cores}. An appropriate method consists in first selecting a rank-$1$ norm matrix $\widecheck{\bm{X}}_{s,s}$ that is ``similar'' to $\bm{X}_{s,s}$, for e.g. one representing an equivalent norm to the one represented by $\bm{X}_{s,s}$. Then, we may compute \eqref{eq: new X cores} by running Alg. \ref{alg: ttsvd orthogonality} on the pair $(\bm{U}_{1:\mu},\widecheck{\bm{X}}_{s,s})$. For instance, we may consider
\begin{equation}
    \label{eq: crossnorm X}
    \widecheck{\bm{X}}_{s,s} \ \text{ such that } \ \|\cdot\|_{\widecheck{\bm{X}}_{s,s}} = \max\limits_{k=1,\hdots,K}\{\|\cdot\|_{\overset{d}{\underset{i=1}{\otimes}} \bm{Y}_{i,i}^k}\},
\end{equation}
which represents a \textit{reasonable crossnorm} \cite{hackbusch2012tensor} on $\mathcal{V}^0_h$. Although the error bound \eqref{eq: ttsvd accuracy X} is no longer guaranteed in this case, we show in Sect.~\ref{sec: results} that the procedure yields a correct error decay with respect to the tolerances. The whole method is summarized in Alg. \ref{alg: ttsvd orthogonality recursive}. We note that this algorithm slightly improves the method described qualitatively so far. Rather than first computing \eqref{eq: new X cores} and then applying the orthogonalization strategy, we may instead run a single, more efficient \emph{for} loop in which the final \ac{tt} decomposition is directly computed.
 
\begin{algorithm}[t]
    \caption{\texttt{$X^K$-TT-SVD}: Given the snapshots tensor in the ``split-axes'' format $ \bm{U}_{1,\hdots,\mu}$, the prescribed accuracy $ \varepsilon $, and the 1-d norm matrices $\bm{Y}^k_{1,1}, \hdots, \bm{Y}^k_{d,d}$ for every $k$, build the $\bm{X}_{st,st}$-orthogonal TT-cores $\bm{\Phi}_{\widehat{0},1,\widehat{1}},\hdots,\bm{\Phi}_{\widehat{d-1},d,\widehat{d}},\bm{\Phi}_{\widehat{d},t,\widehat{t}}$ .} 
    \begin{algorithmic}[1]
        \Function{\texttt{$X^K$-TT-SVD}}{$ \bm{U}_{1,\hdots,\mu}, \bm{Y}^k_{i,i}, \varepsilon $}
        \State Derive a rank-1 norm matrix $\widecheck{\bm{X}}_{s,s}$ from $\bm{Y}^k_{i,i}$
        \For{i = $1, \hdots , d$}
        \If{$i < d$}
        \State Compute $\bm{\Phi}_{\widehat{i-1},i,\widehat{i}}$ as in \texttt{$X^1$-TT-SVD}$(\bm{U}_{1,\hdots,\mu}, \widecheck{\bm{X}}_{s,s}, \varepsilon)$ \Comment{$\mathcal{O}\left( r_{i-1}N_{i:\mu} \log{ (\min \{ N_{i+1:\mu}, r_{i-1}N_i \}) }\right)$}
        \For{k = $1, \hdots , K$}
        \State Update weight matrix $\widehat{\bm{Y}}^k_{\widehat{i},\widehat{i}}$ as in \eqref{eq: ttsvd orthogonality recursive} \Comment{$\mathcal{O}\left(r_{i-1}r_i N_{z_i} + (r_{i-1}r_i)^2N_i + (r_{i-1} r_i)^2\right)$}
        \EndFor
        \ElsIf{$i = d$} 
        \State Compute final weight $\widehat{\bm{X}}_{\widehat{d-1}  d,\widehat{d-1}  d}$ as in \eqref{eq: ttsvd final weight} \Comment{$\mathcal{O}\left( K r_{d-1}^2 N_{z_d} \right)$}
        \State Cholesky factorization: $\widehat{\bm{H}}_{\widehat{d-1}  d,\widehat{d-1}  d} = \texttt{Cholesky}\left(\widehat{\bm{X}}_{\widehat{d-1}  d,\widehat{d-1}  d}\right)$  \Comment{$\mathcal{O}\left(r_{d-1}^2 N_d \right)$}
        \State $d$th rank reduction: $\widetilde{\bm{\Phi}}_{\widehat{d-1}d,\widehat{d}} \
        ,\bm{R}_{\widehat{d},\widehat{d}} = \texttt{RSVD}(\widehat{\bm{H}}_{\widehat{d-1}d,\widehat{d-1} d}\bm{\Phi}_{\widehat{d-1} d,\widehat{d}},\varepsilon)$ \Comment{$\mathcal{O}\left( r_{d-1}r_d^2 N_d  \right)$} 
        \State Inverse rescaling: $\bm{\Phi}_{\widehat{d-1}  d,\widehat{d}} = \widehat{\bm{H}}_{\widehat{d-1}  d,\widehat{d-1}  d}^{-1}\widetilde{\bm{\Phi}}_{\widehat{d-1}  d,\widehat{d}}$ \Comment{$\mathcal{O}\left( r^2_{d-1}r_d N_{z_d}\right)$}
        \EndIf
        \EndFor
        \State Compute $\bm{\Phi}_{\widehat{d}, t, \widehat{t}}$ as in Alg. \ref{alg: ttsvd} 
        \State \Return  $\bm{\Phi}_{\widehat{0},1,\widehat{1}}, \hdots, \bm{\Phi}_{\widehat{d-1},d,\widehat{d}}, \bm{\Phi}_{\widehat{d},t,\widehat{t}}$
        \EndFunction
    \end{algorithmic}
    \label{alg: ttsvd orthogonality recursive}
\end{algorithm} 

\begin{remark}
    Alg. \eqref{alg: ttsvd orthogonality recursive} can be extended to produce an $\bm{X}_{s,s}$-orthogonal \ac{tt} decomposition even when $\bm{X}_{s,s}$ does not admit the form \eqref{eq: generic form X}, but can instead be expressed in the \ac{tt} format. In this case, the orthogonalization procedure must leverage the \ac{als} framework \cite{Holtz2011,doi:10.1137/110833142}. Since our numerical tests are limited to rank-$K$ norm matrices, we do not further investigate this scenario.
\end{remark}

\begin{remark}
    In practice, a slight modification of \eqref{eq: crossnorm X} is often required to ensure the well-posedness of Alg. \ref{alg: ttsvd orthogonality recursive}. Let us define
    \begin{equation*}
        q = \arg\max\limits_{k=1,\hdots,K}\{\|\cdot\|_{\overset{d}{\underset{i=1}{\otimes}} \bm{Y}_{i,i}^k}\}.
    \end{equation*}
    We recall that, by definition of the factors $\bm{Y}_{i,i}^q$ in \eqref{eq: H1 norm matrix}, the matrix $\bm{Y}_{q,q}^q$ is the stiffness matrix on the $q$th 1-$d$ \ac{fe} space. Since this matrix might be singular depending on the boundary conditions of the problem at hand, in practice we choose the norm matrix 
    \begin{equation*}
        \widecheck{\bm{X}}_{s,s} = \overset{d}{\underset{i=1}{\otimes}}\widecheck{\bm{Y}}_{i,i}^q, \quad \text{where} \quad
        \widecheck{\bm{Y}}_{i,i}^q = \bm{M}_{i,i} + \bm{X}_{i,i} \quad \text{if} \quad i=q, \quad \widecheck{\bm{Y}}_{i,i}^q = \bm{M}_{i,i} \quad \text{otherwise},
    \end{equation*}
    which is more akin to a reasonable crossnorm on $\mathcal{V}_h$.
\end{remark}

\subsection{Empirical interpolation method}
\label{subs: eim}
Let us consider the reduced linear system \eqref{eq: space time reduced problem}. Since both the \ac{lhs} and \ac{rhs} are parameter-dependent, they must, in principle, be assembled during the online phase for any new choice of $\bm{\mu}$. A more feasible \ac{rom} involves two main steps. First, one seeks an affine approximant of these quantities, in which each term is expressed as a product of a parameter-independent basis and a parameter-dependent reduced coefficient, as shown in \eqref{eq: affine expansions}. Second, one solves the corresponding affine reduced system \eqref{eq: space time mdeim reduced problem}. To construct such affine approximations, a collocation method is typically employed -- most commonly an \ac{eim}. Notable examples include the \ac{deim} \cite{chaturantabut2010nonlinear} and its matrix counterpart, the \ac{mdeim} \cite{NEGRI2015431}. We recall the procedure for the space-time \ac{mdeim} in Alg. \ref{alg: mdeim}, as presented in \cite{MUELLER2024115767}. Here, the symbol $\bm{e}_s(i)$ indicates the $i$th vector of the $N_s$-dimensional canonical basis.
\begin{algorithm}
    \caption{\texttt{ST-MDEIM}: Given the tensor of space-time residual snapshots $ \bm{L}_{s,t,\mu}$ and the prescribed accuracy $ \varepsilon $, build the $\ell^2$-orthogonal bases $\bm{\Phi}_{s,\widehat{s}}^{\bm{L}}$ , $\bm{\Phi}_{t,\widehat{t}}^{\bm{L}}$ , and sampling matrices $\bm{P}_{s,\widehat{s}}^{\bm{L}} \in \{0,1\}^{N_s \times n_s^{\bm{L}}}$, $\bm{P}_{t,\widehat{t}}^{\bm{L}} \in \{0,1\}^{N_t \times n_t^{\bm{L}}}$.}
	\begin{algorithmic}[1]
    \Function{\texttt{ST-MDEIM}}{$ \bm{L}_{s,t,\mu}, \varepsilon $}
	\State Compute $\bm{\Phi}_{s,\widehat{s}}^{\bm{L}},\bm{\Phi}_{t,\widehat{t}}^{\bm{L}} = \texttt{TPOD}\left(\bm{L}_{s,t,\mu},\varepsilon\right)$ as in Alg. \ref{alg: tpod}
	\State Set $\bm{P}_{s,\widehat{s}}^{\bm{L}} = \left[ \bm{e}_s(\mathcal{i}^1)  \right]$, where $\mathcal{i}^1 = \argmax \vert \bm{\Phi}_{s,\widehat{s}}^{\bm{L}}\left[:,1\right] \vert$ \Comment{Start \texttt{EIM-LOOP}}
	\For{$ q \in \{2,\ldots,{n_s^{\bm{L}}}\}$} 
    \State Set $\bm{V}_s = \bm{\Phi}_{s,\widehat{s}}^{\bm{L}}\left[:,q\right]$, $\bm{V}_{s,\widehat{s}} = \bm{\Phi}_{s,\widehat{s}}^{\bm{L}}\left[:,1:q-1\right]$
	\State Compute residual $\bm{r}_s = \bm{V}_s - \bm{V}_{s,\widehat{s}} \left( \bm{P}_{\widehat{s}, s}^{\bm{L}} \bm{V}_{s,\widehat{s}} \right)^{-1} \bm{P}_{\widehat{s}, s}^{\bm{L}} \bm{V}_s$
	\State Update $\bm{P}_{s,\widehat{s}}^{\bm{L}} = \left[ \bm{P}_{s,\widehat{s}}^{\bm{L}}, \bm{e}_s(\mathcal{i}^q)\right]$, where $\mathcal{i}^q = \argmax \vert \bm{r}_s \vert$
	\EndFor \Comment{End \texttt{EIM-LOOP}}
    \State Compute $\bm{P}_{t,\widehat{t}}^{\bm{L}}$ in the same way 
    \State \Return  $\bm{\Phi}_{s,\widehat{s}}^{\bm{L}}$ , $\bm{\Phi}_{t,\widehat{t}}^{\bm{L}}$ , $\bm{P}_{s,\widehat{s}}^{\bm{L}}$ , $\bm{P}_{t,\widehat{t}}^{\bm{L}}$
    \EndFunction
	\end{algorithmic}
	\label{alg: mdeim}
\end{algorithm}	
This procedure is computationally demanding due to the \ac{tpod} used to compute $\bm{\Phi}_{s,\widehat{s}}^{\bm{L}}$ and $\bm{\Phi}_{t,\widehat{t}}^{\bm{L}}$ ; however, it can be executed entirely offline. After running Alg. \ref{alg: mdeim}, and given an online parameter $\bm{\mu} \in \mathcal{D}_{\mathrm{on}}$, we empirically interpolate $\bm{L}_{st}^{\mu}$ as 
\begin{equation}
    \label{eq: tpod mdeim approximation}
    \bm{L}_{st}^{\mu} \approx \widehat{\bm{L}}_{st}^{\mu} = 
    \left( \bm{\Phi}_{s,\widehat{s}}^{\bm{L}} \otimes \bm{\Phi}_{t,\widehat{t}}^{\bm{L}} \right)  \widehat{\bm{L}}_{\widehat{st}}^{\mu}, 
    \quad \text{where} \quad
    \widehat{\bm{L}}_{\widehat{st}}^{\mu} = \left( \bm{P}^{\bm{L}}_{\widehat{s}, s}\bm{\Phi}_{s,\widehat{s}}^{\bm{L}} \otimes \bm{P}^{\bm{L}}_{\widehat{t}, t}\bm{\Phi}_{t,\widehat{t}}^{\bm{L}} \right)^{-1} \left( \bm{P}^{\bm{L}}_{\widehat{s}, s} \otimes \bm{P}^{\bm{L}}_{\widehat{t}, t} \right) \bm{L}_{st}^{\mu}.
\end{equation}  
Here, $\bm{P}^{\bm{L}}_{s,\widehat{s}} \in \{0,1\}^{N_s \times r_s}$ is a matrix of interpolation indices, constructed iteratively as described in Alg. \ref{alg: mdeim}. As highlighted in the algorithm, we may define the function \texttt{EIM-LOOP} returning an interpolation matrix from a given basis. We recall the accuracy of the procedure: 
\begin{equation}
    \label{eq: tpod mdeim error bound}
    \norm{\bm{L}_{st}^{\mu} - \widehat{\bm{L}}_{st}^{\mu}}_2 
    \leq \varepsilon \Big | \Big | \left( \bm{P}^{\bm{L}}_{\widehat{s}, s}\bm{\Phi}_{s,\widehat{s}}^{\bm{L}} \otimes \bm{P}^{\bm{L}}_{\widehat{t}, t}\bm{\Phi}_{t,\widehat{t}}^{\bm{L}} \right)^{-1} \Big | \Big |_F \sqrt{\norm{\bm{L}_{s,t\mu}}^2_F+\norm{\widehat{\bm{L}}_{t, \widehat{s} \mu}}^2_F}.
\end{equation}
We refer to \cite{MUELLER2024115767} for a complete proof. \\ 

In \ac{ttrb}, we change the first line of Alg. \ref{alg: mdeim} with a call to a standard \ac{ttsvd}, as presented in Alg. \ref{alg: ttsvd} (we simply require a Euclidean orthogonality for the residual basis). In this scenario, we seek an approximation in the form
\begin{equation}
    \label{eq: tt mdeim approximation}
    \bm{L}_{st}^{\mu} \approx \widehat{\bm{L}}_{st}^{\mu} = 
    \left( \bm{\Phi}_{\widehat{0},1,\widehat{1}}^{\bm{L}} \cdots \bm{\Phi}_{\widehat{d},t,\widehat{t}}^{\bm{L}} \right)  \widehat{\bm{L}}_{\widehat{t}}^{\mu}, 
\end{equation}
where
\begin{equation}
    \label{eq: tt mdeim online}
    \widehat{\bm{L}}_{\widehat{t}}^{\mu} = \left( \bm{P}^{\bm{L}}_{\widehat{t}, st}(\bm{\Phi}_{\widehat{0},1,\widehat{1}}^{\bm{L}} \cdots \bm{\Phi}_{\widehat{d},t,\widehat{t}}^{\bm{L}}) \right)^{-1} \bm{P}^{\bm{L}}_{\widehat{t}, st} \bm{L}_{st}^{\mu} = \left(\bm{P}^{\bm{L}}_{\widehat{t}, st} \bm{\Phi}^{\bm{L}}_{st,\widehat{t}}\right)^{-1} \bm{P}^{\bm{L}}_{\widehat{t}, st} \bm{L}_{st}^{\mu}.
\end{equation}
Here, $\{\bm{\Phi}_{\widehat{i-1},i,\widehat{i}}^{\bm{L}}\}_i$ denotes the \ac{tt} decomposition of the residual, with ranks $r^{\bm{L}}_1,\hdots,r^{\bm{L}}_t$. Although these ranks differ from the ones associated with the decomposition \eqref{eq: ttrb basis}, we use the same reduced subscripts $\widehat{i}$ for simplicity. The subscript $\widehat{i}^{\bm{L}}$ is only introduced when needed to avoid ambiguity. Once the \ac{tt} cores $\bm{\Phi}^{\bm{L}}_{\widehat{i-1},i,\widehat{i}}$ have been computed for every $i$, the goal is to determine the matrix of interpolation indices $\bm{P}^{\bm{L}}_{st,\widehat{t}}$. For this purpose, it is sufficient to run a single iteration of the TT-cross-DEIM method proposed in \cite{dektor2024collocationmethodsnonlineardifferential}. We note that our work does not simply use this algorithm, but builds upon it. Notably, \cite{dektor2024collocationmethodsnonlineardifferential} does not assemble the interpolation matrix $\bm{P}^{\bm{L}}_{st,\widehat{t}}$, and we provide an error bound for the method in Thm. \ref{thm: ttmdeim error bound}. Therefore, we believe it is necessary to formally present the simplified version of TT-cross-DEIM we employ, adopting a notation consistent with the rest of the manuscript. As the procedure is quite involved, we describe it both graphically, through Figs. \ref{fig: forward sweep ttmdeim}-\ref{fig: backward sweep ttmdeim}, and algorithmically, via Alg. \ref{alg: ttmdeim}. For clarity and to better reflect the nature of the operations involved, we henceforth refer to this method as \ac{ttmdeim}. \\ 
We begin with the following definition. Given three matrices 
\begin{equation*}
    \bm{R}_{a,c} \in \R^{N_a \times N_c}, \quad 
    \bm{S}_{b,d} \in \R^{N_b \times N_d}, \quad 
    \bm{T}_{ab,cd} = \bm{R}_{a,c} \otimes \bm{S}_{b,d} \in \R^{N_{ab}\times N_{cd}},
\end{equation*}
we introduce the index mapping $\mathcal{K}$, which relates the entries of $\bm{T}_{ab,cd}$ to those of the Kronecker factors:
\begin{equation}
    \label{eq: def kronecker product map}
    \bm{T}_{ab,cd}[\mathcal{K}(i_a,i_b),\mathcal{K}(j_c,j_d)] = \bm{R}_{a,c}[i_a,j_c] \otimes \bm{S}_{b,d}[i_b,j_d].
\end{equation}
Naturally, the expression of $\mathcal{K}$ depends on the size of $\bm{R}_{a,c}$ and $\bm{S}_{b,d}$. However, for simplicity of notation, we omit this dependence when referring to $\mathcal{K}$. Note that $\mathcal{K}$ is bijective and therefore admits a well-defined inverse. Now, consider an interpolation matrix of the form:
\begin{equation*}
    \bm{P}_{ab,c} = \left[\bm{e}_{ab}(\mathcal{i}^1)| \hdots | \bm{e}_{ab}(\mathcal{i}^{N_c})\right] \in \{0,1\}^{N_{ab} \times N_c}.
\end{equation*}
Using $\mathcal{K}^{-1}$, we can extract from $\bm{P}_{ab,c}$ the smaller interpolation matrices 
\begin{equation*}
    \bm{P}_{a,c} = \left[\bm{e}_{a}(\mathcal{i}^1)| \hdots | \bm{e}_{a}(\mathcal{i}^{N_c})\right] \in \{0,1\}^{N_{a} \times N_c}, \quad 
    \bm{P}_{b,c} = \left[\bm{e}_{b}(\mathcal{i}^1)| \hdots | \bm{e}_{b}(\mathcal{i}^{N_c})\right] \in \{0,1\}^{N_{b} \times N_c} \quad 
\end{equation*}
as follows: 
\begin{equation*}
    (\bm{P}_{a,c},\bm{P}_{b,c}) = \mathcal{K}^{-1} \left(\bm{P}_{ab,c}\right).
\end{equation*}
The expression above involves a slight abuse of notation, as we are effectively \textit{broadcasting} the operation $\mathcal{K}^{-1}$ over the indices encoded in $\bm{P}_{ab,c}$.
\begin{algorithm}
    \caption{\texttt{TT-MDEIM}: Given the tensor of space-time residual snapshots $ \bm{L}_{1:\mu}$ and the prescribed accuracy $ \varepsilon $, build the $\ell^2$-orthogonal \ac{tt} decomposition $\bm{\Phi}_{\widehat{0},1,\widehat{1}}^{\bm{L}},\hdots,\bm{\Phi}_{\widehat{d},t,\widehat{t}}^{\bm{L}}$ , and sampling matrices $\bm{P}_{1,\widehat{t}}^{\bm{L}} \in \{0,1\}^{N_1 \times r_t^{\bm{L}}},\hdots,\bm{P}_{t,\widehat{t}}^{\bm{L}} \in \{0,1\}^{N_t \times r_t^{\bm{L}}}$.}
	\begin{algorithmic}[1]
    \Function{\texttt{TT-MDEIM}}{$ \bm{L}_{1:\mu}, \varepsilon $}
	\State Compute $\bm{\Phi}_{\widehat{0},1,\widehat{1}}^{\bm{L}},\hdots,\bm{\Phi}_{\widehat{d},t,\widehat{t}}^{\bm{L}} = \texttt{TT-SVD}\left(\bm{L}_{1:\mu},\varepsilon\right)$ as in Alg. \ref{alg: ttsvd}
    \State Set $\widetilde{\bm{\Phi}}^{\bm{L}}_{\widehat{0}1,\widehat{1}} = \bm{\Phi}^{\bm{L}}_{1,\widehat{1}}$, and $\widetilde{\bm{P}}^{\bm{L}}_{\widehat{0}1,\widehat{1}} = \texttt{EIM-LOOP}\left(\widetilde{\bm{\Phi}}^{\bm{L}}_{\widehat{0}1,\widehat{1}}\right)$
	\For{$ i \in \{1,\hdots,d\}$} \Comment{Forward sweep}
    \State Interpolate \ac{tt} core $\widetilde{\bm{\Phi}}^{\bm{L}}_{\widehat{i},\widehat{i}} = \widetilde{\bm{P}}^{\bm{L}}_{\widehat{i},\widehat{i-1}i}\widetilde{\bm{\Phi}}^{\bm{L}}_{\widehat{i-1}i,\widehat{i}}$ 
    \State Update $\widetilde{\bm{\Phi}}^{\bm{L}}_{\widehat{i}(i+1),\widehat{i+1}} = \widetilde{\bm{\Phi}}^{\bm{L}}_{\widehat{i},\widehat{i}} \bm{\Phi}^{\bm{L}}_{\widehat{i},i+1,\widehat{i+1}}$
    \State Compute interpolation matrix $\widetilde{\bm{P}}^{\bm{L}}_{\widehat{i}(i+1),\widehat{i+1}} = \texttt{EIM-LOOP}\left(\widetilde{\bm{\Phi}}^{\bm{L}}_{\widehat{i}(i+1),\widehat{i+1}}\right)$
	\EndFor
    \For{$ i \in \{d,\hdots,1\}$} \Comment{Backward sweep}
    \State Split $\left(\widetilde{\bm{P}}^{\bm{L}}_{\widehat{i},\widehat{t}} \ , \bm{P}^{\bm{L}}_{i+1,\widehat{t}}\right) = \mathcal{K}^{-1}(\widetilde{\bm{P}}^{\bm{L}}_{\widehat{i}(i+1),\widehat{t}})$
    \State Compute $\widetilde{\bm{P}}^{\bm{L}}_{\widehat{i-1}i,\widehat{t}} = \widetilde{\bm{P}}^{\bm{L}}_{\widehat{i-1}i,\widehat{i}}\widetilde{\bm{P}}^{\bm{L}}_{\widehat{i},\widehat{t}}$
	\EndFor
    \State Set $\bm{P}^{\bm{L}}_{1,\widehat{t}} = \widetilde{\bm{P}}^{\bm{L}}_{\widehat{0}1,\widehat{t}}$
    \State Return $\bm{\Phi}_{\widehat{0},1,\widehat{1}}^{\bm{L}},\hdots,\bm{\Phi}_{\widehat{d},t,\widehat{t}}^{\bm{L}}$ , and $\bm{P}_{1,\widehat{t}}^{\bm{L}},\hdots,\bm{P}_{t,\widehat{t}}^{\bm{L}}$
    \EndFunction
	\end{algorithmic}
	\label{alg: ttmdeim}
\end{algorithm}	
\begin{figure}[t]
    \centering 
    \tikzset{every picture/.style={line width=0.75pt}} 

\tikzset{every picture/.style={line width=0.75pt}} 
    \begin{tikzpicture}[x=0.75pt,y=0.75pt,yscale=-1,xscale=1]
    \draw [fill={rgb, 255:red, 182; green, 208; blue, 238 }  ,fill opacity=1 ]  (160.76,50.25) -- (184.42,50.25) -- (107.67,127) -- (84.01,127) -- cycle ;
    \draw    (111.42,135.25) -- (78.42,135.25) ;
    \draw [shift={(76.42,135.25)}, rotate = 360] [color={rgb, 255:red, 0; green, 0; blue, 0 }  ][line width=0.75]    (10.93,-3.29) .. controls (6.95,-1.4) and (3.31,-0.3) .. (0,0) .. controls (3.31,0.3) and (6.95,1.4) .. (10.93,3.29)   ;
    \draw    (111.42,135.25) -- (195,51.66) ;
    \draw [shift={(196.42,50.25)}, rotate = 135] [color={rgb, 255:red, 0; green, 0; blue, 0 }  ][line width=0.75]    (10.93,-3.29) .. controls (6.95,-1.4) and (3.31,-0.3) .. (0,0) .. controls (3.31,0.3) and (6.95,1.4) .. (10.93,3.29)   ;
    \draw [fill={rgb, 255:red, 182; green, 208; blue, 238 }  ,fill opacity=1 ]  (338.76,37.25) -- (387.42,37.25) -- (264.42,160.25) -- (215.76,160.25) -- cycle ;
    \draw    (267.42,169.25) -- (392.01,43.67) ;
    \draw [shift={(393.42,42.25)}, rotate = 134.77] [color={rgb, 255:red, 0; green, 0; blue, 0 }  ][line width=0.75]    (10.93,-3.29) .. controls (6.95,-1.4) and (3.31,-0.3) .. (0,0) .. controls (3.31,0.3) and (6.95,1.4) .. (10.93,3.29)   ;
    \draw    (267.42,169.25) -- (208.42,168.28) ;
    \draw [shift={(206.42,168.25)}, rotate = 0.94] [color={rgb, 255:red, 0; green, 0; blue, 0 }  ][line width=0.75]    (10.93,-3.29) .. controls (6.95,-1.4) and (3.31,-0.3) .. (0,0) .. controls (3.31,0.3) and (6.95,1.4) .. (10.93,3.29)   ;
    \draw  [fill={rgb, 255:red, 182; green, 208; blue, 238 }  ,fill opacity=1 ] (506.42,70.25) -- (544.42,70.25) -- (451.42,163.25) -- (413.42,163.25) -- cycle ;
    \draw    (459.42,173.25) -- (565,67.66) ;
    \draw [shift={(566.42,66.25)}, rotate = 135] [color={rgb, 255:red, 0; green, 0; blue, 0 }  ][line width=0.75]    (10.93,-3.29) .. controls (6.95,-1.4) and (3.31,-0.3) .. (0,0) .. controls (3.31,0.3) and (6.95,1.4) .. (10.93,3.29)   ;
    \draw    (459.42,173.25) -- (410.42,173.25) ;
    \draw [shift={(408.42,173.25)}, rotate = 360] [color={rgb, 255:red, 0; green, 0; blue, 0 }  ][line width=0.75]    (10.93,-3.29) .. controls (6.95,-1.4) and (3.31,-0.3) .. (0,0) .. controls (3.31,0.3) and (6.95,1.4) .. (10.93,3.29)   ;
    \draw   [color=red  ,draw opacity=1 ][fill=red  ,fill opacity=1 ] (103.76,107.25) -- (127.42,107.25) -- (122.42,112.25) -- (98.76,112.25) -- cycle ;
    \draw   [color=Orchid  ,draw opacity=1 ][fill=Orchid  ,fill opacity=1 ] (257.76,118.25) -- (306.42,118.25) -- (301.42,123.25) -- (252.76,123.25) -- cycle ;
    \draw   [color=Orchid  ,draw opacity=1 ][fill=Orchid  ,fill opacity=1 ] (285.76,90.25) -- (334.42,90.25) -- (329.42,95.25) -- (280.76,95.25) -- cycle ;
    \draw   [color=Orchid  ,draw opacity=1 ][fill=Orchid  ,fill opacity=1 ] (333.76,42.25) -- (382.42,42.25) -- (377.42,47.25) -- (328.76,47.25) -- cycle ;
    \draw   [color=blue  ,draw opacity=1 ][fill=blue  ,fill opacity=1 ] (427.42,149.25) -- (465.42,149.25) -- (460.42,154.25) -- (422.42,154.25) -- cycle ;
    \draw   [color=blue  ,draw opacity=1 ][fill=blue  ,fill opacity=1 ] (463.42,113.25) -- (501.42,113.25) -- (496.42,118.25) -- (458.42,118.25) -- cycle ;
    \draw  [color=red][dash pattern={on 0.84pt off 2.51pt}]  (44.42,130.25) .. controls (194.9,285.19) and (381.64,125.51) .. (397,63.11) ;
    \draw [color=red][shift={(397.42,61.25)}, rotate = 101.13] [color=red  ][line width=0.75]    (10.93,-3.29) .. controls (6.95,-1.4) and (3.31,-0.3) .. (0,0) .. controls (3.31,0.3) and (6.95,1.4) .. (10.93,3.29)   ;
    \draw  [color=Orchid][dash pattern={on 0.84pt off 2.51pt}]  (235.42,58.25) .. controls (209.55,0.04) and (566.82,-23.97) .. (570.39,61.48) ;
    \draw [color=Orchid][shift={(570.42,62.75)}, rotate = 270] [color=Orchid  ][line width=0.75]    (10.93,-3.29) .. controls (6.95,-1.4) and (3.31,-0.3) .. (0,0) .. controls (3.31,0.3) and (6.95,1.4) .. (10.93,3.29)   ;
    \draw  [color=white,fill=white] (162.42,194.83) .. controls (162.42,187.1) and (174.28,180.83) .. (188.92,180.83) .. controls (203.55,180.83) and (215.42,187.1) .. (215.42,194.83) .. controls (215.42,202.57) and (203.55,208.83) .. (188.92,208.83) .. controls (174.28,208.83) and (162.42,202.57) .. (162.42,194.83) -- cycle ;
    \draw  [color=white,fill=white] (391.42,5.83) .. controls (391.42,-1.9) and (403.28,-8.17) .. (417.92,-8.17) .. controls (432.55,-8.17) and (444.42,-1.9) .. (444.42,5.83) .. controls (444.42,13.57) and (432.55,19.83) .. (417.92,19.83) .. controls (403.28,19.83) and (391.42,13.57) .. (391.42,5.83) -- cycle ;

    \draw (196.42,53.25) node [anchor=north west][inner sep=0.75pt]  [font=\scriptsize] [align=left] {{$1$}};
    \draw (78.42,138.25) node [anchor=north west][inner sep=0.75pt]  [font=\scriptsize] [align=left] {{$\widehat{1}$}};
    \draw (393.42,45.25) node [anchor=north west][inner sep=0.75pt]  [font=\scriptsize] [align=left] {{$\textcolor{red}{\widehat{1}}2$}};
    \draw (207.42,171.25) node [anchor=north west][inner sep=0.75pt] [font=\scriptsize]  [align=left] {{$\widehat{2}$}};
    \draw (411,177) node [anchor=north west][inner sep=0.75pt]  [font=\scriptsize] [align=left] {{$\widehat{t}$}};
    \draw (569,62) node [anchor=north west][inner sep=0.75pt] [font=\scriptsize]  [align=left] {{$\textcolor{Orchid}{\widehat{2}}t$}};
    \draw (155,28) node [anchor=north west][inner sep=0.75pt] [align=left] {$\bm{\Phi}_{\widehat{0},1,\widehat{1}}$};
    \draw (341,11) node [anchor=north west][inner sep=0.75pt]   [align=left] {$\widetilde{\bm{\Phi}}_{\widehat{0},\textcolor{red}{\widehat{1}}2,\widehat{2}}$};
    \draw (506,45) node [anchor=north west][inner sep=0.75pt]   [align=left] {$\widetilde{\bm{\Phi}}_{\widehat{0},\textcolor{Orchid}{\widehat{2}}t,\widehat{t}}$};
    \draw (84,102) node [anchor=north west][inner sep=0.75pt]  [font=\scriptsize] [align=left] {$\textcolor{red}{\widetilde{i}_1^1}$};
    \draw (238,113) node [anchor=north west][inner sep=0.75pt]  [font=\scriptsize] [align=left] {$\textcolor{Orchid}{\widetilde{i}_2^1}$};
    \draw (264,86) node [anchor=north west][inner sep=0.75pt]  [font=\scriptsize] [align=left] {$\textcolor{Orchid}{\widetilde{i}_2^2}$};
    \draw (313,38) node [anchor=north west][inner sep=0.75pt]  [font=\scriptsize] [align=left] {$\textcolor{Orchid}{\widetilde{i}_2^3}$};
    \draw (406,145) node [anchor=north west][inner sep=0.75pt]  [font=\scriptsize] [align=left] {$\textcolor{blue}{\widetilde{i}_t^1}$};
    \draw (442,109) node [anchor=north west][inner sep=0.75pt]  [font=\scriptsize] [align=left] {$\textcolor{blue}{\widetilde{i}_t^2}$};
    \draw (24,106) node [anchor=north west][inner sep=0.75pt]   [align=left] {$\widetilde{\bm{\Phi}}_{\widehat{0},\textcolor{red}{\widehat{1}},\widehat{1}}$};
    \draw (216,63) node [anchor=north west][inner sep=0.75pt]   [align=left] {$\widetilde{\bm{\Phi}}_{\widehat{0},\textcolor{Orchid}{\widehat{2}},\widehat{2}}$};
    \draw (380,115) node [anchor=north west][inner sep=0.75pt]   [align=left] {$\widetilde{\bm{\Phi}}_{\widehat{0},\textcolor{blue}{\widehat{t}},\widehat{t}}$};
    \draw (168,183) node [anchor=north west][inner sep=0.75pt]   [align=left] {$ \ \widetilde{\bm{\Phi}}_{\widehat{1},2,\widehat{2}}$};
    \draw (398,-5.17) node [anchor=north west][inner sep=0.75pt]   [align=left] {$ \ \widetilde{\bm{\Phi}}_{\widehat{2},t,\widehat{t}}$};
    \end{tikzpicture}
    \vspace*{-2.2cm}
    \caption{\ac{ttmdeim} forward sweep, case $d=2$.}
    \label{fig: forward sweep ttmdeim}
\end{figure}
\begin{figure}[t]
    \centering

\tikzset{every picture/.style={line width=0.75pt}} 

    \begin{tikzpicture}[x=0.75pt,y=0.75pt,yscale=-1,xscale=1]

    \draw (183.42,131.25) -- (212.42,131.25) -- (212.42,154.25) -- (183.42,154.25) -- cycle ;
    \draw    (154.42,143.75) -- (184,143.75) ;
    \draw  [color={rgb, 255:red, 182; green, 208; blue, 238 }, fill={rgb, 255:red, 182; green, 208; blue, 238 }  ,fill opacity=1] (116.42,143.75) .. controls (116.42,135.19) and (124.92,128.25) .. (135.42,128.25) .. controls (145.91,128.25) and (154.42,135.19) .. (154.42,143.75) .. controls (154.42,152.31) and (145.91,159.25) .. (135.42,159.25) .. controls (124.92,159.25) and (116.42,152.31) .. (116.42,143.75) -- cycle ;
    \draw  [color={rgb, 255:red, 182; green, 208; blue, 238 }, fill={rgb, 255:red, 182; green, 208; blue, 238 }  ,fill opacity=1] (244.42,165.25) .. controls (244.42,156.69) and (252.92,149.75) .. (263.42,149.75) .. controls (273.91,149.75) and (282.42,156.69) .. (282.42,165.25) .. controls (282.42,173.81) and (273.91,180.75) .. (263.42,180.75) .. controls (252.92,180.75) and (244.42,173.81) .. (244.42,165.25) -- cycle ;
    \draw    (213,141.25) -- (244,121.75) ;
    \draw    (213,141.25) -- (244,165.25) ;
    \draw  (311,109.25) -- (340.42,109.25) -- (340.42,132.25) -- (311.42,132.25) -- cycle ;
    \draw    (282,121.75) -- (311.5,121.75) ;
    \draw  [color={rgb, 255:red, 223; green, 205; blue, 239 }, fill={rgb, 255:red, 223; green, 205; blue, 239 }  ,fill opacity=1] (244.42,121.75) .. controls (244.42,113.19) and (252.92,106.25) .. (263.42,106.25) .. controls (273.91,106.25) and (282.42,113.19) .. (282.42,121.75) .. controls (282.42,130.31) and (273.91,137.25) .. (263.42,137.25) .. controls (252.92,137.25) and (244.42,130.31) .. (244.42,121.75) -- cycle ;
    \draw    (340,120.25) -- (372,100.75) ;
    \draw  [color={rgb, 255:red, 223; green, 205; blue, 239 },fill={rgb, 255:red, 223; green, 205; blue, 239 }] (372.42,100.75) .. controls (372.42,92.19) and (380.92,85.25) .. (391.42,85.25) .. controls (401.91,85.25) and (410.42,92.19) .. (410.42,100.75) .. controls (410.42,109.31) and (401.91,116.25) .. (391.42,116.25) .. controls (380.92,116.25) and (372.42,109.31) .. (372.42,100.75) -- cycle ;
    \draw  [fill={rgb, 255:red, 231; green, 168; blue, 176 },color={rgb, 255:red, 231; green, 168; blue, 176 }] (371.42,144.25) .. controls (371.42,135.69) and (379.92,128.75) .. (390.42,128.75) .. controls (400.91,128.75) and (409.42,135.69) .. (409.42,144.25) .. controls (409.42,152.81) and (400.91,159.75) .. (390.42,159.75) .. controls (379.92,159.75) and (371.42,152.81) .. (371.42,144.25) -- cycle ;
    \draw    (340,120.25) -- (371.42,144.25) ;

    \draw (130,135.25) node [anchor=north west][inner sep=0.75pt]   [align=left] {$\textcolor{blue}{\widetilde{\mathcal{i}}_t^q}$};
    \draw (184,134.25) node [anchor=north west][inner sep=0.75pt]   [align=left] {$\mathcal{K}^{-1} $};
    \draw (258,156.25) node [anchor=north west][inner sep=0.75pt]   [align=left] {$\textcolor{blue}{\mathcal{i}_t^q}$};
    \draw (258,113.25) node [anchor=north west][inner sep=0.75pt]   [align=left] {$\textcolor{Orchid}{\widetilde{\mathcal{i}}_2^q}$};
    \draw (312,112.25) node [anchor=north west][inner sep=0.75pt]   [align=left] {$\mathcal{K}^{-1} $};
    \draw (385,92.25) node [anchor=north west][inner sep=0.75pt]   [align=left] {$\textcolor{Orchid}{\mathcal{i}_2^q}$};
    \draw (384,135.25) node [anchor=north west][inner sep=0.75pt]   [align=left] {$\textcolor{red}{\mathcal{i}_1^q}$};
    \end{tikzpicture}
    \caption{\ac{ttmdeim} backward sweep, case $d=2$.}
    \label{fig: backward sweep ttmdeim}
\end{figure}
After executing Alg. \ref{alg: ttmdeim}, we can recover the space-time interpolation matrix $\bm{P}_{st,\widehat{t}}$ by applying $\mathcal{K}$ to the interpolation matrices $\bm{P}_{1,\widehat{t}}$ , $\hdots,\bm{P}_{t,\widehat{t}}$. For example, if $d=2$, we have:
\begin{equation}
    \bm{P}_{st,\widehat{t}} = \mathcal{K}\left(\mathcal{K}\left(\bm{P}_{1,\widehat{t}},\bm{P}_{2,\widehat{t}}\right),\bm{P}_{t,\widehat{t}}\right).
\end{equation}
As before, this expression involves an abuse of notation and should be interpreted in a broadcasting sense.

\begin{remark}
    \label{rmk: remark ttmdeim}
    The quantity $\bm{P}_{\widehat{t},st}^{\bm{L}}\bm{\Phi}_{st,\widehat{t}}^{\bm{L}} \ $, which is required for the \ac{ttmdeim} approximation (see \eqref{eq: tt mdeim online}), can be efficiently computed as a by-product of Alg. \ref{alg: ttmdeim}. Specifically, we have that 
    \begin{equation}
        \label{eq: forward sweep property}
        \bm{P}_{\widehat{t},st}^{\bm{L}}\bm{\Phi}_{st,\widehat{t}}^{\bm{L}} = \widetilde{\bm{P}}^{\bm{L}}_{\widehat{t},\widehat{d}t}\widetilde{\bm{\Phi}}^{\bm{L}}_{\widehat{d}t,\widehat{t}} \ ,
    \end{equation}
    where both $\widetilde{\bm{P}}^{\bm{L}}_{\widehat{t},\widehat{d}t}$ and $\widetilde{\bm{\Phi}}^{\bm{L}}_{\widehat{d}t,\widehat{t}}$ are available at the final iteration of the forward sweep. Rather than presenting a rigorous proof of \eqref{eq: forward sweep property}, we refer to Fig. \ref{fig: forward sweep ttmdeim}, which provides an illustration of this statement in the case $d=2$.
\end{remark}
We refer to \cite{dektor2024collocationmethodsnonlineardifferential} for a discussion on the computational cost of the method. This analysis is omitted here, as the cost is negligible compared to that of computing the \ac{tt} decomposition. We now present a theorem detailing the accuracy of the \ac{ttmdeim} method.
\begin{theorem}
    \label{thm: ttmdeim error bound}
    Let $\bm{P}_{st,\widehat{t}}^{\bm{L}}$ be computed by applying the procedure above on $\bm{\Phi}_{st,\widehat{t}}^{\bm{L}}$ . The following holds:
    \begin{equation}
        \label{eq: ttmdeim error bound}
        \norm{\bm{L}_{st}^{\mu} - \widehat{\bm{L}}_{st}^{\mu}}_2 
        \leq \varepsilon \sqrt{d+1} \chi^{\bm{L}} \norm{\bm{L}_{st,\mu}}_F ; \qquad 
        \chi^{\bm{L}} = \norm{ \left( \widetilde{\bm{P}}^{\bm{L}}_{\widehat{t},\widehat{2}t}\widetilde{\bm{\Phi}}^{\bm{L}}_{\widehat{2}t,\widehat{t}} \right)^{-1}}_F.
    \end{equation}
    \begin{proof}
        Firstly, note that the matrices $\widetilde{\bm{\Phi}}_{\widehat{1},\widehat{1}}^{\bm{L}} \ ,\hdots, \widetilde{\bm{\Phi}}_{\widehat{d},\widehat{d}}^{\bm{L}} \ , \widetilde{\bm{\Phi}}_{\widehat{t},\widehat{t}}^{\bm{L}} \ $ are full-rank (see \cite{dektor2024collocationmethodsnonlineardifferential}, Lemma 2.1). Thus, Lemma 3.2 in \cite{chaturantabut2010nonlinear} applies:
        \begin{equation}
            \label{eq: eq 1 proof 2}
            \norm{\bm{L}_{st}^{\mu} - \widehat{\bm{L}}_{st}^{\mu}}_2 
            \leq \norm{\left( \bm{P}_{\widehat{t},st}^{\bm{L}} \bm{\Phi}_{st,\widehat{t}}^{\bm{L}} \right)^{-1}}_F \norm{\left(\bm{I}_{st,st} - \bm{\Phi}_{st,\widehat{t}}^{\bm{L}}\bm{\Phi}_{\widehat{t},st}^{\bm{L}} \right)\bm{L}_{st,\mu}}_F = \chi^{\bm{L}} \norm{\left(\bm{I}_{st,st} - \bm{\Phi}_{st,\widehat{t}}^{\bm{L}}\bm{\Phi}_{\widehat{t},st}^{\bm{L}} \right)\bm{L}_{st,\mu}}_F.
        \end{equation}
        The equality above was derived using Remark \ref{rmk: remark ttmdeim}. Invoking Thm. \ref{thm: ttsvd accuracy}, we have that 
        \begin{equation}
            \norm{\left(\bm{I}_{st,st} - \bm{\Phi}_{st,\widehat{t}}^{\bm{L}}\bm{\Phi}_{\widehat{t},st}^{\bm{L}} \right)\bm{L}_{st,\mu}}_F \leq \varepsilon \sqrt{d+1} \norm{\bm{L}_{st,\mu}}_F,
        \end{equation}
        which proves the statement \eqref{eq: ttmdeim error bound} of the theorem.
    \end{proof}
\end{theorem}

Notably, the accuracy estimate \eqref{eq: ttmdeim error bound} is identical to the one that would be obtained by first explicitly assembling $\bm{\Phi}_{st,\widehat{t}}$ and then executing the \emph{for} loop in Alg. \ref{alg: mdeim}. 

\subsection{Approximation of the Jacobians}
\label{subs: approximation Jacobians}

To implement the \ac{ttmdeim} approximation for space-time Jacobians, we first require a ``split axes'' format for representing these quantities. We assume for simplicity that the sparsity pattern of the Jacobians does not vary for different values of $\bm{\mu}$. Let us momentarily consider a steady-state Jacobian $\bm{K}_{s,s}^{\mu}$. Since we operate within a Cartesian framework, we can define an index mapping
\begin{equation}
    \label{eq: tensor product sparsity}
    \mathcal{I}_z: \N(N_{z_1}) \times \cdots \times \N(N_{z_d}) \longrightarrow \N(N_z) 
\end{equation}
which associates a global index corresponding to a nonzero entry of $\bm{K}_{s,s}^{\mu}$ with a tuple of indices corresponding to nonzero entries of the 1-d Jacobians 
\begin{equation*}
    \bm{K}_{1,1}^{\mu}, \hdots, \bm{K}_{d,d}^{\mu}.
\end{equation*}
Note that, in general, there is no direct relationship between the entries of $\bm{K}_{s,s}^{\mu}$ and those of the 1-d Jacobians, in the sense that usually
\begin{equation*}
    \bm{K}_{1,1}^{\mu} \otimes \cdots \otimes \bm{K}_{d,d}^{\mu} \neq \bm{K}_{s,s}^{\mu}.
\end{equation*}
Nonetheless, it is still possible to infer information about the sparsity of $\bm{K}_{s,s}^{\mu}$ from the sparsity patterns of $\bm{K}_{1,1}^{\mu}, \hdots, \bm{K}_{d,d}^{\mu}$.
The mapping \eqref{eq: tensor product sparsity} allows us to identify the following ``split-axes'' formulation for the steady-state Jacobian:
\begin{equation*}
    \bm{K}_{z_1,\hdots,z_d}^{\mu}[i_{z_1},\hdots,i_{z_d}] = \bm{K}_z^{\mu}[i_z], \qquad \mathcal{I}_z (i_{z_1},\hdots,i_{z_d}) = i_z.
\end{equation*}
Given our assumption of fixed sparsity across the parameters, we can use the mapping \eqref{eq: tensor product sparsity} to interchangeably consider the snapshots tensors $\bm{K}_{s,s,\mu}$, $\bm{K}_{z,\mu}$ and $ \bm{K}_{z_1,\hdots,z_d,\mu}$. Similarly, in unsteady applications we have the congruence by isometry relationships 
\begin{equation*}
    \bm{K}_{st,st,\mu} \cong \bm{K}_{z,t,\mu} \cong \bm{K}_{z_1,\hdots,z_d,t,\mu}.
\end{equation*}
Therefore, we may run Alg. \ref{alg: ttsvd} on $\bm{K}_{z_1,\hdots,z_d,t,\mu}$ to compute the \ac{tt} decomposition 
\begin{equation*}
    \{\bm{\Phi}^{\bm{K}}_{\widehat{0},z_1,\widehat{1}}, \ 
    \hdots,  
    \bm{\Phi}^{\bm{K}}_{\widehat{d-1},z_d,\widehat{d}} \ , \ 
    \bm{\Phi}^{\bm{K}}_{\widehat{d},t,\widehat{t}}\}
    ,
\end{equation*}
and consequently perform a \ac{ttmdeim} approximation of the Jacobians:
\begin{equation}
    \label{eq: tt mdeim approximation Jacobian}
    \bm{K}_{zt}^{\mu} \approx \widehat{\bm{K}}_{zt}^{\mu} = 
    \left( \bm{\Phi}_{\widehat{0},z_1,\widehat{1}}^{\bm{K}} \cdots \bm{\Phi}_{\widehat{d},t,\widehat{t}}^{\bm{K}} \right)  \widehat{\bm{K}}_{\widehat{t}}^{\mu}, 
\end{equation}
where
\begin{equation}
    \label{eq: tt mdeim online Jacobian}
    \widehat{\bm{K}}_{\widehat{t}}^{\mu} = \left( \bm{P}^{\bm{K}}_{\widehat{t}, zt}(\bm{\Phi}_{\widehat{0},z_1,\widehat{1}}^{\bm{K}} \cdots \bm{\Phi}_{\widehat{d},t,\widehat{t}}^{\bm{K}}) \right)^{-1} \bm{P}^{\bm{K}}_{\widehat{t}, zt} \bm{K}_{zt}^{\mu} = \left(\bm{P}^{\bm{K}}_{\widehat{t}, zt} \bm{\Phi}^{\bm{K}}_{zt,\widehat{t}}\right)^{-1} \bm{P}^{\bm{K}}_{\widehat{t}, zt} \bm{K}_{zt}^{\mu}.
\end{equation}
We lastly remark that, by exploiting the sparsity of $\bm{K}^{\mu}_{i,i}$, the spatial cores can equivalently be represented as 4-d sparse arrays:
\begin{equation}
    \label{eq: 4-d cores jacobian}
    \bm{\Phi}^{\bm{K}}_{\widehat{i-1},i,i,\widehat{i}} \cong \bm{\Phi}^{\bm{K}}_{\widehat{i-1}, z_i,\widehat{i}} \ .
\end{equation} 
We use \eqref{eq: 4-d cores jacobian} in the following subsection, where we detail the Galerkin projection of \eqref{eq: tt mdeim approximation Jacobian} onto the \ac{ttrb} subspace identified by \eqref{eq: ttrb basis}.

\subsection{Galerkin projection}
\label{subs: galerkin projection}

In this subsection, we describe the assembly of the reduced problem \eqref{eq: space time mdeim reduced problem} when employing the \ac{ttrb} method. In this context, the projection operator is given by \eqref{eq: ttrb basis}, and the residuals and Jacobians are approximated via \ac{ttmdeim}. The assembly process involves the Galerkin projection of the \ac{ttmdeim} approximations onto the \ac{ttrb} subspace. As in previous subsections, we commence by recalling the procedure in a standard \ac{strb} setting. Moreover, we only focus on the projection of the Jacobian, as it is more intricate than that of the residual. Recalling the definition of $\mathcal{K}$ from \eqref{eq: def kronecker product map}, the Galerkin projection of the Jacobian in a standard \ac{strb} framework requires computing
\begin{equation}
    \label{eq: tpod galerkin Jacobian}
    \begin{split}
        \widebar{\bm{K}}^{\mu}_{\widehat{st},\widehat{st}} &= \left(\bm{\Phi}_{\widehat{s},s} \otimes \bm{\Phi}_{\widehat{t},t}\right) \left(\sum\limits_{i_s=1}^{r_s^{\bm{K}}}\sum\limits_{i_t=1}^{r_t^{\bm{K}}} \left(\bm{\Phi}^{\bm{K}}_{s,\widehat{s}^{\bm{K}},s}\left[:,i_s,:\right] \otimes \bm{\Phi}^{\bm{K}}_{t,\widehat{t}^{\bm{K}}}\left[:,i_t\right]\right) \widehat{\bm{K}}_{\widehat{st}^{\bm{K}}}^{\mu}\left[\mathcal{K}\left(i_s,i_t\right)\right] \right) \left(\bm{\Phi}_{s,\widehat{s}} \otimes \bm{\Phi}_{t,\widehat{t}}\right) \\
        &= 
        \sum\limits_{i_s=1}^{r_s^{\bm{K}}}\sum\limits_{i_t=1}^{r_t^{\bm{K}}} \bigg(\left(\bm{\Phi}_{\widehat{s},s}\bm{\Phi}^{\bm{K}}_{s,\widehat{s}^{\bm{K}},s}\left[:,i_s,:\right] \bm{\Phi}_{s,\widehat{s}}\right) \otimes \widehat{\bm{\Phi}}^{\bm{K}}_{\widehat{t},\widehat{t}^{\bm{K}},\widehat{t}}\left[:,i_t,:\right] \bigg) \widehat{\bm{K}}_{\widehat{st}^{\bm{K}}}^{\mu}\left[\mathcal{K}\left(i_s,i_t\right)\right],
    \end{split}
\end{equation}
where we have introduced 
\begin{equation*}
    \widehat{\bm{\Phi}}^{\bm{K}}_{\widehat{t},\widehat{t}^{\bm{K}},\widehat{t}} \in \R^{n_t \times n_t^{\bm{K}} \times n_t}, \quad
    \widehat{\bm{\Phi}}^{\bm{K}}_{\widehat{t},\widehat{t}^{\bm{K}},\widehat{t}}\left[\alpha,i_t,\beta\right] = \sum\limits_{n=1}^{N_t} \bm{\Phi}_{t,\widehat{t}}\left[n,\alpha\right]  \bm{\Phi}^{\bm{K}}_{t,\widehat{t}^{\bm{K}}}\left[n,i_t\right] \bm{\Phi}_{t,\widehat{t}}\left[n,\beta\right].
\end{equation*}
As shown above, the Jacobian reduction comprises a Kronecker product between a spatial and a temporal factor, the result of which is then multiplied by the $\bm{\mu}$-dependent coefficient. Since $N_z \gg N_t$ in practical applications, the cost of \eqref{eq: tpod galerkin Jacobian} scales as $\mathcal{O}\left( r_s^{\bm{K}} r_s^2 N_z \right)$, i.e. the complexity of computing the spatial factor. 
\\

Let us now consider the Jacobian projection in a \ac{tt} framework. Exploiting the mixed-product property of the Kronecker product, and skipping some computations that are conceptually straightforward but of tedious notation, we have
\begin{equation}
    \label{eq: tt galerkin Jacobian}
    \begin{split}
        \widebar{\bm{K}}^{\mu}_{\widehat{t},\widehat{t}} = &\sum_{\alpha_1,\beta_1,\delta_1}
        \bm{\Phi}_{\widehat{0},1,\widehat{1}}\left[1,:,\alpha_1\right]^T \bm{\Phi}_{\widehat{0}^{\bm{K}},1,1,\widehat{1}^{\bm{K}}}^{\bm{K}} \left[1,:,:,\beta_1\right] \bm{\Phi}_{\widehat{0},1,\widehat{1}}\left[1,:,\delta_1\right] \cdot \\
        & \sum_{\alpha_2,\beta_2,\delta_2} \hdots 
        \sum_{\alpha_d,\beta_d,\delta_d} \bm{\Phi}_{\widehat{d-1},d,\widehat{d}}\left[\alpha_{d-1},:,\alpha_d\right]^T \bm{\Phi}_{\widehat{d-1}^{\bm{K}},d,d,\widehat{d}^{\bm{K}}}^{\bm{K}} \left[\beta_{d-1},:,:,\beta_d\right] \bm{\Phi}_{\widehat{d-1},d,\widehat{d}}\left[\delta_{d-1},:,\delta_d\right] \cdot \\
        &\ \ \ \sum_{\beta_t} 
        \sum\limits_{n=1}^{N_t} \bm{\Phi}_{\widehat{d},t,\widehat{t}}\left[\alpha_{d},n,:\right]^T \bm{\Phi}_{\widehat{d}^{\bm{K}},t,\widehat{t}^{\bm{K}}}^{\bm{K}} \left[\beta_{d},n,\beta_t\right] \bm{\Phi}_{\widehat{d},t,\widehat{t}}\left[\delta_{d},n,:\right] \widehat{\bm{K}}_{\widehat{t}^{\bm{K}}}^{\mu}\left[\beta_t\right] 
        \\
        = & \ \ \sum\limits_{\beta_t} \widehat{\bm{\Phi}}_{\widehat{t},\widehat{t}^{\bm{K}},\widehat{t}}^{\bm{K}}[:,\beta_t,:] \widehat{\bm{K}}_{\widehat{t}^{\bm{K}}}^{\mu}[\beta_t].
    \end{split}
\end{equation}
Despite the presence of numerous indices, \eqref{eq: tt galerkin Jacobian} simply expresses, in terms of several 3d and 4-d \ac{tt} cores, the same spatial and temporal operations already discussed in \eqref{eq: tpod galerkin Jacobian} for \ac{strb}. Before deriving the cost of \eqref{eq: tt galerkin Jacobian}, we recall that the number of nonzeros in a sparse \ac{fe} matrix in a 1-d problem is $\mathcal{O}\left(M\right)$. Consequently, the cost of computing each spatial compression
\begin{equation*}
    \bm{\Phi}_{\widehat{i-1},i,\widehat{i}}\left[\alpha_{i-1},:,\alpha_i\right]^T \bm{\Phi}_{\widehat{i-1}^{\bm{K}},i,i,\widehat{i}^{\bm{K}}}^{\bm{K}} \left[\beta_{i-1},:,:,\beta_i\right] \bm{\Phi}_{\widehat{i-1},i,\widehat{i}}\left[\delta_{i-1},:,\delta_i\right]
\end{equation*}
scales as $\mathcal{O}\left(M\right)$. Summing over the indices, and introducing a ``bounding rank'' for the Jacobians
\begin{equation*}
    r_i^{\bm{K}} < r^{\bm{K}} \ \forall \ i 
\end{equation*}
we can show that the cost scales as $\mathcal{O}\left(d M (r^2 r^{\bm{K}})^2\right)$. Depending on the expression of $r$ and $r^{\bm{K}}$, the cost of this step can be considerable, though in our experience it is comparable to the Galerkin projection in an \ac{strb} framework. 
\begin{remark}
    When solving a vector-valued problem, the hyper-reduction for the residuals/Jacobians involves the \ac{tt} decompositions
    \begin{align}
        \label{eq: elasticity res/jac cores}
        \left\{\bm{\Phi}^{\bm{L}}_{\widehat{0},1,\widehat{1}} \ , \hdots , \bm{\Phi}^{\bm{L}}_{\widehat{d-1},d,\widehat{d}} \ , \bm{\Phi}^{\bm{L}}_{\widehat{d},c,\widehat{c}} \ , \bm{\Phi}^{\bm{L}}_{\widehat{c},t,\widehat{t}}\right\} \ ; \qquad 
        \left\{\bm{\Phi}^{\bm{K}}_{\widehat{0},1,1,\widehat{1}} \ , \hdots , \bm{\Phi}^{\bm{K}}_{\widehat{d-1},d,d,\widehat{d}} \ , \bm{\Phi}^{\bm{K}}_{\widehat{d},c,c,\widehat{c}} \ , \bm{\Phi}^{\bm{K}}_{\widehat{c},t,\widehat{t}}\right\} \ .
    \end{align}
    Recall that the subscript $c$ denotes the components axis, as explained in \eqref{eq: vector value tt}. In this case, the Galerkin projection increases by one step, due to the presence of the additional component core.
\end{remark}

\subsection{A posteriori error estimate}
\label{subs: accuracy}
In this subsection we present the accuracy estimate for the \ac{ttrb} method. The result follows directly from an adaptation of the analysis in \cite{MUELLER2024115767}, originally developed for the standard \ac{strb} approach. Before stating the result, we introduce a few key quantities:
\begin{itemize}
    \item The coercivity constant, in the $\bm{X}_{st,st}$-norm, of the full-order Jacobian:
    \begin{align}
        \label{eq: X-coercivity constant}
       C =
       \inf\limits_{\bm{V}_{st} \neq \bm{0}_{st}}
       \frac{ \| \bm{K}^{\mu}_{st,st} \bm{V}_{st} \|_{\bm{X}_{st,st}^{-1}}}{\norm{\bm{V}_{st}}_{\bm{X}_{st,st}} } 
       = \norm{\bm{X}_{st,st}^{-1/2}\bm{K}^{\mu}_{st,st}\bm{X}_{st,st}^{-1/2}}_2.
    \end{align}
    \item The \ac{ttmdeim} error due to interpolation for the Jacobian and the residual, respectively denoted as $\chi^{\bm{K}}$ and $\chi^{\bm{L}}$ (see \eqref{eq: ttmdeim error bound}). 
\end{itemize}
As in \cite{NEGRI2015431,MUELLER2024115767}, the total error introduced by the \ac{ttrb} method can be decomposed into two main contributions. The first is associated with the \ac{ttmdeim} hyper-reduction, while the second stems from the Galerkin projection onto the subspace spanned by $\bm{\Phi}_{st,\widehat{t}}$. 

\begin{theorem}
    \label{thm: error of tt-rb X norm}
    Let us consider the well-posed problem defined in \eqref{eq: space time FOM}, and its reduced approximation \eqref{eq: space time mdeim reduced problem}, obtained by combining the \ac{ttsvd} and \ac{ttmdeim} procedures. Let $\bm{U}^{\mu}_{st}$ denote the solution of the full-order model (\ac{fom}), and let $\widehat{\bm{U}}^{\mu}_{st}$ denote its reduced approximation. The following estimate holds:
    \begin{align}
        \label{eq: error of tt-rb X norm}
        \norm{\bm{U}^{\mu}_{st} - \widehat{\bm{U}}_{st}^{\mu}}_{\bm{X}_{st,st}} 
        \leq & \ C^{-1} \left(  \chi^{\bm{L}} \norm{\bm{L}_{s,t\mu}}_F \norm{\bm{X}_{st,st}^{-1/2}}_2 +  \chi^{\bm{K}} \norm{\bm{K}_{s,t\mu}}_F \norm{\bm{X}_{st,st}^{-1}}_2\norm{\widehat{\bm{U}}_{st}^{\mu}}_{\bm{X}_{st,st}} \right) \sqrt{d+1} \varepsilon 
        \\&+ C^{-1} \norm{\widehat{\bm{L}}^{\mu}_{st} - \widehat{\bm{K}}^{\mu}_{st,st}\widehat{\bm{U}}_{st}^{\mu}}_{\bm{X}_{st,st}^{-1}}.
    \end{align}
    \begin{proof}
        See the proof in \cite{MUELLER2024115767}, Theorem 1. The only difference lies in the presence of the $\sqrt{d+1}$ factor, which arises from applying a \ac{ttsvd} on the residuals and Jacobians. 
    \end{proof}
\end{theorem}
We refer to \cite{MUELLER2024115767} for additional details. Here, we simply remark that the bound in \eqref{eq: error of tt-rb X norm} may appear conservative, since not all terms explicitly decay with $\varepsilon$. Nevertheless, in practice, the residual-like quantity 
\begin{equation*}
    \norm{\widehat{\bm{L}}^{\mu}_{st} - \widehat{\bm{K}}^{\mu}_{st,st}\widehat{\bm{U}}_{st}^{\mu}}_{\bm{X}_{st,st}^{-1}}
\end{equation*}
is strongly correlated with $\varepsilon$. Specifically, when the problem is reducible according to the definition in \cite{quarteroni2015reduced}, if the reduced subspace is adequately constructed (i.e. via a sufficiently rich and representative parameter sampling during snapshot generation), this residual exhibits the same behavior as the estimate in \eqref{eq: ttsvd accuracy X}. 
\section{Numerical results} 
\label{sec: results} 
In this section, we analyze the numerical performance of the proposed \ac{ttrb} method by comparing it with the standard \ac{strb} approach. The comparison is conducted across different tolerances $\varepsilon$ and spatial \acp{dof} per direction $M$. Specifically, we focus on:
\begin{itemize}
    \item \emph{Offline performance}: We evaluate the efficiency of the \ac{ttrb} method compared to the \ac{strb} approach in constructing the $\bm{X}_{st,st}$-orthogonal reduced subspace. The analysis focuses on the impact of the spatial resolution $M$, the spatial dimension $d$, and the inclusion of time on the computational cost. Since the hyper-reduction costs are comparable for both methods, this step is excluded from the study.
    \item \emph{Online performance}: We assess and compare the error and computational speedup achieved by the \ac{strb} and \ac{ttrb} methods relative to the \ac{hf} solutions. This evaluation validates the accuracy estimate in \eqref{eq: error of tt-rb X norm} for various tolerances $\varepsilon \in \mathcal{E} = \{10^{-2},10^{-3},10^{-4}\}$.
\end{itemize}
We consider three benchmark problems: the Poisson equation, the heat equation, and the linear elasticity model. The Poisson equation is solved on both 2-$d$ and 3-$d$ geometries, while the heat and linear elasticity problems are restricted to 3-$d$ domains. For the spatial discretization, we use $Q1$ Lagrangian elements, and the transient problems are time-integrated using the Crank-Nicolson scheme. The offline and online parameter sets are disjoint, as defined in \eqref{eq: disjoint params}. We use $N_{\mu} = 80$ offline samples for the Poisson problem and $N_{\mu} = 50$ for the other benchmarks, with $N_{\mathrm{on}} = 10$ online samples for all tests. Hyper-reduction is performed using the first $30$ parameters from $\mathcal{D}_{\mathrm{off}}$. Offline parameters are sampled using a Halton sequence \cite{PHARR2017401}, while online parameters are uniformly drawn from $\mathcal{D}$. The accuracy of the methods is assessed using the metric defined in
\begin{equation}
\label{eq: error definition}
    E = \frac{1}{N_{\mathrm{on}}}\overset{N_{\mathrm{on}}}{\underset{i=1}{\sum}}\dfrac{\norm{\widehat{\bm{U}}_{st}^{\mu_i} - \bm{U}_{st}^{\mu_i}}_{\bm{X}_{st,st}}}{\norm{\bm{U}_{st}^{\mu_i}}_{\bm{X}_{st,st}}}.
\end{equation}
Since the Poisson equation is time-independent, a steady version of \eqref{eq: error definition} is employed. In all tests, the norm matrix $\bm{X}_{s,s}$ is defined as in \eqref{eq: H1 norm matrix}. To assess the computational efficiency of the \ac{strb} and \ac{ttrb} methods, we compute their speedup relative to the \ac{hf} simulations. The speedup is quantified as the ratio of the \ac{hf} cost, measured in terms of either wall time or memory usage, to the corresponding \ac{rom} cost. Additionally, we report the reduction factor, defined as the ratio of the \ac{fom} dimension $N_{st}$ to the dimension of the reduced subspace (i.e., $r_{st}$ for \ac{strb} and $r_t$ for \ac{ttrb}). All numerical experiments are conducted on a local machine equipped with 66$\si{Gb}$ of RAM and an Intel Core i7 processor running at 3.40$\si{GHz}$. The simulations utilize our \ac{rom} library \texttt{GridapROMs.jl} \cite{muellerbadiagridaproms}, implemented in the Julia programming language. Table~\ref{tb: fom details} summarizes the details of the \ac{hf} simulations. For the linear elasticity problem, which is defined on a cuboid rather than a cube, we report the average number of \acp{dof} per direction instead of $M$.
\begin{table}[t]
    \begin{tabular}{ccccccccccccc} 
        \toprule

        \multicolumn{1}{c}{\textbf{Measure}}
        &\multicolumn{3}{c}{\textbf{Poisson eq. 2 - }$\bm{d}$}
        &\multicolumn{3}{c}{\textbf{Poisson eq. 3 - }$\bm{d}$}
        &\multicolumn{3}{c}{\textbf{Heat eq. 3 - }$\bm{d}$}
        &\multicolumn{3}{c}{\textbf{Elasticity eq. 3 - }$\bm{d}$} \\
        
        \cmidrule(lr){2-4} \cmidrule(lr){5-7} \cmidrule(lr){8-10} \cmidrule(lr){11-13}

        \multirow{1}{*}[0.1cm]{}
        \textbf{Avg.} $\bm{M}$ &$250$ &$350$ &$460$ &$40$ &$50$ &$60$ &$40$ &$45$ &$50$ &$40$ &$45$ &$50$ \\

        \cmidrule(lr){2-4} \cmidrule(lr){5-7} \cmidrule(lr){8-10} \cmidrule(lr){11-13}

        \multirow{1}{*}[0.1cm]{}
        $\bm{N_t}$ &$//$ &$//$ &$//$ &$//$ &$//$ &$//$ &$10$ &$10$ &$10$ &$10$ &$10$ &$10$\\

        \cmidrule(lr){2-4} \cmidrule(lr){5-7} \cmidrule(lr){8-10} \cmidrule(lr){11-13}
    
        \multirow{1}{*}[0.1cm]{}
        \textbf{WT} $(\si{\second})$ &$0.29$ &$0.54$ &$0.82$ &$5.84$ &$23.21$ &$51.81$ &$72.28$ &$120.26$ &$229.89$ &$43.17$ &$95.31$ &$212.01$ \\
    
        \cmidrule(lr){2-4} \cmidrule(lr){5-7} \cmidrule(lr){8-10} \cmidrule(lr){11-13}
    
        \multirow{1}{*}[0.15cm]{}
        \textbf{MEM} $(Gb)$ &$0.12$ &$0.22$ &$0.38$ &$1.71$ &$4.84$ &$9.47$ &$18.68$ &$25.61$ &$52.89$ &$14.88$ &$27.97$ &$48.94$ \\
        \bottomrule 
    \end{tabular}
    \caption{Details of the \ac{hf} simulations. From left to right: Poisson equation on a 2-$d$ domain, Poisson equation on a 3-$d$ domain, heat equation on a 3-$d$ domain, and transient linear elasticity equation on a 3-$d$ domain. From top to bottom: average number of spatial \acp{dof} per direction (Avg. $M$), number of temporal \acp{dof}, average wall time (WT), and average memory allocations (MEM) of a \ac{hf} simulation.}
    \label{tb: fom details}
\end{table}

\subsection{Poisson equation}
The Poisson equation reads as 
\begin{equation}
    \label{eq: poisson equation}
    \begin{cases}
        - \bm{\nabla} \cdot (\alpha^{\mu} \bm{\nabla} u^{\mu}) = f^{\mu}  & \text{in} \ \Omega, \\
        u^{\mu} = g^{\mu}  & \text{on} \ \Gamma_D, \\
        \alpha^{\mu} \underline{n} \cdot \bm{\nabla} u^{\mu}  = h^{\mu} & \text{on} \ \Gamma_N,
    \end{cases} 
\end{equation}
where $\underline{n}$ is the normal vector to $\partial\Omega$, and
\begin{equation}
    \label{eq: poisson equation data}
    \alpha^{\mu}(\underline{x}) = \mu_1 + \mu_2 x_1,
    \quad
    f^{\mu}(\underline{x}) = \mu_3,
    \quad
    g^{\mu}(\underline{x}) = e^{-\mu_4 x_2},
    \quad
    h^{\mu}(\underline{x}) = \mu_5.
\end{equation}
The parameter space for this test is defined as $\mathcal{D} = [1,5]^5$ and
\begin{equation*}
    \bm{\mu} = \left(\mu_1,\mu_2,\mu_3,\mu_4,\mu_5\right)^T.
\end{equation*}
To validate the offline cost estimates in \eqref{eq: cost tpod}-\eqref{eq: cost ttsvd}, we solve the Poisson equation \eqref{eq: poisson equation} on two domains: $\Omega = [0,1]^2$ and $\Omega = [0,1]^3$. The values of $M$ are chosen such that the problem sizes are equivalent in both cases. For the 2-$d$ test, the Dirichlet and Neumann boundaries are defined as: 
\begin{equation}
    \label{eq: boundaries 2d}
    \Gamma_D = \{\underline{x} = (x_1,x_2)^T \in \partial\Omega : x_1 = 0 \}, \qquad 
    \Gamma_N = \{\underline{x} = (x_1,x_2)^T \in \partial\Omega : x_1 = 1 \},
\end{equation}
while for the 3-$d$ we use
\begin{equation}
    \label{eq: boundaries 3d}
    \Gamma_D = \{\underline{x} = (x_1,x_2,x_3)^T \in \partial\Omega : x_1 = 0 \}, \qquad 
    \Gamma_N = \{\underline{x} = (x_1,x_2,x_3)^T \in \partial\Omega : x_1 = 1 \}.
\end{equation}
In Table~\ref{tb: rb offline poisson equation}, we compare the offline phase costs of the steady \ac{rb} algorithm and the \ac{ttrb} method. For the \ac{rb} approach, the offline cost is largely independent of the tolerance, so we only report results for $\varepsilon = 10^{-4}$. In contrast, the offline cost of \ac{ttrb} varies with the tolerance, as the computational cost of each \ac{ttsvd} step depends on the size of the core computed in the previous iteration. Consequently, the \ac{ttrb} cost is presented as intervals: the lower bound corresponds to $\varepsilon = 10^{-2}$, and the upper bound to $\varepsilon = 10^{-4}$. As expected, \ac{ttsvd} outperforms \ac{tpod}, particularly in the 3-$d$ case. Although the 3-$d$ problem size is equivalent to the 2-$d$ one, \ac{tpod} is significantly more expensive in 3-$d$ due to the computations involving $\bm{X}_{s,s}$. In contrast, the cost of \ac{ttsvd} is largely unaffected by the spatial dimension $d$, as it scales with the \ac{rsvd} cost (see \eqref{eq: cost ttsvd}), which remains relatively stable given that both tests are designed to retain the same number of \acp{dof}. In the 2-$d$ case, \ac{ttsvd} is also more efficient, primarily due to smaller multiplicative constants, which aligns with the theoretical cost estimates in \eqref{eq: cost tpod}-\eqref{eq: cost ttsvd}.
\begin{table}[t]
    \small
    \begin{tabular}{cccccccc} 
        \toprule

        & &\multicolumn{3}{c}{\textbf{2-}$\bm{d}$} &\multicolumn{3}{c}{\textbf{3-}$\bm{d}$} \\

        \cmidrule(lr){3-5} \cmidrule(lr){6-8}
    
        &\multicolumn{1}{c}{\textbf{Measure}}
        &\multicolumn{1}{c}{$\bm{M} = 250$}
        &\multicolumn{1}{c}{$\bm{M} = 350$} 
        &\multicolumn{1}{c}{$\bm{M} = 460$} 
        &\multicolumn{1}{c}{$\bm{M} = 40$}
        &\multicolumn{1}{c}{$\bm{M} = 50$} 
        &\multicolumn{1}{c}{$\bm{M} = 60$}\\
        
        \midrule
    
        \multirow{2}{*}[0.1cm]{\textbf{TPOD}}
        &\textbf{WT} $(\si{\second})$ &$0.85$ &$2.13$ &$3.22$ &$5.74$ &$13.34$ &$30.31$ \\

        &\textbf{MEM} $(Gb)$ &$0.30$ &$0.63$ &$1.00$ &$2.24$ &$5.21$ &$10.29$ \\
    
        \midrule
    
        \multirow{2}{*}[0.1cm]{\textbf{TT-SVD}}
        &\textbf{WT} $(\si{\second})$ &$[0.16,0.19]$ &$[0.18,0.32]$ &$[0.58,0.75]$ &$[0.15,0.18]$ &$[0.34,0.39]$ &$[0.60,0.69]$ \\

        &\textbf{MEM} $(Gb)$ &$[0.05,0.16]$ &$[0.10,0.29]$ &$[0.17,0.42]$ &$[0.07,0.11]$ &$[0.12,0.18]$ &$[0.19,0.28]$ \\

        \bottomrule 
    \end{tabular}
    \caption{Offline results, Poisson equation. From left to right: results obtained on a 2-$d$ geometry, and a 3-$d$ geometry. From top to bottom: wall time (WT) and memory allocations (MEM) associated with the construction of the $\bm{X}_{s,s}$-orthogonal basis, using \ac{tpod} and \ac{ttsvd}, respectively. The \ac{tpod} results are computed with $\varepsilon = 10^{-4}$, while for \ac{ttsvd} we display the lower and upper bounds for every $\varepsilon \in \mathcal{E}$.}
    \label{tb: rb offline poisson equation}
\end{table}
Next, in Tbs. \ref{tb: online results 2d poisson equation}-\ref{tb: online results 3d poisson equation}, we report the results related to the online phase. We first note that, in the 2-$d$ benchmark, both methods suffer a loss of accuracy when $\varepsilon = 10^{-3}$. This is entirely due to the hyper-reduction of the \ac{lhs} selecting the same rank as for $\varepsilon = 10^{-2}$, thus inflating the error \eqref{eq: error of tt-rb X norm}. In our experience, this is a fairly common phenomenon. As expected, both methods achieve similar convergence rates, even though \ac{ttrb} exhibits larger constants than the traditional \ac{rb} (particularly in the 2-$d$ test). This is partly due to the presence of the $\sqrt{d}$ factor in the \ac{ttrb} estimate \eqref{eq: error of tt-rb X norm}; however, we have observed that this is mostly caused by higher hyper-reduction errors, particularly for the Jacobian. What occurs in this case is analogous to the aforementioned behavior for $\varepsilon = 10^{-3}$: for certain tolerances, \ac{ttmdeim} appears to select a rank similar to that chosen for a regime with a higher error threshold, resulting in a spike in the observed error. In terms of speedup, \ac{ttrb} performs slightly better, particularly in the 3-$d$ test, as it achieves larger reduction factors in both the reduced subspace and the hyper-reduction steps.
\begin{table}[t]
    \small
    \begin{tabular}{ccccccccccc} 
        \toprule
    
        & 
        &\multicolumn{3}{c}{$\bm{M} = 250$}
        &\multicolumn{3}{c}{$\bm{M} = 350$}
        &\multicolumn{3}{c}{$\bm{M} = 460$} \\

        \cmidrule(lr){3-5} \cmidrule(lr){6-8} \cmidrule(lr){9-11}

        &\multicolumn{1}{c}{\textbf{Measure}}
        &\multicolumn{1}{c}{$\bm{\varepsilon} = 10^{-2}$}
        &\multicolumn{1}{c}{$\bm{\varepsilon} = 10^{-3}$} 
        &\multicolumn{1}{c}{$\bm{\varepsilon} = 10^{-4}$} 
        &\multicolumn{1}{c}{$\bm{\varepsilon} = 10^{-2}$}
        &\multicolumn{1}{c}{$\bm{\varepsilon} = 10^{-3}$} 
        &\multicolumn{1}{c}{$\bm{\varepsilon} = 10^{-4}$}
        &\multicolumn{1}{c}{$\bm{\varepsilon} = 10^{-2}$}
        &\multicolumn{1}{c}{$\bm{\varepsilon} = 10^{-3}$} 
        &\multicolumn{1}{c}{$\bm{\varepsilon} = 10^{-4}$}\\

        \cmidrule{2-11}
    
        \multirow{4}{*}[0.1cm]{\rotatebox[origin=c]{90}{\textbf{RB}}}
        &\textbf{E} / $\bm{\varepsilon}$ &$3.66$ &$35.80$ &$0.97$ &$2.45$ &$21.00$ &$0.81$ &$4.58$ &$40.57$ &$0.81$ \\

        &\textbf{RF} / $10^4$ &$1.25$ &$0.78$ &$0.57$ &$2.46$ &$1.53$ &$1.12$ &$4.24$ &$3.03$ &$1.93$ \\

        &\textbf{SU}-\textbf{WT} / $10^2$ &$3.38$ &$3.33$ &$3.17$ &$4.78$ &$4.50$ &$4.13$ &$5.26$ &$5.26$ &$5.21$ \\

        &\textbf{SU}-\textbf{MEM} / $10^2$ &$0.48$ &$0.48$ &$0.48$ &$0.48$ &$0.48$ &$0.48$ &$0.50$ &$0.50$ &$0.50$ \\

        \cmidrule{2-11}
    
        \multirow{4}{*}[0.1cm]{\rotatebox[origin=c]{90}{\textbf{TT-RB}}}
        &\textbf{E} / $\bm{\varepsilon}$ &$5.39$ &$52.05$ &$10.81$ &$3.34$ &$26.53$ &$16.45$ &$5.07$ &$45.14$ &$21.89$ \\

        &\textbf{RF} / $10^4$ &$1.57$ &$0.78$ &$0.57$ &$3.07$ &$2.46$ &$1.12$ &$4.24$ &$3.03$ &$1.93$ \\

        &\textbf{SU}-\textbf{WT} / $10^2$ &$3.59$ &$3.55$ &$3.51$ &$4.67$ &$4.44$ &$4.44$ &$5.71$ &$5.09$ &$4.71$ \\

        &\textbf{SU}-\textbf{MEM} / $10^2$ &$0.61$ &$0.61$ &$0.61$ &$0.61$ &$0.61$ &$0.61$ &$0.64$ &$0.64$ &$0.64$ \\

        \bottomrule 
    \end{tabular}
    \caption{Online results for the Poisson equation on a 2-$d$ domain. Metrics include: average accuracy (E), normalized with respect to $\varepsilon$; reduction factor (RF), expressed in tens of thousands; and average computational speedup in terms of wall time (SU-WT) and memory usage (SU-MEM), both expressed in hundreds. Results compare the performance of \ac{strb} and \ac{ttrb} relative to the \ac{hf} simulations.}
    \label{tb: online results 2d poisson equation}
\end{table}
\begin{table}[t]
    \small
    \begin{tabular}{ccccccccccc} 
        \toprule
    
        & 
        &\multicolumn{3}{c}{$\bm{M} = 40$}
        &\multicolumn{3}{c}{$\bm{M} = 50$}
        &\multicolumn{3}{c}{$\bm{M} = 60$} \\

        \cmidrule(lr){3-5} \cmidrule(lr){6-8} \cmidrule(lr){9-11}

        &\multicolumn{1}{c}{\textbf{Measure}}
        &\multicolumn{1}{c}{$\bm{\varepsilon} = 10^{-2}$}
        &\multicolumn{1}{c}{$\bm{\varepsilon} = 10^{-3}$} 
        &\multicolumn{1}{c}{$\bm{\varepsilon} = 10^{-4}$} 
        &\multicolumn{1}{c}{$\bm{\varepsilon} = 10^{-2}$}
        &\multicolumn{1}{c}{$\bm{\varepsilon} = 10^{-3}$} 
        &\multicolumn{1}{c}{$\bm{\varepsilon} = 10^{-4}$}
        &\multicolumn{1}{c}{$\bm{\varepsilon} = 10^{-2}$}
        &\multicolumn{1}{c}{$\bm{\varepsilon} = 10^{-3}$} 
        &\multicolumn{1}{c}{$\bm{\varepsilon} = 10^{-4}$}\\

        \cmidrule{2-11}
    
        \multirow{4}{*}[0.1cm]{\rotatebox[origin=c]{90}{\textbf{RB}}}
        &\textbf{E} / $\bm{\varepsilon}$ &$5.75$ &$0.89$ &$1.32$ &$5.58$ &$0.90$ &$1.46$ &$7.29$ &$1.05 $ &$1.24$  \\

        &\textbf{RF} / $10^4$ &$1.34$ &$0.84$ &$0.61$ &$2.60$ &$1.62$ &$1.18$ &$4.46$ &$2.79$ &$2.03$ \\

        &\textbf{SU}-\textbf{WT} / $10^2$  &$5.68$ &$5.54$ &$5.09$ &$8.50$ &$8.44$ &$8.35$ &$9.13$ &$8.79$ &$8.56$ \\

        &\textbf{SU}-\textbf{MEM} / $10^2$  &$0.62$ &$0.62$ &$0.62$ &$0.97$ &$0.96$ &$0.96$ &$1.13$ &$1.13$ &$1.13$ \\

        \cmidrule{2-11}
    
        \multirow{4}{*}[0.1cm]{\rotatebox[origin=c]{90}{\textbf{TT-RB}}}
        &\textbf{E} / $\bm{\varepsilon}$ &$16.26$ &$3.94$ &$4.13$ &$12.91$ &$3.41$ &$5.31$ &$13.76$ &$3.92$ &$4.93$ \\

        &\textbf{RF} / $10^4$ &$1.34$ &$0.96$ &$0.67$ &$3.25$ &$1.86$ &$1.18$ &$7.44$ &$3.19$ &$2.23$ \\

        &\textbf{SU}-\textbf{WT} / $10^2$ &$4.92$ &$4.74$ &$4.42$ &$13.68$ &$13.17$ &$13.10$ &$21.25$ &$20.97$ &$20.15$ \\

        &\textbf{SU}-\textbf{MEM} / $10^2$ &$0.64$ &$0.64$ &$0.64$ &$1.00$ &$0.99$ &$0.99$ &$1.16$ &$1.16$ &$1.16$ \\

        \bottomrule 
    \end{tabular}
    \caption{Online results for the Poisson equation on a 3-$d$ domain. Metrics include: average accuracy (E), normalized with respect to $\varepsilon$; reduction factor (RF), expressed in tens of thousands; and average computational speedup in terms of wall time (SU-WT) and memory usage (SU-MEM), both expressed in hundreds. Results compare the performance of \ac{strb} and \ac{ttrb} relative to the \ac{hf} simulations.}
    \label{tb: online results 3d poisson equation}
\end{table}

\subsection{Heat equation}
\label{subs: heat equation results}
In this section, we present the numerical solution of the heat equation:
\begin{equation} 
    \label{eq: heat equation}
    \begin{cases}
        \frac{\partial u^{\mu}}{\partial t} - \bm{\nabla} \cdot (\alpha^{\mu} \bm{\nabla} u^{\mu}) = f^{\mu}  & \text{in} \ \Omega \times (0,T], \\
        u^{\mu} = g^{\mu}  & \text{on} \ \Gamma_D \times (0,T], \\
        \alpha^{\mu} \underline{n} \cdot \bm{\nabla} u^{\mu}  = h^{\mu} & \text{on} \ \Gamma_N \times (0,T], \\
        u^{\mu} = u_0^{\mu} & \text{in} \ \Omega \times \{0\},
    \end{cases} 
\end{equation}
characterized by the following parametric data:
\begin{alignat}{3}
    \label{eq: heat equation data}
        \alpha^{\mu}(\underline{x},t) &= \mu_1 + \mu_2 x_1,
        \qquad
        &&f^{\mu}(\underline{x},t) &&= \mu_3,
        \qquad
        g^{\mu}(\underline{x},t) = e^{-\mu_4 x_2}\left(1-\cos{(2 \pi t/T)} + \sin{(2 \pi t/T)}/\mu_5\right), \\
        h^{\mu}(\underline{x},t) &= \sin{(2 \pi t/T)}/\mu_6, \quad
        &&u^{\mu}(\underline{x}) &&= 0.
\end{alignat}
In this test case, we only consider $\Omega = [0,1]^3$, with the Dirichlet and Neumann boundaries defined as in \eqref{eq: boundaries 3d}. The temporal domain is $[0,T]$, with $T =0.1$, discretized into $N_t = 10$ uniform intervals. The parameter space we consider for this test is $\mathcal{D} = [1,5]^6$. \\ 
The results concerning the offline and online phases are shown in Tbs.~\ref{tb: rb offline heat equation}-\ref{tb: online results heat equation}. Similarly to the previous test case, the generation of $\bm{X}_{st,st}$-orthogonal subspaces is far cheaper in \ac{ttrb}. The speedup the latter achieves with respect to \ac{strb} is less impressive than in the 3-$d$ Poisson equation, since we had to consider lower values of $M$ for the heat equation due to memory constraints. Additionally, the presence of time decreases the speedup for \ac{ttrb}, as it increases the cost of the \acp{rsvd} without impacting the operations involving $\bm{X}_{s,s}$ -- which represent the leading computational cost for \ac{strb}. In terms of online results, we notice the significantly higher reduction factors for \ac{ttrb}, which is a result of considering a \ac{rb} subspace of dimension $r_t$ instead of $r_{st}$, as is the case for \ac{strb}. This significant dimensionality reduction translates in improved online speedups, particularly in terms of memory consumption.  
\begin{table}[t]
    \small
    \begin{tabular}{ccccc} 
        \toprule

        &\multicolumn{1}{c}{\textbf{Measure}}
        &\multicolumn{1}{c}{$\bm{M} = 40$}
        &\multicolumn{1}{c}{$\bm{M} = 45$} 
        &\multicolumn{1}{c}{$\bm{M} = 50$}\\
        
        \midrule
    
        \multirow{2}{*}[0.1cm]{\textbf{TPOD}}
        &\textbf{WT} $(\si{\second})$ &$30.38$ &$48.90$ &$75.56$ \\

        &\textbf{MEM} $(Gb)$ &$2.76$ &$4.21$ &$6.21$ \\
    
        \midrule
    
        \multirow{2}{*}[0.1cm]{\textbf{TT-SVD}}
        &\textbf{WT} $(\si{\second})$ &$[2.26,3.08]$ &$[3.36,3.86]$ &$[4.44,5.03]$ \\

        &\textbf{MEM} $(Gb)$ &$[0.56,0.66]$ &$[0.61,0.87]$ &$[0.80,1.11]$ \\

        \bottomrule 
    \end{tabular}
    \caption{Offline results, heat equation on a 3-$d$ domain. From top to bottom: wall time (WT) and memory allocations (MEM) associated with the construction of the $\bm{X}_{s,s}$-orthogonal basis, using \ac{tpod} and \ac{ttsvd}, respectively. The \ac{tpod} results are computed with $\varepsilon = 10^{-4}$, while for \ac{ttsvd} we display the lower and upper bounds for every $\varepsilon \in \mathcal{E}$.}
    \label{tb: rb offline heat equation}
\end{table}
 
\begin{table}[t]
    \small
    \begin{tabular}{ccccccccccc} 
        \toprule
    
        & 
        &\multicolumn{3}{c}{$\bm{M} = 40$}
        &\multicolumn{3}{c}{$\bm{M} = 45$}
        &\multicolumn{3}{c}{$\bm{M} = 50$} \\

        \cmidrule(lr){3-5} \cmidrule(lr){6-8} \cmidrule(lr){9-11}

        &\multicolumn{1}{c}{\textbf{Measure}}
        &\multicolumn{1}{c}{$\bm{\varepsilon} = 10^{-2}$}
        &\multicolumn{1}{c}{$\bm{\varepsilon} = 10^{-3}$} 
        &\multicolumn{1}{c}{$\bm{\varepsilon} = 10^{-4}$} 
        &\multicolumn{1}{c}{$\bm{\varepsilon} = 10^{-2}$}
        &\multicolumn{1}{c}{$\bm{\varepsilon} = 10^{-3}$} 
        &\multicolumn{1}{c}{$\bm{\varepsilon} = 10^{-4}$}
        &\multicolumn{1}{c}{$\bm{\varepsilon} = 10^{-2}$}
        &\multicolumn{1}{c}{$\bm{\varepsilon} = 10^{-3}$} 
        &\multicolumn{1}{c}{$\bm{\varepsilon} = 10^{-4}$}\\

        \cmidrule{2-11}
    
        \multirow{4}{*}[0.1cm]{\rotatebox[origin=c]{90}{\textbf{ST-RB}}}
        &\textbf{E} / $\bm{\varepsilon}$ &$4.58$ &$2.09$ &$2.38$ &$3.97$ &$2.40$ &$2.29$ &$4.29$ &$2.89$ &$3.00$ \\

        &\textbf{RF} / $10^4$ &$1.25$ &$0.42$ &$0.28$ &$3.40$ &$1.32$ &$0.57$ &$2.41$ &$0.81$ &$0.54$ \\

        &\textbf{SU}-\textbf{WT} / $10^3$ &$12.78$ &$10.31$ &$8.68$ &$17.75$ &$15.18$ &$13.55$ &$23.68$ &$21.82$ &$21.36$ \\

        &\textbf{SU}-\textbf{MEM} / $10^3$ &$3.67$ &$3.09$ &$2.63$ &$3.88$ &$3.68$ &$3.32$ &$6.00$ &$5.42$ &$4.88$ \\

        \cmidrule{2-11}
    
        \multirow{4}{*}[0.1cm]{\rotatebox[origin=c]{90}{\textbf{TT-RB}}}
        &\textbf{E} / $\bm{\varepsilon}$ &$7.90$ &$9.64$ &$50.30$ &$7.96$ &$8.59$ &$46.71$ &$6.61$ &$6.52$ &$26.37$ \\

        &\textbf{RF} / $10^4$ &$96.06$ &$42.03$ &$25.86$ &$136.03$ &$59.51$ &$35.27$ &$185.78$ &$81.29$ &$50.02$ \\

        &\textbf{SU}-\textbf{WT} / $10^3$ &$13.28$ &$11.68$ &$9.16$ &$17.15$ &$15.04$ &$14.55$ &$26.89$ &$ 23.76$ &$23.74$ \\

        &\textbf{SU}-\textbf{MEM} / $10^3$ &$4.45$ &$3.79$ &$2.98$ &$4.85$ &$4.40$ &$4.25$ &$7.18$ &$6.65$ &$6.64$ \\

        \bottomrule 
    \end{tabular}
    \caption{Online results, heat equation on a 3-$d$ domain. From top to bottom: average accuracy (E) normalized with respect to $\varepsilon$; reduction factor (RF), in hundreds of thousands; average computational speedup in time (SU-WT) and in memory (SU-MEM) achieved by \ac{strb} and \ac{ttrb} with respect to the \ac{hf} simulations.}
    \label{tb: online results heat equation}
\end{table} 

\subsection{Linear Elasticity Problem}
\label{subs: results elasticity}
In this subsection, we solve a transient version of the linear elasticity equation:
\begin{equation} 
    \label{eq: strong form elasticity equation}
    \begin{cases}
        \frac{\partial \underline{u}^{\mu}}{\partial t} - \bm{\nabla} \cdot (\underline{\underline{\sigma}}^{\mu} (\underline{u}^{\mu}) ) = \underline{0}  & \text{in} \ \Omega \times (0,T], \\
        \underline{u}^{\mu} = \underline{g}^{\mu}  & \text{on} \ \Gamma_D \times (0,T], \\
        \underline{\underline{\sigma}}^{\mu}(\underline{u}^{\mu}) \cdot \underline{n} = \underline{h}^{\mu} & \text{on} \ \Gamma_N \times (0,T], \\
        \underline{u}^{\mu} = \underline{u}_0^{\mu} & \text{in} \ \Omega \times \{0\}.
    \end{cases} 
\end{equation}
The displacement field $\underline{u}$ is vector-valued, requiring the inclusion of a component axis as illustrated in \eqref{eq: vector value tt}. The stress tensor $\underline{\underline{\sigma}}^{\mu}$ is expressed as
\begin{equation}
    \label{eq: stress tensor}
    \underline{\underline{\sigma}}^{\mu} (\underline{u}^{\mu}) = 2p^{\mu} \underline{\epsilon} (\underline{u}^{\mu}) + \lambda^{\mu} \bm{\nabla} \cdot (\underline{u}^{\mu}) \underline{I},
\end{equation}
where $\underline{\epsilon}$ is the symmetric gradient operator, and $\lambda^{\mu}$, $p^{\mu}$ are the Lam\'e coefficients one can express as functions of the Young modulus $E^{\mu}$ and the Poisson coefficient $\nu^{\mu}$ as
\begin{align}
    \lambda^{\mu} = \frac{E^{\mu} \nu^{\mu}}{(1+\nu^{\mu})(1-2\nu^{\mu})}; \qquad
    p^{\mu} = \frac{E^{\mu} }{2(1+\nu^{\mu})}.
\end{align} 
We consider a 3-$d$ domain 
\begin{equation*}
    \Omega = [0,1] \times [0,1/8]^2
\end{equation*}
equipped with a Dirichlet boundary defined as in \eqref{eq: boundaries 3d}, on which we impose a homogeneous condition, and three Neumann boundaries
\begin{equation}
    \begin{split}
        \Gamma_{N_1} &= \{\underline{x} = (x_1,x_2,x_3)^T \in \partial\Omega : x_1 = 1\}; \\
        \Gamma_{N_2} &= \{\underline{x} = (x_1,x_2,x_3)^T \in \partial\Omega : x_2 = 0, x_1 \in (0,1)\};\\
        \Gamma_{N_3} &= \{\underline{x} = (x_1,x_2,x_3)^T \in \partial\Omega : x_3 = 0, x_1 \in (0,1)\}. 
    \end{split}
\end{equation}
We consider the parametric data 
\begin{equation}
    \label{eq: elasticity data}
    \begin{split}
        E^{\mu}(t) &= \mu_1 e^{\sin{(2\pi t/T)}},
        \quad
        \nu^{\mu}(t) = \mu_2 e^{\sin{(2\pi t/T)}},
        \quad
        \underline{g}^{\mu}(\underline{x},t) = \underline{0}, \quad 
        u_0^{\mu}(\underline{x}) = \underline{0}, \\
        \underline{h}^{\mu}(\underline{x},t) &= \begin{cases}
            \mu_3 (1+t) \underline{n}_1 \qquad \ \ \underline{x} \in \Gamma_{N_1}, \ \underline{n}_1 = (1,0,0)^T; \\
            \mu_4 e^{\sin{(2\pi t/T)}} \underline{n}_2 \quad \underline{x} \in \Gamma_{N_2}, \ \underline{n}_2 = (0,1,0)^T; \\ 
            \mu_5 e^{\cos{(2\pi t/T)}} \underline{n}_3 \quad \underline{x} \in \Gamma_{N_3}, \ \underline{n}_3 = (0,0,1)^T.
        \end{cases} \quad 
    \end{split}
\end{equation}
The parametric domain is 
\begin{equation}
    \mathcal{D} = \left[10^{10},9\cdot 10^{10}\right] \times \left[0.25,0.42\right] \times \left[-4 \cdot 10^5,4 \cdot 10^5\right]^3.
\end{equation} 
\begin{table}[t]
    \small
    \begin{tabular}{ccccc} 
        \toprule

        &\multicolumn{1}{c}{\textbf{Measure}}
        &\multicolumn{1}{c}{\textbf{Avg.} $\bm{M} = 40$}
        &\multicolumn{1}{c}{\textbf{Avg.} $\bm{M} = 45$}
        &\multicolumn{1}{c}{\textbf{Avg.} $\bm{M} = 50$}\\
        
        \midrule
    
        \multirow{2}{*}[0.1cm]{\textbf{TPOD}}
        &\textbf{WT} $(\si{\second})$ &$25.29$ &$43.21$ &$93.15$ \\

        &\textbf{MEM} $(Gb)$ &$2.72$ &$4.46$ &$7.04$ \\
    
        \midrule
    
        \multirow{2}{*}[0.1cm]{\textbf{TT-SVD}}
        &\textbf{WT} $(\si{\second})$ &$[1.26,1.27]$ &$[1.87,1.97]$ &$[3.93,5.75]$ \\

        &\textbf{MEM} $(Gb)$ &$[0.41,0.41]$ &$[0.58,0.59]$ &$[0.64,1.21]$ \\

        \bottomrule 
    \end{tabular}
    \caption{Offline results for the transient linear elasticity problem. Metrics include wall time (WT) and memory usage (MEM) for constructing the $\bm{X}_{s,s}$-orthogonal basis using \ac{tpod} and \ac{ttsvd}. The \ac{tpod} results are computed with $\varepsilon = 10^{-4}$, while for \ac{ttsvd} we display the lower and upper bounds for every $\varepsilon \in \mathcal{E}$.}
    \label{tb: rb offline elasticity}
\end{table}

\begin{table}[t]
    \small
    \begin{tabular}{ccccccccccc} 
        \toprule
    
        & 
        &\multicolumn{3}{c}{\textbf{Avg.} $\bm{M} = 40$}
        &\multicolumn{3}{c}{\textbf{Avg.} $\bm{M} = 45$}
        &\multicolumn{3}{c}{\textbf{Avg.} $\bm{M} = 50$} \\

        \cmidrule(lr){3-5} \cmidrule(lr){6-8} \cmidrule(lr){9-11}

        &\multicolumn{1}{c}{\textbf{Measure}}
        &\multicolumn{1}{c}{$\bm{\varepsilon} = 10^{-2}$}
        &\multicolumn{1}{c}{$\bm{\varepsilon} = 10^{-3}$} 
        &\multicolumn{1}{c}{$\bm{\varepsilon} = 10^{-4}$} 
        &\multicolumn{1}{c}{$\bm{\varepsilon} = 10^{-2}$}
        &\multicolumn{1}{c}{$\bm{\varepsilon} = 10^{-3}$} 
        &\multicolumn{1}{c}{$\bm{\varepsilon} = 10^{-4}$}
        &\multicolumn{1}{c}{$\bm{\varepsilon} = 10^{-2}$}
        &\multicolumn{1}{c}{$\bm{\varepsilon} = 10^{-3}$} 
        &\multicolumn{1}{c}{$\bm{\varepsilon} = 10^{-4}$}\\

        \cmidrule{2-11}
    
        \multirow{4}{*}[0.1cm]{\rotatebox[origin=c]{90}{\textbf{ST-RB}}}
        &\textbf{E} / $\bm{\varepsilon}$ &$5.57$ &$17.02$ &$2.81$ &$8.22$ &$22.96$ &$3.92$ &$9.03$ &$20.80$ &$5.16$ \\

        &\textbf{RF} / $10^4$ &$152.88$ &$61.15$ &$43.68$ &$230.40$ &$153.60$ &$76.80$ &$330.48$ &$220.32$ &$110.16$ \\

        &\textbf{SU}-\textbf{WT} / $10^3$ &$16.51$ &$16.35$ &$16.20$ &$21.64$ &$21.63$ &$21.61$ &$31.04$ &$30.77$ &$30.75$ \\

        &\textbf{SU}-\textbf{MEM} / $10^3$ &$7.49$ &$7.49$ &$7.47$ &$9.00$ &$9.00$ &$8.99$ &$10.11$ &$10.11$ &$10.09$ \\

        \cmidrule{2-11}
    
        \multirow{4}{*}[0.1cm]{\rotatebox[origin=c]{90}{\textbf{TT-RB}}}
        &\textbf{E} / $\bm{\varepsilon}$ &$89.90$ &$61.55$ &$75.17$ &$90.01$ &$59.94$ &$16.76$ &$89.13$ &$68.19$ &$22.50$ \\

        &\textbf{RF} / $10^4$ &$305.76$ &$152.88$ &$101.92$ &$307.20$ &$230.40$ &$184.32$ &$660.96$ &$330.48$ &$264.38$ \\

        &\textbf{SU}-\textbf{WT} / $10^3$ &$16.32$ &$16.13$ &$15.94$ &$21.81$ &$21.78$ &$21.77$ &$33.48$ &$30.93$ &$30.68$ \\

        &\textbf{SU}-\textbf{MEM} / $10^3$ &$7.51$ &$7.48$ &$7.46$ &$8.92$ &$8.92$ &$8.92$ &$10.93$ &$10.92$ &$10.92$ \\

        \bottomrule 
    \end{tabular}
    \caption{Online results for the transient linear elasticity problem on a 3-$d$ domain. Metrics include: average accuracy (E), normalized with respect to $\varepsilon$; reduction factor (RF), expressed in hundreds of thousands; and average computational speedup in terms of wall time (SU-WT) and memory usage (SU-MEM), both relative to the \ac{hf} simulations.}
    \label{tb: online results elasticity}
\end{table} 
From the results shown in Tbs.~\ref{tb: rb offline elasticity}-\ref{tb: online results elasticity}, we can draw similar conclusions to those observed in the previous benchmarks. \ac{ttrb} performs a significantly cheaper subspace construction than \ac{strb} and exhibits larger constants in its convergence with respect to $\varepsilon$, due to increased hyper-reduction errors. We observe that, compared to the test case in Subsection \ref{subs: heat equation results}, the reduction factors achieved by \ac{ttrb} are no longer significantly greater than that of \ac{strb}, which explains why both methods achieve similar online speedups. This similarity in reduction factors stems from the low dimensionality of the temporal subspace in \ac{strb} -- a characteristic specific to this test case and not generally representative.
\section{Conclusions and future work}
\label{sec: conclusions}
In this work, we introduce \ac{ttrb}, a projection-based \ac{rom} that leverages the \ac{tt} decomposition of \ac{hf} snapshots as an alternative to conventional \ac{tpod}-based \ac{rb} methods. Through various numerical experiments, we demonstrate that \ac{ttrb} achieves significant offline speedup compared to its \ac{strb} counterpart. This efficiency stems from the ``split-axes'' representation, which organizes the \acp{dof} associated with \ac{fe} functions into tensors. This structure facilitates the use of tensor rank-reduction techniques, such as the \ac{tt} decomposition, which we show to be computationally more efficient than traditional \ac{tpod}-based approaches. The key innovation of the \ac{ttrb} algorithm lies in its reliance on operations whose complexity scales with the dimension of a single full-order axis, rather than the overall size of the \ac{fom}. Furthermore, the proposed method achieves superior compression in terms of reduction factor compared to traditional algorithms, making it particularly advantageous for approximating \acp{pde} in very high dimensions. From an accuracy perspective, we demonstrate that \ac{ttrb} provides approximation capabilities comparable to those of \ac{strb}, as evidenced by our numerical results.

A promising direction for future work is extending our method to parameterized problems defined on non-Cartesian grids. To this end, we plan to explore unfitted element discretizations that enable the use of the ``split-axes'' representation of the snapshots. We anticipate no significant theoretical challenges, as the \ac{ttrb} algorithm is largely independent of the discretization or geometry. Additionally, we aim to tackle more complex and practically relevant problems, such as saddle point problems and nonlinear applications.

\printbibliography

@article{MUELLER2024115767,
title = {Model order reduction with novel discrete empirical interpolation methods in space–time},
journal = {Journal of Computational and Applied Mathematics},
volume = {444},
pages = {115767},
year = {2024},
issn = {0377-0427},
doi = {https://doi.org/10.1016/j.cam.2024.115767},
url = {https://www.sciencedirect.com/science/article/pii/S0377042724000165},
author = {Nicholas Mueller and Santiago Badia}
}

@article{choi2021space, 
	author = {Y. Choi and P. Brown and W. Arrighi and R. Anderson and K. Huynh},
	title = {Space–time reduced order model for large-scale linear dynamical systems with application to {B}oltzmann transport problems},
	journal = {Journal of Computational Physics},
	volume = {424},
	number = {},
	pages = {},
	year = {2021},
}

@article{BADIA2018533,
title = {The aggregated unfitted finite element method for elliptic problems},
journal = {Computer Methods in Applied Mechanics and Engineering},
volume = {336},
pages = {533-553},
year = {2018},
issn = {0045-7825},
doi = {https://doi.org/10.1016/j.cma.2018.03.022},
url = {https://www.sciencedirect.com/science/article/pii/S0045782518301476},
author = {Santiago Badia and Francesc Verdugo and Alberto F. Martín}
}

@article{MAMONOV2022115122,
title = {Interpolatory tensorial reduced order models for parametric dynamical systems},
journal = {Computer Methods in Applied Mechanics and Engineering},
volume = {397},
pages = {115122},
year = {2022},
issn = {0045-7825},
doi = {https://doi.org/10.1016/j.cma.2022.115122},
url = {https://www.sciencedirect.com/science/article/pii/S0045782522003061},
author = {Alexander V. Mamonov and Maxim A. Olshanskii}
}

@misc{mamonov2024priorianalysistensorrom,
      title={A priori analysis of a tensor ROM for parameter dependent parabolic problems}, 
      author={Alexander V. Mamonov and Maxim A. Olshanskii},
      year={2024},
      eprint={2311.07883},
      archivePrefix={arXiv},
      primaryClass={math.NA},
      url={https://arxiv.org/abs/2311.07883}, 
}

@misc{muellerbadiagridaproms,
      title={GridapROMs.jl: Efficient reduced order modelling in the Julia programming language}, 
      author={Nicholas Mueller and Santiago Badia},
      year={2025},
      eprint={2503.15994},
      archivePrefix={arXiv},
      primaryClass={math.NA},
      url={https://arxiv.org/abs/2503.15994},
}

@incollection{PHARR2017401,
title = {07 - Sampling and Reconstruction},
editor = {Matt Pharr and Wenzel Jakob and Greg Humphreys},
booktitle = {Physically Based Rendering (Third Edition)},
publisher = {Morgan Kaufmann},
edition = {Third Edition},
address = {Boston},
pages = {401-504},
year = {2017},
isbn = {978-0-12-800645-0},
doi = {https://doi.org/10.1016/B978-0-12-800645-0.50007-5},
url = {https://www.sciencedirect.com/science/article/pii/B9780128006450500075},
author = {Matt Pharr and Wenzel Jakob and Greg Humphreys}
}

@Inbook{LeDret2016,
  author="Le Dret, Herv{\'e}
  and Lucquin, Brigitte",
  title="The Finite Element Method in Dimension Two",
  bookTitle="Partial Differential Equations: Modeling, Analysis and Numerical Approximation",
  year="2016",
  publisher="Springer International Publishing",
  address="Cham",
  pages="167--218",
  isbn="978-3-319-27067-8",
  doi="10.1007/978-3-319-27067-8_6",
  url="https://doi.org/10.1007/978-3-319-27067-8_6"
}

@misc{halko2010findingstructurerandomnessprobabilistic,
  title={Finding structure with randomness: Probabilistic algorithms for constructing approximate matrix decompositions}, 
  author={Nathan Halko and Per-Gunnar Martinsson and Joel A. Tropp},
  year={2010},
  eprint={0909.4061},
  archivePrefix={arXiv},
  primaryClass={math.NA},
  url={https://arxiv.org/abs/0909.4061}, 
}

@misc{dektor2024collocationmethodsnonlineardifferential,
      title={Collocation methods for nonlinear differential equations on low-rank manifolds}, 
      author={Alec Dektor},
      year={2024},
      eprint={2402.18721},
      archivePrefix={arXiv},
      primaryClass={math.NA},
      url={https://arxiv.org/abs/2402.18721}, 
}

@article{dalsanto2019algebraic,
	title={An algebraic least squares reduced basis method for the solution of nonaffinely parametrized {S}tokes equations},
	author={Dal Santo, Niccol{\`o} and Deparis, Simone and Manzoni, Andrea and Quarteroni, Alfio},
	journal={Computer Methods in Applied Mechanics and Engineering},
	volume={344},
	pages={186--208},
	year={2019},
	publisher={Elsevier}
}

@article{ballarin2015supremizer,
	title={Supremizer stabilization of {POD}--{G}alerkin approximation of parametrized steady incompressible {N}avier--{S}tokes equations},
	author={Ballarin, Francesco and Manzoni, Andrea and Quarteroni, Alfio and Rozza, Gianluigi},
	journal={International Journal for Numerical Methods in Engineering},
	volume={102},
	number={5},
	pages={1136--1161},
	year={2015},
	publisher={Wiley Online Library}
}

@article{carlberg2017galerkin,
  title={{G}alerkin v. least--squares {P}etrov--{G}alerkin projection in nonlinear model reduction},
  author={Carlberg, Kevin and Barone, Matthew and Antil, Harbir},
  journal={Journal of Computational Physics},
  volume={330},
  pages={693--734},
  year={2017},
  publisher={Elsevier}
}

@article{prudhomme:hal-00798326,
  TITLE = {{Reliable Real-Time Solution of Parametrized Partial Differential Equations: Reduced-Basis Output Bound Methods}},
  AUTHOR = {Prud'Homme, Christophe and Rovas, Dimitrios V. and Veroy, Karen and Machiels, Luc and Maday, Yvon and Patera, Anthony T. and Turinici, Gabriel},
  URL = {https://hal.science/hal-00798326},
  JOURNAL = {{Journal of Fluids Engineering}},
  PUBLISHER = {{American Society of Mechanical Engineers}},
  VOLUME = {124},
  NUMBER = {1},
  PAGES = {70-80},
  YEAR = {2001},
  MONTH = Nov,
  DOI = {10.1115/1.1448332},
  HAL_ID = {hal-00798326},
  HAL_VERSION = {v1},
}

@article{Prudhomme,
author = {Prud'homme, Christophe and Rovas, Dimitrios and Veroy, Karen and Patera, Anthony},
year = {2002},
month = {09},
pages = {747 - 771},
title = {A Mathematical and Computational Framework for Reliable Real-Time Solution of Parametrized Partial Differential Equations},
volume = {36},
journal = {ESAIM: Mathematical Modelling and Numerical Analysis},
doi = {10.1051/m2an:2002035}
}

@article{negri2015reduced,
	title={Reduced basis approximation of parametrized optimal flow control problems for the {S}tokes equations},
	author={Negri, Federico and Manzoni, Andrea and Rozza, Gianluigi},
	journal={Computers \& Mathematics with Applications},
	volume={69},
	number={4},
	pages={319--336},
	year={2015},
	publisher={Elsevier}
}

@article{rozza2013reduced,
	title={Reduced basis approximation and a posteriori error estimation for {S}tokes flows in parametrized geometries: roles of the inf--sup stability constants},
	author={Rozza, Gianluigi and Huynh, DB and Manzoni, Andrea},
	journal={Numerische Mathematik},
	volume={125},
	number={1},
	pages={115--152},
	year={2013},
	publisher={Springer}
}

@book{quarteroni2015reduced,
	title={Reduced basis methods for partial differential equations: an introduction},
	author={Quarteroni, Alfio and Manzoni, Andrea and Negri, Federico},
	volume={92},
	year={2015},
	publisher={Springer}
}

@article{chaturantabut2010nonlinear,
	title={Nonlinear model reduction via discrete empirical interpolation},
	author={Chaturantabut, Saifon and Sorensen, Danny C},
	journal={SIAM Journal on Scientific Computing},
	volume={32},
	number={5},
	pages={2737--2764},
	year={2010},
	publisher={SIAM}
}

@article{NEGRI2015431,
title = {Efficient model reduction of parametrized systems by matrix discrete empirical interpolation},
journal = {Journal of Computational Physics},
volume = {303},
pages = {431-454},
year = {2015},
issn = {0021-9991},
doi = {https://doi.org/10.1016/j.jcp.2015.09.046},
url = {https://www.sciencedirect.com/science/article/pii/S0021999115006543},
author = {Federico Negri and Andrea Manzoni and David Amsallem}
}

@article{doi:10.1137/22M1509114,
author = {Tenderini, Riccardo and Mueller, Nicholas and Deparis, Simone},
title = {Space-Time Reduced Basis Methods for Parametrized Unsteady Stokes Equations},
journal = {SIAM Journal on Scientific Computing},
volume = {46},
number = {1},
pages = {B1-B32},
year = {2024},
doi = {10.1137/22M1509114},
URL = {https://doi.org/10.1137/22M1509114},
eprint = {https://doi.org/10.1137/22M1509114}
}

@article{Unger_2019,
	doi = {10.1007/s10444-019-09701-0},
	url = {https://doi.org/10.10072Fs10444-019-09701-0},
  	year = 2019,
	publisher = {Springer Science and Business Media {LLC}},
	volume = {45},
	number = {5-6},
	pages = {2273--2286},
	author = {Benjamin Unger and Serkan Gugercin},
	title = {Kolmogorov n-widths for linear dynamical systems},
	journal = {Advances in Computational Mathematics}
}

@article{cui2021deep,
  title={Deep composition of Tensor-Trains using squared inverse Rosenblatt transports},
  author={Cui, Tiangang and Dolgov, Sergey},
  journal={Foundations of Computational Mathematics},
  pages={1--60},
  year={2021},
  publisher={Springer}
}

@article{gorodetsky2019continuous,
  title={A continuous analogue of the tensor-train decomposition},
  author={Gorodetsky, Alex and Karaman, Sertac and Marzouk, Youssef M},
  journal={Computer Methods in Applied Mechanics and Engineering},
  volume={347},
  pages={59--84},
  year={2019},
  publisher={Elsevier}
}

@article{dolgov2018approximation,
  title={Approximation and sampling of multivariate probability distributions in the tensor train decomposition},
  author={Dolgov, Sergey and Anaya-Izquierdo, Karim and Fox, Colin and Scheichl, Robert},
  journal={arXiv preprint arXiv:1810.01212},
  year={2018}
}

@article{bigoni2016spectral,
	title={Spectral tensor-train decomposition},
	author={Bigoni, Daniele and Engsig-Karup, Allan P and Marzouk, Youssef M},
	journal={SIAM Journal on Scientific Computing},
	volume={38},
	number={4},
	pages={A2405--A2439},
	year={2016},
	publisher={SIAM}
}

@article{Hansbo,
author = {Hansbo, Peter},
year = {2005},
month = {11},
pages = {},
title = {Nitsche’s method for interface problems in computational mechanics},
volume = {28},
journal = {GAMM-Mitteilungen},
doi = {10.1002/gamm.201490018}
}

@article{BurmanHansbo,
author = {Burman, Erik and Hansbo, Peter},
year = {2014},
month = {05},
pages = {},
title = {Fictitious domain methods using cut elements: III. A stabilized Nitsche method for Stokes’ problem},
volume = {48},
journal = {European Series in Applied and Industrial Mathematics (ESAIM): Mathematical Modelling and Numerical Analysis},
doi = {10.1051/m2an/2013123}
}

@article{hackbusch2014numerical,
  title={Numerical tensor calculus},
  author={Hackbusch, Wolfgang},
  journal={Acta numerica},
  volume={23},
  pages={651--742},
  year={2014},
  publisher={Cambridge University Press}
}

@article{dolgov2014alternating,
  title={Alternating minimal energy methods for linear systems in higher dimensions},
  author={Dolgov, Sergey V and Savostyanov, Dmitry V},
  journal={SIAM Journal on Scientific Computing},
  volume={36},
  number={5},
  pages={A2248--A2271},
  year={2014},
  publisher={SIAM}
}

@book{hackbusch2012tensor,
  title={Tensor spaces and numerical tensor calculus},
  author={Hackbusch, Wolfgang},
  volume={42},
  year={2012},
  publisher={Springer Science \& Business Media}
}

@article{oseledets2011tensor,
  title={Tensor-train decomposition},
  author={Oseledets, Ivan V},
  journal={SIAM Journal on Scientific Computing},
  volume={33},
  number={5},
  pages={2295--2317},
  year={2011},
  publisher={SIAM}
}

@article{oseledets2010tt,
  title={TT-cross approximation for multidimensional arrays},
  author={Oseledets, Ivan and Tyrtyshnikov, Eugene},
  journal={Linear Algebra and its Applications},
  volume={432},
  number={1},
  pages={70--88},
  year={2010},
  publisher={Elsevier}
}

@article{tucker1966some,
  title={Some mathematical notes on three-mode factor analysis},
  author={Tucker, Ledyard R},
  journal={Psychometrika},
  volume={31},
  number={3},
  pages={279--311},
  year={1966},
  publisher={Springer}
}

@article{Holtz2011,
author = {Holtz, Sebastian and Rohwedder, Thorsten and Schneider, Reinhold},
year = {2011},
month = {01},
pages = {},
title = {The Alternating Linear Scheme for Tensor Optimization in the Tensor Train Format},
journal = {SIAM Journal on Scientific Computing},
doi = {10.1137/100818893}
}

@article{doi:10.1137/110833142,
author = {Oseledets, I. V. and Dolgov, S. V.},
title = {Solution of Linear Systems and Matrix Inversion in the TT-Format},
journal = {SIAM Journal on Scientific Computing},
volume = {34},
number = {5},
pages = {A2718-A2739},
year = {2012},
doi = {10.1137/110833142},
URL = {https://doi.org/10.1137/110833142},
eprint = {https://doi.org/10.1137/110833142}
}

@software{ken_ho_2022_6394438,
  author       = {Ken Ho and
                  Sheehan Olver and
                  Tony Kelman and
                  Elias Jarlebring and
                  Julia TagBot and
                  BambOoxX and
                  Mikael Slevinsky},
  title        = {JuliaMatrices/LowRankApprox.jl: v0.5.2},
  month        = mar,
  year         = 2022,
  publisher    = {Zenodo},
  version      = {v0.5.2},
  doi          = {10.5281/zenodo.6394438},
  url          = {https://doi.org/10.5281/zenodo.6394438}
}

@inbook{doi:https://doi.org/10.1002/0471249718.ch1,
publisher = {John Wiley \& Sons, Ltd},
isbn = {9780471249719},
title = {Gaussian Elimination and Its Variants},
booktitle = {Fundamentals of Matrix Computations},
chapter = {1},
pages = {1-110},
doi = {https://doi.org/10.1002/0471249718.ch1},
url = {https://onlinelibrary.wiley.com/doi/abs/10.1002/0471249718.ch1},
eprint = {https://onlinelibrary.wiley.com/doi/pdf/10.1002/0471249718.ch1},
year = {2002}
}

@article{goreinov1997pseudo,
  title={Pseudo-skeleton approximations by matrices of maximal volume},
  author={Goreinov, Sergei A and Zamarashkin, Nikolai Leonidovich and Tyrtyshnikov, Eugene E},
  journal={Mathematical Notes},
  volume={62},
  number={4},
  pages={515--519},
  year={1997},
  publisher={Springer}
}
\end{document}